\documentclass[11pt,oneside,british,a4wide]{amsart}
\usepackage{amsmath, amssymb,verbatim,appendix}
\usepackage{hyperref}
\usepackage[mathscr]{eucal}
\usepackage{amscd}
\usepackage{amsthm}
\usepackage{stmaryrd}
\usepackage{enumerate}
\usepackage{comment}
\usepackage[latin1]{inputenc} 
\usepackage{tikz}
\usetikzlibrary{shapes,arrows}
\usepackage{cite}
\usepackage{url}

\makeatletter
\newcommand{\dotminus}{\mathbin{\text{\@dotminus}}}

\newcommand{\@dotminus}{%
  \ooalign{\hidewidth\raise1ex\hbox{.}\hidewidth\cr$\m@th-$\cr}%
}
\makeatother

\usepackage{setspace,a4wide,color,xcolor,graphicx}

\def\showauthornotes{1}

\ifnum\showauthornotes=1
\newcommand{\Authornote}[2]{{\sf\small\color{red}{[#1: #2]}}}
\else
\newcommand{\Authornote}[2]{}
\fi

\newtheorem{theorem}{Theorem}[section]
\newtheorem{lemma}[theorem]{Lemma}
\newtheorem{corollary}[theorem]{Corollary}
\newtheorem{proposition}[theorem]{Proposition}

\newtheorem{notation}[theorem]{Notation}
\newtheorem{question}[theorem]{Question}

\newtheorem{observation}[theorem]{Observation}

\newtheorem{claim}[theorem]{Claim}
\newtheorem{fact}[theorem]{Fact}

\theoremstyle{definition}

\newtheorem{definition}[theorem]{Definition}
\newtheorem{convention}[theorem]{Convention}

\def\e{\epsilon}

\newcommand{\VC}{\textnormal{VC}}

\newcommand{\SVC}{\textnormal{SVC}}
\newcommand{\PS}{\textnormal{PS}}

\newcommand{\vdisc}{\textnormal{vdisc}}

\newcommand{\calL}{\mathcal{L}}
\newcommand{\calF}{\mathcal{F}}
\newcommand{\calH}{\mathcal{H}}

\newcommand{\calP}{\mathcal{P}}

\newcommand{\calE}{\mathcal{E}}
\newcommand{\calG}{\mathcal{G}}

\newcommand{\calQ}{\mathcal{Q}}
\newcommand{\calR}{\mathcal{R}}
\newcommand{\calS}{\mathcal{S}}

\newcommand{\calU}{\mathcal{U}}

\newcommand{\calX}{\mathcal{X}}

\DeclareMathOperator{\Tw}{\textrm{Tw}}
\DeclareMathOperator{\reg}{\textrm{reg}}
\DeclareMathOperator{\Irr}{\textnormal{AP}}

\DeclareMathOperator{\sm}{\textrm{sm}}
\DeclareMathOperator{\bg}{\textrm{big}}

\def\Forb{\operatorname{Forb}}

\begin{document}
\title[]{On the growth of weak regular partitions in 3-uniform hypergraphs}
\author{C. Terry}\thanks{The author was partially supported by NSF CAREER Award DMS-2115518 and a Sloan Research Fellowship}
\address{Department of Mathematics, Statistics and Computer Science, University of Illinois at Chicago, Chicago, IL, USA}
\email{caterry@uic.edu}

\begin{abstract}
This paper studies the asymptotic growth of regular partitions in hereditary properties of $3$-uniform hypergraphs.  The focus is on the notion of weak hypergraph regularity, first developed by Chung and Haviland--Thomason.  Given a hereditary property of $3$-uniform hypergraphs $\mathcal{H}$, we define a function $M_{\mathcal{H}}\colon(0,1)\rightarrow \mathbb{N}$ by letting $M_{\mathcal{H}}(\epsilon)$ be the smallest integer $M$ such that all sufficiently large elements of $\mathcal{H}$ admit a weak regular partition of size at most $M$.   We show that the asymptotic growth rate of such a function falls into one of four categories: constant, polynomial,  between a single and a double exponential, or tower.   This is the second in a series of papers about the growth of different kinds of regular partitions in $3$-uniform hypergraphs, and the results of this paper are a crucial component in part 3 of the series, which considers vertex partitions associated to a stronger notion of hypergraph regularity. 
\end{abstract}
\maketitle
\tableofcontents{}

\section{Introduction}

This is the second in a series of papers about the growth of regular partitions in hereditary properties of graphs and hypergraphs.  A \emph{hereditary graph property} $\calH$ is a class of finite graphs closed under induced subgraphs and isomorphisms.  There is a rich literature studying structural dichotomies among hereditary graph properties and their generalizations.  Examples include dichotomies related to speeds \cite{BaBoMo, BaBoMo1, BaBoMo2, BBW2, BBW1,Alekseev2,Alekseev1,DN,BoTh2,BoTh1, BBSS, Terry.2018, BDEP, TL}, bounds for the removal lemma \cite{GishbolinerShapira},  sizes of cliques and anticliques \cite{erdos1, erdos2, NSP, EHsurvey, Fox.2017bfo, Malliaris.2014}, and regular partitions \cite{Lovasz.2010, Terry.2021a, Terry.2021b, Fox.2017bfo, Alon.2018is}.  Many of these results have substantive connections to model theory (see for example, \cite{Malliaris.2014, LT2, Braunfeld, domination,  BraunfeldLaskowski1}).  This paper  and its companions \cite{Terry.2024a,Terry.2024c,Terry.2024d}  explore dichotomies for bounds in Szemer\'{e}di's Regularity Lemma and its extensions to $3$-uniform hypergraphs.

Szemer\'{e}di's Regularity Lemma provides structural decompositions for finite graphs. Informally speaking, the regularity lemma says that any large finite graph can be partitioned into a bounded number of pieces, so that most pairs of pieces behave quasirandomly.  We state here a version of the regularity lemma and refer the reader to Section \ref{ss:reghom} for precise definitions.

\begin{theorem}[Szemer\'{e}di \cite{Szemeredi.1978}]\label{thm:sz}
For all $\e>0$ there exists $M=M(\e)$ so that any finite graph has an $\e$-regular partition with at most $M$ parts.
\end{theorem}

Proofs of Theorem \ref{thm:sz} yield a tower-type dependence between $\e$ and $M(\e)$, where the \emph{tower function}, $\Tw\colon \mathbb{Z}^{\geq 1}\rightarrow \mathbb{Z}^{\geq 1}$, is defined by setting $\Tw(1)=1$, and for $i>1$ setting  $\Tw(i)=2^{\Tw(i-1)}$.\footnote{We extend this to a function $\Tw\colon\mathbb{R}^{\geq 1}\rightarrow\mathbb{Z}^{\geq 1}$ by setting $\Tw(x)=\Tw(\lceil x\rceil)$.}  This kind of dependence was shown to be necessary by Gowers, who proved $M(\e)\geq \Tw(\e^{-1/16})$ \cite{Gowers.1997}.  Essentially tight upper and lower bounds are now known. Both are due to Fox and Lov\'{a}sz, although they are formulated for a slightly different version of regularity than the one used in this paper.  In our context, their results yield $ \Tw(\Omega(\e^{-2}))\leq M(\e)\leq \Tw(O(\e^{-4}))$ \cite{Fox.2014}.  

On the other hand, it is known that under certain hypotheses, the dependence of $M(\e)$ on $\e$ is much better.  In particular, work of Alon--Fischer--Newman and Lov\'{a}sz--Szegedy showed a polynomial type dependence suffices for graphs of uniformly bounded VC-dimension \cite{Alon.2000,Lovasz.2010}. The partitions constructed in these papers are not only regular, but also \emph{homogeneous}, meaning most pairs of parts have edge density near $0$ or $1$.

The results discussed above produce a dichotomy for hereditary graph properties, as was first observed by Alon--Fox--Zhao \cite{Alon.2018is}.  This dichotomy is stated in terms of a function which measures the growth of the bound $M(\e)$ in Theorem \ref{thm:sz}, when restricted to only the graphs in a fixed hereditary property $\calH$.

\begin{definition}\label{def:M}
Suppose $\calH$ is a hereditary graph property.  Define $M_{\calH}\colon(0,1)\rightarrow \mathbb{N}$ by letting $M_{\calH}(\e)$ be the smallest integer $M$ so that any sufficiently large graph in $\calH$ has an $\e$-regular partition with at most $M$ parts.
\end{definition}

 In analogy to the question posed in \cite{ZitoSch} about speeds of hereditary graph properties, one can pose the following question about $M_{\calH}(\e)$. 

\begin{question}\label{q:M}
What are the possible asymptotic behaviors of the function $M_{\calH}\colon(0,1)\rightarrow \mathbb{N}$ when $\calH$ ranges over all hereditary graph properties?
\end{question}

Here, by ``asymptotic," we mean as $\e\rightarrow 0$.  Alon--Fox--Zhao observed in \cite{Alon.2018is} that such a function cannot exhibit arbitrary behavior.  Specifically, the results of Alon--Fischer--Newman and Lov\'{a}sz--Szegedy  \cite{Alon.2000,Lovasz.2010}, in conjunction with the work of Fox--Lov\'{a}sz  \cite{Fox.2014} imply a dichotomy, or ``jump," in the possible behavior of $M_{\calH}(\e)$, based on whether the VC-dimension of $\calH$ is finite or infinite (for precise definitions see Section \ref{ss:vc}).

\begin{theorem}\label{thm:vcjump}
Suppose $\calH$ is a hereditary graph property.  Then one of the following holds. 
\begin{enumerate}
\item $\VC(\calH)=\infty$.  In this case $M_{\calH}(\e)\geq \Tw(\Omega(\e^{-2}))$.
\item $\VC(\calH)<\infty$. In this case, there is a constant $C>0$ so that $M_{\calH}(\e)\leq  \e^{-C}$.
\end{enumerate}
\end{theorem}

We will show in this paper that there is one additional jump, yielding the following coarse answer to Question \ref{q:M}.

\begin{theorem}\label{thm:alljump}
Suppose $\calH$ is a hereditary graph property.  Then one of the following holds. 
\begin{enumerate}
\item (Tower) For some constants $C,C'>0$, $\Tw(\e^{-C})\leq M_{\calH}(\e)\leq \Tw(\e^{-C'})$.
\item (Polynomial) For some constants $C,C'>0$, $\e^{-C}\leq M_{\calH}(\e)\leq \e^{-C'}$.
\item (Constant) There is a constant $C\geq 1$ so that $M_{\calH}(\e)=C$.
\end{enumerate}
\end{theorem}
 
 Theorem \ref{thm:alljump} is closely related to the analogue of Question \ref{q:M} for  \emph{homogeneous} partitions in graphs.  In particular, the work of \cite{Alon.2000,Lovasz.2010} on graphs of bounded VC-dimension naturally gives rise to the following analogue of Definition \ref{def:M} (see Section \ref{ss:vc} for more details).

\begin{definition}\label{def:Mhom2}
Suppose $\calH$ is a hereditary graph property.  Define $M^{\hom}_{\calH}\colon (0,1)\rightarrow \mathbb{N}\cup \{\infty\}$ by letting $M^{\hom}_{\calH}(\e)$ be the smallest integer $M$ (or $\infty$ if no such integer exists) such that any sufficiently large graph in $\calH$ has an $\e$-homogeneous partition with at most $M$ parts.
\end{definition}

One can deduce from \cite{Alon.2000,Lovasz.2010} that $M_{\calH}^{\hom}<\infty$ if and only if $\calH$ has finite VC-dimension.  An important aspect of our proof of Theorem \ref{thm:alljump} is showing that when $M_{\calH}^{\hom}<\infty$,  $M_{\calH}$ grows at roughly the same rate as $M_{\calH}^{\hom}$.  Specifically, we prove the following stronger version of Theorem \ref{thm:alljump}, which reflects the role played by $M_{\calH}^{\hom}$. 

\begin{theorem}\label{thm:alljumphom}
Suppose $\calH$ is a hereditary graph property.  Then one of the following holds. 
\begin{enumerate}
\item (Tower) $\calH$ has infinite VC-dimension and 
$$
\Tw(\Omega(\e^{-2}))\leq M_{\calH}(\e)\leq \Tw(O(\e^{-4})).
$$
\item (Polynomial) $\calH$ has VC-dimension $k<\infty$, and for some constant $c=c(k)$,
$$
\e^{-1+o(1)}\leq M_{\calH}(\e)\leq M_{\calH}^{\hom}(\e^3)\leq c\e^{-6k-3}.
$$
\item (Constant) There is a constant $C\geq 1$ so that $M_{\calH}(\e)=M^{\hom}_{\calH}(\e)=C$.
\end{enumerate}
\end{theorem}

The proof of Theorem \ref{thm:alljumphom} uses results from \cite{Fox.2014} in the upper and lower bounds of range (1), and a quantitative improvement on \cite{Alon.2000,Lovasz.2010} due to Fox--Pach--Suk \cite{Fox.2017bfo} in the upper bound in range (2). 

The main goal of this paper  is an analogue of Theorem \ref{thm:alljump} for $3$-uniform hypergraphs.  There are several inequivalent versions of hypergraph regularity.  The first such notion, called \emph{weak regularity}, was developed by Chung \cite{Chung.1991} and Haviland--Thomason \cite{Haviland.1989}.  While this notion is the most natural extension of graph regularity to the hypergraph setting, it has limitations in terms of applications (see \cite{Kohayakawa.2010}). For this reason, more complicated versions of hypergraph regularity were later developed by Gowers \cite{Gowers.20063gk, Gowers.2007} on the one hand, and Frankl,  Kohayakawa, Nagle,  R\"{o}dl, Skokan, and Schacht  \cite{Frankl.2002, Gowers.20063gk, Gowers.2007, Nagle.2006, Rodl.2005, Rodl.2005a} on the other.  For more details on the history of hypergraph regularity, we refer the reader to the introduction of \cite{Gowers.2007}, as well as \cite{Nagle.2013}. 

The topic of this paper is the growth of \emph{weak} regular partitions, while parts 3 and 4 of the series \cite{Terry.2024c, Terry.2024d} deal with the stronger regularity for $3$-uniform hypergraphs first developed by Gowers \cite{Gowers.20063gk, Gowers.2007}.  Perhaps surprisingly,  \cite{Terry.2024c} relies on connections between weak and strong regular partitions, and the results of this paper play a crucial role there.  On the other hand, the rest of this paper focuses exclusively on the ``weak" type of regularity. For this reason, we will often omit the word ``weak" in our discussions,  referring simply to regularity for $3$-uniform hypergraphs.  

Weak regularity for $3$-uniform hypergraphs is defined in close analogy to graph regularity.  Informally speaking, given a $3$-uniform hypergraph $H=(V,E)$, a triple of subsets $(X,Y,Z)$ is called \emph{$\e$-regular} if for all large subsets $X'\subseteq X$, $Y'\subseteq Y$, and $Z'\subseteq Z$, the density of edges between $X,Y,Z$ is within $\e$ of the density of edges between $X',Y',Z'$. 
An \emph{$\e$-regular partition} of $H$ is then a partition $\calP$ of $V$ so that at least $(1-\e)|V|^3$ many triples from $V^3$ lie in an $\e$-regular triple of parts from $\calP$  (see Section \ref{ss:reghom} for precise definitions).  

All $3$-uniform hypergraphs admit $\e$-regular partitions of this kind. We state below the version of this result  proved by Chung in \cite{Chung.1991}.  For a more precise statement of the bound involved, we refer the reader to the discussion following Theorem \ref{thm:chung}.

\begin{theorem}[Chung \cite{Chung.1991}]\label{thm:chungintro}
For all $\e>0$ there exists $T\leq \Tw(6\e^{-4})$ such that the following holds.  If $H=(V,E)$ is a sufficiently large $3$-uniform hypergraph, then there is some $1\leq t\leq T$ and an $\e$-regular partition of  $H$ with $t$ parts. 
\end{theorem}

In light of Theorem \ref{thm:chungintro}, we can define the $3$-uniform analogue of $M_{\calH}$.  

\begin{definition}\label{def:M3}
Given a hereditary property $\calH$ of $3$-uniform hypergraphs,  define the function $M_{\calH}\colon(0,1)\rightarrow \mathbb{N}$ by letting $M_{\calH}(\e)$ be the smallest integer $M$ so that every sufficiently large $3$-uniform hypergraph in $\calH$ has an  $\e$-regular partition with at most $M$ parts.
\end{definition}

Theorem \ref{thm:chungintro} implies that for any  hereditary property of $3$-uniform hypergraphs,  $\calH$, the integer $M_{\calH}(\e)$ is bounded above by a tower of height polynomial in $\e^{-1}$.  The only other previous  results about this function are due to Chernikov and Starchenko, and independently, Fox, Pach, and Suk, who showed that when the property $\calH$ has finite VC-dimension (see Section \ref{sec:weakreg}), then $M_{\calH}(\e)$ is bounded above by a polynomial in $\e^{-1}$ \cite{Fox.2017bfo,Chernikov.2016zb}.

The main result of this paper is  that any function of the form $M_{\calH}$ must fall into four distinct growth classes: constant, polynomial, between single and double exponential, or tower.

\begin{theorem}\label{thm:weak}
Suppose $\calH$ is a hereditary property of $3$-uniform hypergraphs.  Then one of the following holds.
\begin{enumerate}
\item (Tower) For some $C,C'>0$, $\Tw(\e^{-C})\leq M_{\calH}(\e)\leq \Tw(\e^{-C'})$.
\item (Almost Exponential)  For some $C,C'>0$, $2^{\e^{-C}}\leq M_{\calH}(\e)\leq 2^{2^{\e^{-C'}}}$.
\item (Polynomial) For some $C,C'>0$, $\e^{-C}\leq M_{\calH}(\e)\leq \e^{-C'}$.
\item (Constant) For some $C\geq 1$, $M_{\calH}(\e)=C$.
\end{enumerate}
\end{theorem}

In analogy to the graphs case, Theorem \ref{thm:weak} is closely related to homogeneous partitions.  In particular, using  the natural analogue of Definition \ref{def:Mhom2} for $3$-uniform hypergraphs (see Section \ref{sec:weakreg}), we will prove  the following stronger version of Theorem \ref{thm:weak}. 

\begin{theorem}\label{thm:weakhom}
Suppose $\calH$ is a hereditary property of $3$-uniform hypergraphs.  Then one of the following holds.
\begin{enumerate}
\item (Tower) $\Tw(\Omega(\e^{-1}))\leq M_{\calH}(\e)\leq \Tw(6\e^{-4})$.
\item (Almost Exponential) For some $C>0$,
$$
2^{\Omega(\e^{-1/8})}\leq M_{\calH}(\e)\leq M_{\calH}^{\hom}(\e^4)\leq 2^{2^{\e^{-C}}}.
$$
\item (Polynomial) For some $C>0$, $\Omega(\e^{-1/8})\leq M_{\calH}(\e)\leq M_{\calH}^{\hom}(\e^4)\leq \e^{-C}$.
\item (Constant) For some $C\geq 1$, $M_{\calH}(\e)=M_{\calH}^{\hom}(\e)=C$.
\end{enumerate}
\end{theorem}

The proof of Theorem \ref{thm:weakhom} draws on many results from the literature. We will show the growth of $M_{\calH}$ is determined by the structure of an auxiliary property, ${\bf B}_{\calH}$, consisting of those $3$-uniform hypergraphs $H$ with the property that arbitrarily large blowups of $H$ appear in $\calH$ (see Definition \ref{def:bhgraphs}).  We leverage a connection between $\calH$ and ${\bf B}_{\calH}$ which relies on the induced removal lemma for $3$-uniform hypergraphs due to Kohayakawa, R\"{o}dl, and Skokan \cite{Kohayakawa.2002}. Other ingredients  in the proof include a lower bound construction for graphs due to Fox--Lov\'{a}sz \cite{Fox.2014}, results of the author and Wolf \cite{Terry.2021b} characterizing when $M_{\calH}^{\hom}<\infty$,  bounds on the size of regular partitions of $3$-uniform hypergraphs with bounded slicewise VC-dimension (see \cite{Terry.2024a, GSW}), and the efficient regularity lemma for hypergraphs of small VC-dimension due to Chernikov--Starchenko \cite{Chernikov.2016zb} and Fox--Pach--Suk \cite{Fox.2017bfo}.   Our proofs provide explicit characterizations of the properties in each growth class according to forbidden and allowed substructures. An overview of these characterizations is provided in Appendix \ref{ss:charapp}. 

We now discuss some open problems.  The most glaring question regarding Theorem \ref{thm:weak} is the gap between the single and double exponential bounds in range (2). Since the first draft of this paper appeared on arXiv, Gishboliner, Shapira, and Wigderson \cite{GSW} obtained an optimal improvement on the main result of part 1 \cite{Terry.2024a}. Replacing our use of \cite{Terry.2024a} here with \cite{GSW} resolves the gap in range (2) by improving the upper bound to a single exponential.  

There are many other open problems relating to refinements of Theorem  \ref{thm:alljumphom}, especially in (2), the polynomial range.  For instance, can the lower bound in (2) be improved to $\Omega(\e^{-1})$?  What is the optimal exponent in the upper bound of range (2)? Could there be additional jumps within range (2)?   Analogous questions can be asked about Theorem \ref{thm:weakhom}, especially in ranges (2) and (3), the exponential and polynomial ranges.

\subsection{Outline} 
We now give an outline of the rest of the paper.  In Section \ref{sec:pre} we cover notation, basic definitions, and background material needed for the paper. In Section \ref{sec:graphs}, we prove Theorem \ref{thm:alljumphom}.  In Section \ref{sec:weakreg}, we introduce background on two higher arity generalizations of VC-dimension and their connections to homogeneous partitions. In Section \ref{sec:exptower}, we prove the existence of a jump between the double exponential and tower speeds.  In Section \ref{sec:blowuplemma}, we prove a general lower bound lemma for certain blowups of $3$-uniform hypergraphs.  In Section \ref{sec:polyexp}, we prove the existence of a jump between polynomial and exponential speeds.  Finally, in Section \ref{sec:constpoly}, we prove the existence of a jump between constant and polynomial speeds, after which we prove Theorem \ref{thm:weakhom}.  The appendix contains the proofs of several auxiliary results. These fall into four general categories: combinatorial characterizations of growth classes, standard lemmas, results about almost prime graphs and hypergraphs, and results about VC-dimension in hereditary properties of $3$-uniform hypergraphs.

\subsection{Acknowledgements} The author thanks the anonymous referee for their many helpful suggestions, and for identifying several gaps in an earlier version of this paper.  The author also thanks Hannah Sheats for identifying further oversights in an earlier version of this paper, and suggesting a simplification implemented in Lemma \ref{lem:12blowup}.

\section{Preliminaries}\label{sec:pre}

This section contains preliminaries for the rest of the paper.  Subsection \ref{ss:notation} covers basic notation and definitions.  Subsection \ref{ss:hps} contains background on hereditary properties.  Subsection \ref{ss:blowups} defines blowups  and  states an important result relating blowups to closeness.  Subsection \ref{ss:reghom} introduces weak regularity, and Subsection \ref{ss:hom} defines homogeneous partitions and related notions. Finally, Subsection \ref{ss:vcbasic} contains the definition of VC-dimension for set systems.  We note that for the sake of efficiency, many definitions will be stated for $k$-uniform hypergraphs. However, all theorems in the paper deal only with graphs or $3$-uniform hypergraphs.

\subsection{Basic notation and definitions}\label{ss:notation}

Given an integer $n\geq 1$, let $[n]=\{1,\ldots, n\}$.  Throughout the paper, we assume the natural numbers $\mathbb{N}$ start at $0$.  Given $r_1,r_2\in \mathbb{R}$ and $\e>0$, we use the notation $r_1=r_2\pm \e$  to mean that $r_1\in (r_2-\e,r_2+\e)$.  

For a set $V$ and an integer $k\geq 1$, define ${V\choose k}=\{X\subseteq V: |X|=k\}$.  Given a $k$-set $\{x_1,\ldots, x_k\}\in {V\choose k}$, we write $x_1\ldots x_k$ to denote the set $\{x_1,\ldots, x_k\}$. We will mainly use this convention when $k=2$ or $k=3$. In particular, we write $x_1x_2$ to denote  the $2$-element set $\{x_1,x_2\}$ and  $x_1x_2x_3$ to denote the $3$-element set $\{x_1,x_2,x_3\}$.  For each integer $\ell\geq 2$ and sets $X_1,\ldots, X_{\ell}$, we let $K_{\ell}[X_1,\ldots, X_{\ell}]$ denote the set of $\ell$-element sets $\{x_1,\ldots, x_{\ell}\}$ where $x_i\in X_i$ for each $1\leq i\leq \ell$.  We will primarily use this notation when $\ell$ is $2$ or $3$. In these cases we have 
\begin{align*}
K_2[X_1,X_2]&:=\{xy: x\in X_1, y\in X_2, x\neq y\}\text{ and }\\
K_3[X_1,X_2,X_3]&:=\{xyz: x\in X_1, y\in X_2, z\in X_3, x\neq y, y\neq z, x\neq z\}.
\end{align*}

An \emph{equipartition} of a set $V$ is a partition $V=V_1\cup \cdots \cup V_t$ satisfying $||V_i|-|V_j||\leq 1$ for all $1\leq i,j\leq t$.

Given an integer $k\geq 1$, a \emph{$k$-uniform hypergraph} is a pair $(V,E)$ where $V$ is a nonempty set and $E\subseteq {V\choose k}$.  To ease notation, we will refer to $k$-uniform hypergraphs as simply \emph{$k$-graphs}, and $2$-graphs as simply \emph{graphs}. By convention, all $k$-graphs in this paper have finite vertex sets.  

Two $k$-graphs $G=(V,E)$ and $G'=(V',E')$ are \emph{isomorphic}, denoted $G\cong G'$, if there is a bijection $f\colon V\rightarrow V'$ such that $\{v_1,\ldots, v_k\}\in E$ if and only if $\{f(v_1),\ldots, f(v_k)\}\in E'$. 

Given a $k$-graph $G$, $V(G)$ denotes its vertex set and $E(G)$ denotes its edge set.  An \emph{induced sub-$k$-graph of $G$} is a $k$-graph of the form $(V',E')$, for some $V'\subseteq V(G)$ and $E'=E(G)\cap {V'\choose k}$.  We let $G[V']$ denote the induced sub-$k$-graph of $G$ with vertex set $V'$, i.e.
$$
G[V']=\left(V', E(G)\cap {V'\choose k}\right). 
$$
 Given another $k$-graph $G'$, we say $G$ \emph{contains an induced copy of $G'$} if there is some $V'\subseteq V(G)$ so that $G[V']\cong G'$.  If $G$ contains no induced copy of $G'$, we say $G$ \emph{omits $G'$}. Given  $X\subseteq V(G)$, we say $X$ is a \emph{clique in $G$} if ${X\choose k}\subseteq E(G)$ and an \emph{anticlique in $G$} if ${X\choose k}\cap E(G)=\emptyset$ (we will omit the ``in $G$" when it is clear from context).  Similarly, if $E={V\choose k}$ we say $G$ is itself  a clique, and if $E=\emptyset$, we say $G$ is an anticlique.

Suppose $G=(V,E)$ is a $k$-graph.  We let $\overline{E}$ denote the \emph{ordered edge set of $G$}, i.e.
$$
\overline{E}:=\{(x_1,\ldots,x_k)\in V^k: x_1\ldots x_k\in E\}.
$$
 Given nonempty sets $X_1,\ldots, X_k\subseteq V$, the \emph{density of $(X_1,\ldots, X_k)$ in $G$} is  
$$
d_G(X_1,\ldots, X_k):=\frac{|\overline{E}\cap (X_1\times \cdots \times X_k)|}{|X_1|\cdots |X_k|}.
$$ 

Suppose now $G=(V,E)$ is a graph. For any $v\in V$, the \emph{neighborhood of $v$ in $G$} is 
$$
N_G(v)=\{x\in V: xv\in E\}.
$$
More generally, given two sets $A,B$, any subset $\calE\subseteq A\times B$, and any element $a_0\in A$, we write 
  $$
  N_{\calE}(a_0)=\{b\in B: (a_0,b)\in \calE\}.
  $$
 Suppose now $H=(V,E)$ is a $3$-graph. For any $x\neq y\in V$, we let
$$
N_H(x):=\Big\{uv\in {V\setminus\{x\}\choose 2}: xuv\in E\Big\}\text{ and }N_H(xy):=\{v\in V\setminus \{x,y\}: xyv\in E\}.
$$
More generally, given three sets $A,B,C$, any subset $\calE\subseteq A\times B\times C$, and any $a_0\in A$ and $b_0\in B$, we write 
  $$
  N_{\calE}(a_0)=\{(b,c)\in B\times C: (a_0,b,c)\in \calE\}\text{ and }N_{\calE}(a_0,b_0)=\{c\in C: (a_0,b_0,c)\in \calE\}.
  $$
 We will frequently refer back to the following notation.
 \begin{notation}\label{not:01nbrs}
Let $H=(V,E)$ be a $k$-graph. Define $E^1=E$ and $E^0={V\choose k} \setminus E$, and set
$$
H^1=H=(V,E^1)\text{ and }H^0=(V,E^0).
$$
 \end{notation}
  Given integers $\ell\geq k\geq 2$, we say that a $k$-graph $H$ is \emph{$\ell$-partite} if there exists a partition $V(H)=V_1\cup \cdots \cup V_{\ell}$ such that for every $e\in E(H)$ and $1\leq i\leq \ell$, $|e\cap V_i|\leq 1$.  In this case, we will write $H=(V_1\cup \cdots \cup V_{\ell},E)$ to denote that $H$ is $\ell$-partite with compatible partition given by $V(H)=V_1\cup \cdots \cup V_{\ell}$.  We emphasize this notation is meant to indicate the displayed sets $V_1,\ldots, V_{\ell}$ are pairwise disjoint.

We say a graph $G=(V,E)$ is \emph{bipartite} if it is $2$-partite, meaning there is a partition $V=V_1\cup V_2$ such that both $V_1$ and $V_2$ are anticliques.   We say $G$ is \emph{co-bipartite} if there is a partition $V=V_1\cup V_2$ so that   $V_1$ and $V_2$ are cliques. Finally, we say   $G$ is a \emph{split graph} if there is a partition $V=V_1\cup V_2$ so that  $V_1$ is a clique, and $V_2$ is an anticlique.  In all of these cases, we will write $G=(V_1\cup V_2, E)$ to indicate $V_1, V_2$ are the relevant cliques/anticliques.

\subsection{Hereditary properties}\label{ss:hps}

We begin by defining a hereditary $k$-graph property.   

\begin{definition}\label{def:HP}
A \emph{hereditary $k$-graph property} is a nonempty class of finite $k$-graphs closed under induced sub-$k$-graphs and isomorphisms.   
\end{definition}

When $k=2$, we refer to these as simply  \emph{hereditary graph properties}. It will at times be convenient to have notation for the class of all finite $k$-graphs.

\begin{definition}\label{def:allkgraphs}
For each integer $k\geq 1$,  $\calG^{(k)}$ denotes the class of all finite $k$-graphs.
\end{definition}

We next set notation for the class of finite $k$-graphs omitting a fixed collection of induced sub-$k$-graphs. 

\begin{definition}\label{def:forb}
Suppose $\calF$ is a class of finite $k$-graphs. Define
$$
\Forb(\calF)=\{G\in \calG^{(k)}: \text{ $G$ omits all elements of $\calF$ as induced sub-$k$-graphs}\}.
$$
\end{definition} 

It is well-known that every nonempty class of the form $\Forb(\calF)$ is a hereditary $k$-graph property, and conversely, every hereditary $k$-graph property has the form $\Forb(\calF)$ for some class $\calF$ of finite $k$-graphs. 

We next define a notion of ``closeness," first for $k$-graphs, and then for classes of $k$-graphs.

\begin{definition}
Suppose $H=(V,E)$ and $H'=(V,E')$ are two $k$-graphs on the same vertex set and $\delta>0$.  We say $H$ and $H'$ are \emph{$\delta$-close} if $|\overline{E}\Delta \overline{E'}|\leq \delta |V|^k$.
\end{definition}

\begin{definition}\label{def:close}
Suppose $\calH$ and $\calH'$ are classes of finite $k$-graphs.  We say $\calH$ is \emph{close to} $\calH'$ if for all $\delta>0$ there is an integer $N\geq 1$ such that for all $H\in \calH$ on at least $N$ vertices, $H$ is $\delta$-close to some $H'\in \calH'$ on the same vertex set.
\end{definition}

 From the perspective of this paper, a hereditary $k$-graph property is only interesting if it contains arbitrarily large $k$-graphs.  To avoid having to repeatedly make exceptions to deal with the trivial cases where this is not true, we set the following convention for the rest of the paper.

\begin{convention}\label{con:inf}
From here onwards, all hereditary $k$-graph properties are assumed to contain  arbitrarily large $k$-graphs.
\end{convention}

\subsection{Blowups and closeness}\label{ss:blowups}
It will be important for us to have some criteria for determining when one hereditary  property is close to another, in the sense of Definition \ref{def:close}.   In this section we present such criteria in terms of blowups.  We begin by defining blowups of $k$-graphs.  

\begin{definition}\label{def:blowupgraph}
Suppose $k\geq 2$ is an integer and $H=(U, E)$ is a $k$-graph.  
\begin{enumerate}
\item A  \emph{blowup of $H$}  is any $k$-graph $\mathbf{H}$ with vertex set of the form  $V(\mathbf{H})=\bigsqcup_{u\in U}V_u$ and edge set $E(\mathbf{H})$ satisfying
$$
\bigcup_{u_1\ldots u_k\in E}K_k[V_{u_1},\ldots, V_{u_k}]\subseteq E(\mathbf{H})\text{ and }\left(\bigcup_{u_1\ldots u_k \in {U\choose k}\setminus E}K_k[V_{u_1},\ldots, V_{u_k}]\right)\cap E(\mathbf{H})=\emptyset.
$$
\item Given an integer $n\geq 1$, we say $\mathbf{H}$ from (1) is an \emph{$n$-blowup of $H$} if $|V_u|=n$ for all $u\in U$. 
\item Given integers $n_1,n_2\geq 1$ and sets $X,Y\subseteq U$, we say ${\bf H}$ from (1) is an \emph{$(n_1,n_2; X,Y)$-blowup of $H$} if $|V_x|=n_1$ for all $x\in X$ and $|V_y|=n_2$ for all $y\in Y$.
\end{enumerate}
\end{definition}

We note item (3) in Definition \ref{def:blowupgraph} will only appear in our results on $3$-graphs.  

We now define an auxiliary class associated to a hereditary $k$-graph property $\calH$ which will play a crucial role in our results.

\begin{definition}\label{def:bhgraphs}
 Suppose $k\geq 2$ and $\calH$ is a hereditary $k$-graph property.  Recalling that $\calG^{(k)}$ denotes the class of all finite $k$-graphs, we set
$$
{\bf B}_{\calH}=\{H\in \calG^{(k)}: \text{ for all integers $n\geq 1$, $\calH$ contains an $n$-blowup of $H$}\}.
$$ 
\end{definition}

This auxiliary class ${\bf B}_{\calH}$ is always a hereditary $k$-graph property contained in $\calH$, as we now show.

\begin{fact}\label{fact:bhhgraphs}
For any hereditary $k$-graph property $\calH$, ${\bf B}_{\calH}$ is a hereditary $k$-graph property satisfying ${\bf B}_{\calH}\subseteq \calH$.
\end{fact}
\begin{proof}
We first show ${\bf B}_{\calH}\subseteq \calH$. Fix $H\in {\bf B}_{\calH}$. By definition of ${\bf B}_{\calH}$, $\calH$ contains a $1$-blowup ${\bf H}$ of $H$. By definition of a $1$-blowup, ${\bf H}\cong H$.   Since $\calH$ is closed under isomorphism, $H\in \calH$.  

We next show ${\bf B}_{\calH}$ is a hereditary $k$-graph property.  In keeping with Convention \ref{con:inf}, we need to show ${\bf B}_{\calH}$ contains arbitrarily large $k$-graphs.  By Convention \ref{con:inf}, $\calH$ contains arbitrarily large $k$-graphs. Consequently, by Ramsey's theorem, and since $\calH$ is closed under induced sub-$k$-graphs, $\calH$ contains either arbitrarily large cliques, or arbitrarily large anticliques.  Since a clique (respectively anticlique) on $nt$ vertices is an $n$-blowup of a clique  (respectively anticlique) on $t$ vertices, this implies ${\bf B}_{\calH}$ contains either arbitrarily large cliques or arbitrarily large anticliques. In particular, ${\bf B}_{\calH}$ contains arbitrarily large $k$-graphs.  

It is clear that ${\bf B}_{\calH}$ is closed under isomorphisms by definition.  Therefore, we just have left to show ${\bf B}_{\calH}$ is closed under induced sub-$k$-graphs. Suppose $H=(U,E)\in {\bf B}_{\calH}$ and $H'=(U',E')$ is an induced sub-$k$-graph of $H$.  Fix $n\geq 1$. Since $H\in {\bf B}_{\calH}$, there is an $n$-blowup $\mathbf{H}$ of $H$ with $\mathbf{H}\in \calH$.  By definition, $\mathbf{H}$ has vertex set of the form $V(\mathbf{H})=\bigsqcup_{u\in U}V_u$ and edge set $E(\mathbf{H})$ satisfying
$$
\bigcup_{u_1\ldots u_k\in E}K_k[V_{u_1},\ldots, V_{u_k}]\subseteq E(\mathbf{H})\text{ and }\left(\bigcup_{u_1\ldots u_k \in {U\choose k}\setminus E}K_k[V_{u_1},\ldots, V_{u_k}]\right)\cap E(\mathbf{H})=\emptyset.
$$
Let $\mathbf{H}'$ be the induced sub-$k$-graph of $\mathbf{H}$ with vertex set $V(\mathbf{H}')=\bigsqcup_{u\in U'}V_u$. By definition, $\mathbf{H}'$ is an $n$-blowup of $H'$.  Since $\calH$ is hereditary, $\mathbf{H}'\in \calH$. We have now shown $\calH$ contains an $n$-blowup of $H'$ for all $n\geq 1$, and consequently, $H'\in {\bf B}_{\calH}$.
\end{proof}

We now state the theorem which gives us the crucial connection between ${\bf B}_{\calH}$ and the closeness of Definition \ref{def:close}.  

\begin{theorem}\label{thm:blowupthm}
Suppose $\calH$ is a hereditary graph (respectively $3$-graph) property and $\calF$ is a finite collection of finite graphs (respectively $3$-graphs). Then the following are equivalent.
\begin{enumerate}
\item $\calH$ is close to $\Forb(\calF)$. 
\item $\calF\cap {\bf B}_{\calH}=\emptyset$.  
\end{enumerate}
\end{theorem}

In the setting of graphs, Theorem \ref{thm:blowupthm} is a well-known (but non-trivial) consequence of the regularity lemma and the induced removal lemma \cite{Alon.2000}.  In the setting of $3$-graphs, Theorem \ref{thm:blowupthm}  is likely well-known, but the author could not find an explicit proof in the literature.  It can be proved using the stronger regularity for $3$-graphs developed by Frankl--R\"{o}dl and Gowers \cite{Frankl.2002, Gowers.20063gk} and an induced removal lemma for $3$-graphs \cite{Kohayakawa.2002}.  For the sake of completeness, we have provided a proof of Theorem \ref{thm:blowupthm} in part 3 of this series, which deals with this stronger form of regularity (see the appendix of \cite{Terry.2024c}).  A sketch of the proof of the graphs version also appears there. 

We will freely use the fact that Theorem \ref{thm:blowupthm}  applies not only to finite families $\calF$, but also to families $\calF$ containing only finitely many non-isomorphic  graphs/$3$-graphs.  It is worth noting that Theorem \ref{thm:blowupthm} is still true when $\calF$ contains infinitely many non-isomorphic graphs/$3$-graphs.  In the setting of graphs, this is due to work of Alon and Shapira  \cite{Alon.2008}. In the setting of $3$-graphs, this can be deduced from work of R\"{o}dl and Schacht \cite{Rodl.2009}.  We will not use these stronger results in this paper.

\subsection{Background on weak regularity}\label{ss:reghom}

We begin by defining weak regularity for $k$-graphs, first developed by Chung \cite{Chung.1991}  and Haviland--Thomason \cite{Haviland.1989}. 

\begin{definition}\label{def:regularity}
Fix an integer $k\geq 2$, and suppose $H=(V,E)$ is a $k$-graph.  Given nonempty sets $X_1,\ldots, X_k\subseteq V$, we say $(X_1,\ldots, X_k)$ is \emph{$\e$-regular with respect to $H$} if for all $X_1'\subseteq X_1,\ldots, X'_k\subseteq X_k$ satisfying $|X_i'|\geq \e |X_i|$ for each $i\in [k]$,  it  holds that
$$
\Big|d_H(X_1,\ldots, X_k)-d_H(X_1',\ldots, X_k')\Big|\leq \e.
$$ 
\end{definition}

We now state the definition of an $\e$-regular partition.

\begin{definition}\label{def:regp}
Fix an integer $k\geq 2$, and suppose $H=(V,E)$ is a $k$-graph.  A partition $\calP$ of $V$ is  \emph{$\e$-regular with respect to $H$} if 
$$
\left| \bigcup_{(X_1,\ldots, X_k)\in \Sigma_{\reg}}X_1\times \cdots \times X_k\right|\geq (1-\e)|V|^k,
$$
where $\Sigma_{\reg}=\{(X_1,\ldots, X_k)\in \calP^k: (X_1,\ldots, X_k)\text{ is $\e$-regular with respect to $H$}\}$. In this case, we will call $\calP$ an \emph{$\e$-regular partition of $H$}.
\end{definition}

We will drop the ``with respect to $H$" from Definitions \ref{def:regularity} and \ref{def:regp} when $H$ is clear from context.  We note Definition \ref{def:regp} is slightly different from the typical definition of a regular partition. For instance, in the case of graphs, Definition \ref{def:regp} considers pairs of the form $(X,Y)$ where $X=Y$, and also does not require $\calP$ to be an equipartition.  Definition \ref{def:regp} is the correct formulation for the purposes of this paper. In particular, we need to allow sets to be ``regular with themselves" in order to allow regular partitions of constant size. For instance, if $\calH$ is the hereditary property consisting of all finite graphs with no edges,   we want the smallest possible regular partition to have size $1$.  This is exactly what Definition \ref{def:regp} permits.  

We next state the regularity lemma for $k$-graphs, using the bounds implicit in Chung's original paper \cite{Chung.1991}.  

\begin{theorem}[Chung \cite{Chung.1991}]\label{thm:chung}
Fix an integer $k\geq 2$,  and let $f\colon \mathbb{Z}^{\geq 1}\rightarrow \mathbb{Z}^{\geq 1}$ be defined by setting $f(x)=k$ if $x\leq k$ and  $f(x)=(x-1)2^{f(x-1)\choose k}$ for $x>k$.  Extend $f$ to a function from $\mathbb{R}^{\geq 1}\rightarrow \mathbb{R}^{\geq 1}$ by setting $f(x)=f(\lceil x\rceil)$.

For all $\e>0$ there exists $T\leq f(2\e^{-4})$ such that every sufficiently large $k$-graph has an $\e$-regular equipartition with at most $T$ parts.  
\end{theorem}

Note that for $f$ as in Theorem \ref{thm:chung} and $x>k$, $f(x)\leq 2^{f(x-1)^k}\leq 2^{2^{2^{f(x-1)}}}$. When $k=2$ or $k=3$, this implies $f(x)\leq \Tw(3x)$, and consequently, the bound in Theorem \ref{thm:chung} is at most $\Tw(6\e^{-4})$. For $k=3$, this yields the bound stated in Theorem \ref{thm:chungintro} in the introduction.  When $k=2$, this matches the shape of the upper bound for  Szemer\'{e}di's Regularity Lemma that can be deduced from  \cite{Fox.2014} (note the definition of regularity used in \cite{Fox.2014} differs formally from the one used here).  Theorem \ref{thm:chung} tells us the following functions are well defined.

\begin{definition}\label{def:M3inpaper}
Given a hereditary $k$-graph property $\calH$,  define $M_{\calH}\colon(0,1)\rightarrow \mathbb{N}$ by letting $M_{\calH}(\e)$ be the smallest integer $M$ so that every sufficiently large $k$-graph in $\calH$ has an  $\e$-regular partition with at most $M$ parts.
\end{definition}

We will use in Section \ref{sec:graphs} that when $\calH$ is a hereditary graph property, and $\calH$ is close to $\calH'$ (in the sense of Definition \ref{def:close}), then $M_{\calH}$ can be roughly bounded above by $M_{\calH'}$.  To prove this we will use the following well-known averaging lemma. 

\begin{lemma}\label{lem:averaging}
Let $a,b,\e\in (0,1)$ be such that $ab=\e$. Suppose $X$ is a finite set and $A\subseteq X$ satisfies $|A|\geq (1-\e)|X|$.  For any partition $\calP$ of $X$, if we let 
$$
\Sigma=\{Y\in \calP: |A\cap Y|\geq (1-a)|Y|\},
$$
 then we have $|\bigcup_{Y\in \Sigma}Y|\geq (1-b)|X|$.
\end{lemma}
\begin{proof}
Suppose towards a contradiction $|\bigcup_{Y\in \Sigma}Y|< (1-b)|X|$.  Then 
$$
|X\setminus A|> \sum_{Y\in \calP\setminus \Sigma}a|Y|=a\sum_{Y\in \calP\setminus \Sigma}|Y|\geq ab|X|=\e|X|,
$$
a contradiction.
\end{proof}

\begin{proposition}\label{prop:closegraphs}
Suppose $\calH$ and $\calH'$ are hereditary graph properties and $\calH$ is close to $\calH'$. Then for all sufficiently small $\e>0$, $M_{\calH}(2\e)\leq M_{\calH'}(\e)$.  
\end{proposition}
\begin{proof}
Fix $\e>0$ sufficiently small, and let $\delta>0$ be sufficiently small compared to $\e$.  

Fix $G=(V,E)\in \calH$ with $|V|$ sufficiently large in terms of $\e$ and $\delta$.  We show $G$ has a $2\e$-regular partition with at most $M_{\calH'}(\e)$ parts.  Since $|V|$ is sufficiently large, there exists some $G'=(V,E')\in \calH'$ which is $\delta$-close to $G$.  Let $\calP$ be an $\e$-regular partition of $G'$ with at most $M_{\calH'}(\e)$ parts.  We show $\calP$ is $2\e$-regular with respect to $G$.  Define 
$$
\Sigma_{\reg}=\{(X,Y)\in \calP^2: (X,Y)\text{ is $\e$-regular with respect to $G'$}\}.
$$
Since $\calP$ is $\e$-regular with respect to $G'$, we have $|\bigcup_{(X,Y)\in \Sigma_{\reg}}X\times Y|\geq (1-\e)|V|^2$.  Since $G$ is $\delta$-close to $G'$, $|\overline{E}\Delta \overline{E'}|\leq \delta |V|^2$. Consequently, by Lemma \ref{lem:averaging}, if we set 
$$
\Sigma_{\textrm{close}}=\{(X,Y)\in \calP^2: |(\overline{E}\Delta \overline{E'})\cap (X\times Y)|\leq \sqrt{\delta}|X||Y|\},
$$
 then $|\bigcup_{(X,Y)\in \Sigma_{\textrm{close}}}X\times Y|\geq (1-\sqrt{\delta})|V|^2$. We now have 
$$
\left|\bigcup_{(X,Y)\in \Sigma_{\reg} \cap \Sigma_{\textrm{close}}}X\times Y\right|\geq (1-\e-\sqrt{\delta})|V|^2\geq (1-2\e)|V|^2,
$$
 where the last inequality is because $\delta$ is sufficiently small compared to $\e$.  Therefore, it suffices to show  that each pair $(X,Y)$ from $\Sigma_{\reg}\cap \Sigma_{\textrm{close}}$ is $2\e$-regular with respect to $G$. To this end, fix $(X,Y)\in \Sigma_{\reg}\cap \Sigma_{\textrm{close}}$, and suppose $X'\subseteq X$ and $Y'\subseteq Y$ satisfy $|X'|\geq 2\e |X|$ and $|Y'|\geq 2\e|Y|$.  Then by the triangle inequality,
\begin{align*}
&|d_{G}(X',Y')-d_{G}(X,Y)|\\
&\leq |d_{G}(X',Y')-d_{G'}(X',Y')|+ |d_{G'}(X',Y')-d_{G'}(X,Y)|+ |d_{G'}(X,Y)-d_{G}(X,Y)|\\
&\leq \frac{\sqrt{\delta}|X||Y|}{|X'||Y'|}+\e +\frac{\sqrt{\delta}|X||Y|}{|X||Y|}\\
&\leq \sqrt{\delta}(2\e)^{-2}+\e+\sqrt{\delta}\\
&\leq 2\e,
\end{align*} 
where the second inequality is because $(X,Y)\in \Sigma_{\reg}\cap \Sigma_{\textrm{close}}$, the third inequality is because $|X'|\geq 2\e|X|$ and $|Y'|\geq 2\e |Y|$, and the last inequality is because $\delta$ is sufficiently small compared to $\e$.   This shows $\calP$ is $2\e$-regular with respect to $G$, and consequently, $M_{\calH}(2\e)\leq M_{\calH'}(\e)$, as desired.
\end{proof}

The proof of Proposition \ref{prop:closegraphs} above can be adjusted to show $M_{\calH}(\e+o(1))\leq M_{\calH'}(\e)$, however this would not impact the bounds in our theorems  significantly.  We will also use the following analogue of Proposition \ref{prop:closegraphs} in the setting of $3$-graphs (we omit the proof as it is similar to the proof of Proposition \ref{prop:closegraphs}). 

\begin{proposition}\label{prop:close3graphs}
Suppose $\calH,\calH'$ are hereditary $3$-graph properties and $\calH$ is close to $\calH'$. Then for all sufficiently small $\e>0$, $M_{\calH}(2\e)\leq M_{\calH'}(\e)$.
\end{proposition}

\subsection{Background on homogeneity}\label{ss:hom}
We begin by defining homogeneous partitions. 

\begin{definition}\label{def:hom}
Suppose $H=(V,E)$ is a $k$-graph.
\begin{enumerate}
\item Given nonempty sets $X_1,\ldots, X_k\subseteq V$, we say $(X_1,\ldots, X_k)$ is \emph{$\e$-homogeneous with respect to $H$} if $d_H(X_1,\ldots, X_k)\in [0,\e)\cup (1-\e,1]$.  
\item We say a partition $\calP$ of $V$ is \emph{$\e$-homogeneous with respect to $H$} if 
$$
\left|\bigcup_{(X_1,\ldots, X_k)\in \Sigma_{\hom}}X_1\times \cdots \times X_k\right|\geq (1-\e)|V|^k,
$$
where $\Sigma_{\hom}=\{(X_1,\ldots, X_k)\in \calP^k: (X_1,\ldots, X_k)\text{ is $\e$-homogeneous with respect to $H$}\}$. In this case, we say   $\calP$ is an \emph{$\e$-homogeneous partition of $H$}.
\end{enumerate}
\end{definition}

We will drop the ``with respect to $H$" from Definition \ref{def:hom} when $H$ is clear from context. We can now define a more general version of the growth function in Definition \ref{def:Mhom2}.

\begin{definition}\label{def:Mhom}
Let $\calH$ be a hereditary $k$-graph property.  Define $M^{\hom}_{\calH}\colon (0,1)\rightarrow \mathbb{N}\cup \{\infty\}$ by letting $M^{\hom}_{\calH}(\e)$ be the smallest integer $M$ (or $\infty$ if no such integer exists) such that any sufficiently large $k$-graph in $\calH$ has an $\e$-homogeneous partition with at most $M$ parts.
\end{definition}

The function $M_{\calH}^{\hom}$ plays a crucial role  in both Theorem \ref{thm:alljumphom} for graphs, and Theorem \ref{thm:weakhom} for $3$-graphs.  In fact, we will show that when $M_{\calH}^{\hom}$ takes finite values, it largely controls the behavior of $M_{\calH}$.  Thus, we will need to understand when we can guarantee $M_{\calH}^{\hom}<\infty$.  In the graphs case, this question is addressed in Subsection \ref{ss:vc} on VC-dimension. In the setting of $3$-graphs, this is considered in Section  \ref{sec:weakreg},  which covers two versions of VC-dimension in $3$-graphs.   

We spend the rest of this subsection stating basic facts we will need about homogeneous partitions. First, in the setting of graphs, homogeneous pairs are well-known to be regular, as the next proposition shows.  The proof is an exercise we include for completeness.

\begin{proposition}\label{prop:homimpliesregulargraphs}
Let $\e \in (0,1)$. If $G=(V, E)$ is a graph, and $X,Y\subseteq V$ satisfy $d_G(X,Y)\in [0,\e)\cup (1-\e,1]$, then $(X,Y)$ is $\e^{1/3}$-regular with respect to $G$.
\end{proposition}
\begin{proof}
Assume first $d_G(X,Y)\leq \e$.  Fix $X'\subseteq X$, and $Y'\subseteq Y$ satisfying $|X'|\geq \e^{1/3}|X|$, and $|Y'|\geq \e^{1/3}|Y|$. Then 
$$
|\overline{E}\cap (X'\times Y')|\leq \e |X||Y| \leq \e (\e^{-1/3}|X'|)(\e^{-1/3}|Y'|)= \e^{1/3} |X'||Y'|.  
$$
We now have that  $0\leq d_G(X',Y')\leq \e^{1/3}$.  Combining with the fact that $0\leq d_G(X,Y)\leq \e$, we can conclude that $|d_G(X,Y)-d_G(X',Y')|\leq \e^{1/3}$. We omit the argument for the case $d_G(X,Y)\geq 1-\e$ as it is essentially identical.
\end{proof}

 Proposition  \ref{prop:homimpliesregulargraphs} implies that an $\e$-homogeneous partition of a graph is automatically $\e^{1/3}$-regular.   An immediate corollary of this is that $M_{\calH}^{\hom}$ roughly upper bounds $M_{\calH}$ (when $M_{\calH}^{\hom}$ takes finite values). 

\begin{fact}\label{fact:homub}
Suppose $\calH$ is a hereditary graph property such that $M_{\calH}^{\hom}(\e)<\infty$ for all $\e\in (0,1)$.  Then for all $\e\in (0,1)$, $M_{\calH}(\e)\leq M_{\calH}^{\hom}(\e^3)$.
\end{fact}

It is because of Fact \ref{fact:homub} that the function $M_{\calH}^{\hom}$ plays an important role in the proof of Theorem \ref{thm:alljumphom}.  We will use in Section \ref{sec:graphs} the analogue of Proposition \ref{prop:closegraphs} for $M_{\calH}^{\hom}$. 

\begin{proposition}\label{prop:closehomgraphs}
Suppose $\calH$ and $\calH'$ are hereditary graph properties, and assume that for all $\e\in (0,1)$, $M_{\calH'}^{\hom}(\e)<\infty$.   If $\calH$ is close to $\calH'$, then $M^{\hom}_{\calH}(2\e)\leq M^{\hom}_{\calH'}(\e)$ holds for all sufficiently small $\e>0$.
\end{proposition}
\begin{proof}
Let $\e>0$ be sufficiently small.   Let $\delta>0$ be sufficiently small compared to $\e$, and let $N$ be an integer which is sufficiently large in terms of $\e$ and $\delta$.  Fix $G=(V,E)\in \calH$ with $|V|\geq N$. Since $N$ is sufficiently large, there exists some $G'\in \calH'$ which is $\delta$-close to $G$.  An almost identical argument to the proof of Proposition \ref{prop:closegraphs} shows any $\e$-homogeneous partition of $G'$ is also $2\e$-homogeneous with respect to $G$.  The argument is sufficiently similar that we omit it here.  This implies $M_{\calH}^{\hom}(2\e)\leq M^{\hom}_{\calH'}(\e)$.
\end{proof}

We conclude this section by stating $3$-graph analogues of the results above. First, homogeneous triples in $3$-graphs are regular.   

\begin{proposition}\label{prop:2.21}
Let $\e\in (0,1)$ and assume $H$ is a $3$-graph.  Assume $X,Y,Z\subseteq V(H)$ are nonempty, and  $d_H(X,Y,Z)\in [0,\e)\cup (1-\e,1]$. Then $(X,Y,Z)$ is $\e^{1/4}$-regular for $H$.
\end{proposition}

We omit the proof of Proposition \ref{prop:2.21} as it is almost identical to the proof of Proposition \ref{prop:homimpliesregulargraphs}.   An immediate corollary of Proposition \ref{prop:2.21} is the following analogue of Fact \ref{fact:homub}.

\begin{fact}\label{lem:homup3graphs}
Suppose $\calH$ is a hereditary $3$-graph property and $M_{\calH}^{\hom}(\e)<\infty$ for all $\e\in (0,1)$.   Then $M_{\calH}(\e)\leq M_{\calH}^{\hom}(\e^4)$.
\end{fact}

We will use the following analogue of Proposition \ref{prop:closehomgraphs} in several sections on $3$-graphs. 

\begin{proposition}\label{prop:closehom}
Suppose $\calH$ and $\calH'$ are hereditary $3$-graph properties and $M_{\calH'}^{\hom}(\e)<\infty$ for all $\e\in (0,1)$.   If $\calH$ is close to $\calH'$, then $M^{\hom}_{\calH}(2\e)\leq M^{\hom}_{\calH'}(\e)$.
\end{proposition}

We leave the proof of Proposition \ref{prop:closehom} as an exercise to the reader, as it is similar to the proof of Proposition \ref{prop:closehomgraphs}.

\subsection{VC-dimension}\label{ss:vcbasic}
The notion of VC-dimension will appear several times in the paper.  We define in this subsection what is needed for our applications, and refer the reader to the literature for more extensive background on the topic (see e.g. \cite{Matousek}).

\begin{definition} A \emph{set system} is a pair $(X,\calF)$ where $X$ is a finite set, called the \emph{ground set}, and $\calF\subseteq \calP(X)$ is a nonempty collection of subsets of $X$.   
\end{definition}

We note that the definition of a set system usually allows the ground set to be infinite, but we will not need that level of generality in this paper.  

\begin{definition}
Given a set system $(X,\calF)$ and $Y\subseteq X$, we say $Y$ is \emph{shattered by $\calF$} if 
$$
|\{F\cap Y: F\in \calF\}|=2^{|Y|}.
$$
\end{definition}

 We are now ready to define the VC-dimension of a set system.

\begin{definition}\label{def:vcsetsystem}
Given a set system $(X,\calF)$, the \emph{VC-dimension of $(X,\calF)$} is defined to be
\begin{align*}
\VC(X,\calF)=\max\{|Y|: Y\subseteq X \text{ and $Y$ is shattered by $\calF$}\}.
\end{align*}
 \end{definition}
 
Observe that the VC-dimension of a set system $(X,\calF)$ is always an integer between $0$ and $\lfloor \log_2|\calF|\rfloor$.

\section{The graphs case}\label{sec:graphs}
This section contains the proof of Theorem \ref{thm:alljumphom}, which considers the growth of regular partitions in hereditary graph properties. As this section is somewhat long,  we provide here an outline. First, we introduce VC-dimension for graphs and its connection to homogeneous partitions in Subsection \ref{ss:vc}.  In Subsection \ref{ss:polytowergraphs}, we give an exposition of the previously known jump between polynomial and tower growth rates.  In Subsection \ref{ss:lbgraphs} we prove a general lower bound on regular partitions of certain graph blowups. Finally, in Subsection \ref{ss:contpolygraphs}, we prove the existence of the jump between constant and polynomial growth rates, after which we deduce Theorem  \ref{thm:alljumphom}.

\subsection{VC-dimension in graphs}\label{ss:vc}
We define the VC-dimension of a graph as follows.
 
\begin{definition}\label{def:vc}
Suppose $G=(V,E)$ is a finite graph.  The \emph{VC-dimension of $G$}, denoted $\VC(G)$, is the VC-dimension of the set system $(V,\calF)$, where 
$$
\calF=\{N_G(v): v\in V\}.
$$ 
\end{definition}

We extend the notion of VC-dimension to hereditary graph properties below.  

\begin{definition}\label{def:vchp}
Suppose $\calH$ is a hereditary graph property. The \emph{VC-dimension of $\calH$} is defined to be
$$
\VC(\calH)=\sup\{\VC(G): G\in \calH\}  \in \mathbb{N}\cup \{\infty\}.
$$
When $\VC(\calH)<\infty$, we say $\calH$ \emph{has finite VC-dimension}, and when $\VC(\calH)=\infty$, we say $\calH$ \emph{has infinite VC-dimension}.
\end{definition}

Our next task is to state a needed combinatorial characterization  of when $\VC(\calH)=\infty$. For this, we define the following class of ``powerset" graphs.   

\begin{definition}[Powerset Graphs]\label{def:pset}
Given $k\geq 1$, define $\PS(k)$ to be the set of all graphs $G=(V,E)$ where $V$ admits an enumeration of the form $V=\{a_i: i\in [k]\}\cup \{b_S:S\subseteq [k]\}$ satisfying
\begin{align}\label{al:psdisjoint}
\{a_i: i\in [k]\}\cap \{b_S: S\subseteq [k]\}=\emptyset,
\end{align}
and such that $\{a_ib_S:i\in S\}\subseteq E$ and $\{a_ib_S: i\notin S\}\cap E=\emptyset$.
\end{definition}

We note the disjointness condition (\ref{al:psdisjoint}) in Definition \ref{def:pset} will be useful in later applications. We can now state the necessary characterization of when a hereditary graph property has infinite VC-dimension. 

 \begin{fact}\label{fact:bipvc}
Suppose $\calH$ is a hereditary graph property. The following are equivalent.
\begin{enumerate}
\item $\VC(\calH)=\infty$.
\item One of the following holds.
\begin{enumerate}
\item $\calH$  contains every bipartite graph.
\item $\calH$ contains every co-bipartite graph.
\item $\calH$ contains every split graph.
\end{enumerate}
\item $\calH\cap \PS(k)\neq \emptyset$ for all $k\geq 1$.
\item ${\bf B}_{\calH}\cap \PS(k)\neq \emptyset$ for all $k\geq 1$. 
\end{enumerate}
\end{fact} 

The equivalence of (1) and (2) is a well-known folklore result which we will not prove here.  Given the equivalence of (1) and (2), proving Fact \ref{fact:bipvc} becomes straightforward. Indeed, it is an exercise to show (2) implies each of (3) and (4), while it is almost immediate from the relevant definitions that  (3) and (4) each imply (1).  We note that Fact \ref{fact:bipvc} implies $\calH$ has finite VC-dimension if and only if ${\bf B}_{\calH}$ has finite VC-dimension.  

It is not difficult to see there is a connection between VC-dimension and homogeneity.  For example, a standard application of the counting lemma shows regular pairs in graphs of bounded VC-dimension are homogeneous.  

\begin{fact}\label{lem:graphhom}
For all $\e\in (0,1)$ and integers $k\geq 0$, there are $\delta\in (0,1)$ and $n_0\geq 1$ so that the following holds.  Suppose $G=(V,E)$ is a graph with VC-dimension at most $k$, and $U,W\subseteq V$ are such that $\min\{|U|,|W|\}\geq n_0$, and such that $(U,W)$ is a $\delta$-regular pair with respect to $G$.  Then 
$$
d_G(U,W)\in [0,\e)\cup (1-\e,1].
$$
\end{fact}

Fact \ref{lem:graphhom} implies that if $G$ is a large graph with  $\VC(G)\leq k$, then any sufficiently regular partition of $G$ will be automatically $\e$-homogeneous.  This immediately tells us that when $\calH$ has finite VC-dimension, $M_{\calH}^{\hom}(\e)<\infty$.    It is not hard to show the converse is also true (e.g. using Fact \ref{fact:bipvc}(2)). This yields the following.

\begin{fact}
Given a hereditary graph property $\calH$, $\VC(\calH)<\infty$ if and only if $M_{\calH}^{\hom}(\e)<\infty$ for all $\e\in (0,1)$.
\end{fact}

While Fact \ref{lem:graphhom} tells us graphs of bounded VC-dimension admit homogeneous partitions, the theorems of Alon--Fischer--Newman \cite{Alon.2007} and Lov\'{a}sz--Szegedy \cite{Lovasz.2010} mentioned in the introduction show something much stronger, namely that such graphs admit homogeneous partitions \emph{with efficient bounds}.  We give a more detailed statement of this result here, with improved  bounds from \cite{Fox.2017bfo}.

\begin{theorem}\label{thm:foxpachsukagain}
For all $\e\in (0,1)$ and all integers $k\geq 0$, there is $c=c(k)$ so that the following holds.  Suppose $G=(V,E)$ is a sufficiently large graph with VC-dimension at most $k$.  Then $G$ has an $\e$-homogeneous equipartition of size $K$, for some $8\e^{-1}\leq K\leq c\e^{-2k-1}$.
\end{theorem}

Thus, by Theorem \ref{thm:foxpachsukagain}, we know that when $\calH$ has finite VC-dimension, not only is $M_{\calH}^{\hom}(\e)$ finite, it is in fact bounded above by a polynomial in $\e^{-1}$.  This plays a crucial role in the polynomial to tower jump in the next subsection.
 
\subsection{Polynomial to tower jump}\label{ss:polytowergraphs}

In this section, we provide an exposition of the jump between polynomial and tower growth rates.  This dichotomy was first observed by Fox, and is stated explicitly in a paper by Alon--Fox--Zhao \cite{Alon.2018is}.  As the reader will see, it is an immediate corollary of Theorem \ref{thm:foxpachsukagain} and a construction of Fox--Lov\'{a}sz \cite{Fox.2014}.  We include the details in part to guide analogous arguments later in the paper.  We begin by stating the relevant theorem from \cite{Fox.2014}.

\begin{theorem}[Fox--Lov\'{a}sz \cite{Fox.2014}]\label{thm:bipgowers}
There is a constant $c>0$ such that for all sufficiently small $\e>0$, and all $n_0\in \mathbb{N}$, there exists a bipartite graph $G=(U\cup W, E)$ such that $|U|=|W|\geq n_0$, and such that any $\e$-regular partition of $G$  has at least $\Tw(c\e^{-2})$ many parts. 
\end{theorem}

We remark that in \cite{Fox.2014}, the authors do not state explicitly that the $G$ they construct for Theorem \ref{thm:bipgowers}  is bipartite, or that the parts $U$ and $W$ are of equal size. However, it is made clear that this is the case during the proof (see Section 3 of \cite{Fox.2014}). Moreover, it is not difficult to see the analysis in \cite{Fox.2014} relies only on the behavior of the edges between $U$ and $W$, and would work equally well if $U$ and/or $W$ were chosen to be cliques rather than anticliques.   This yields the following corollary.

\begin{corollary}\label{cor:bipgowers}
There is a constant $c>0$ such that for all sufficiently small $\e>0$, and all $n_0\in \mathbb{N}$, there exist sets $U, W$ satisfying $|U|=|W|\geq n_0$, and a bipartite graph $G_1=(U\cup W, E_1)$, a co-bipartite graph $G_2=(U\cup W, E_2)$, and a split graph $G_3=(U\cup W, E_3)$, such that any $\e$-regular partition of $G_1$, $G_2$, or $G_3$ has at least $\Tw(c\e^{-2})$ many parts.  
\end{corollary}

We now deduce the existence of a jump between polynomial and tower growth.

\begin{corollary}\label{cor:fastgraphs}
For any hereditary graph property $\calH$, one of the following holds. 
\begin{enumerate}
\item $\VC(\calH)=\infty$.  In this case,  
$$
\Tw(\Omega(\e^{-2}))\leq M_{\calH}(\e)\leq \Tw(O(\e^{-4})).
$$
\item There is $k\in \mathbb{N}$ so that $\VC(\calH)=k$. In this case, there is a constant $c=c(k)>0$ such that $M_{\calH}(\e)\leq M_{\calH}^{\hom}(\e^3)\leq c\e^{-6k-3}$.
\end{enumerate}
\end{corollary}
\begin{proof}
Suppose first $\VC(\calH)=\infty$.  The upper bound $M_{\calH}(\e)\leq \Tw(O(\e^{-4}))$ holds by the discussion following Theorem \ref{thm:chung}.  For the lower bound,  let $\e\in (0,1)$ be sufficiently small, and fix any integer $n_0\geq 1$.  Let $G_1$, $G_2$, and $G_3$ be as in Corollary \ref{cor:bipgowers} for $\e$ and $n_0$.  By Fact \ref{fact:bipvc}(2), $\calH$   contains one of $G_1$, $G_2$, or $G_3$. Consequently, $M_{\calH}(\e)\geq \Tw(\Omega(\e^{-2}))$ by Corollary \ref{cor:bipgowers}.  This concludes our verification of (1).

Suppose now $\VC(\calH)=k<\infty$. By Theorem \ref{thm:foxpachsukagain}, there is $c=c(k)>0$ such that for all sufficiently small $\e>0$, $M_{\calH}^{\hom}(\e)\leq c\e^{-2k-1}$.  Combining with Fact \ref{fact:homub}, we have  
$$
M_{\calH}(\e)\leq M_{\calH}^{\hom}(\e^3)\leq c\e^{-6k-3}.
$$
  Thus (2) holds.\end{proof}

\subsection{Lower bound lemma for graphs}\label{ss:lbgraphs}

In this subsection, we prove a  lower bound lemma  for regular partitions of certain graph blowups. This result, Lemma \ref{lem:blowup0} below, will produce the lower bounds in both the constant and polynomial ranges of Theorem \ref{thm:alljumphom}. Lemma \ref{lem:blowup0} also serves as a warmup for a more complicated analogue in the 3-graph setting, Lemma \ref{lem:12blowup}.  

We begin by defining a canonical equivalence relation on the vertices of any graph.  Roughly speaking, two vertices are equivalent if they are connected to the same vertices (other than themselves).

\begin{definition}\label{def:sim}
Suppose $G=(V,E)$ is a graph.  Given $x,x'\in V$, define $x\sim_G x'$ if and only if 
$$
N_G(x)\cap (V\setminus \{x,x'\})=N_G(x')\cap (V\setminus \{x,x'\}).
$$  
\end{definition}

It is not difficult to show that for any graph $G$, $\sim_G$  is an equivalence relation on $V(G)$, and further, each $\sim_G$-equivalence class is either a clique or anticlique.  We will refer to these equivalence classes as the $\sim_G$-classes of $G$.

Informally, Lemma \ref{lem:blowup0} tells us that if ${\bf G}$ is an $n$-blowup of a graph $G$, then a sufficiently regular partition of ${\bf G}$ must approximately refine the partition of ${\bf G}$ generated by the blowups of the $\sim_G$-classes of $G$.  The reader may wish to review Definition \ref{def:blowupgraph} which covers blowups.

 \begin{lemma}[Lower bound lemma for graphs]\label{lem:blowup0}
Let $G=(U,E)$ be a graph, and let $U_1,\ldots, U_t$ be an enumeration of its $\sim_G$-classes.  Suppose $n\in\mathbb{N}^{\geq 1}$, $s,\e\in (0,1)$ satisfy
$$
0<\e\leq \min\{1/4,(2|U|)^{-1/s}\}.
$$  
Assume ${\bf G}$ is an $n$-blowup of $G$ with vertex set $V({\bf G})=\bigsqcup_{x\in U}V_x$ as in Definition \ref{def:blowupgraph}, and for each $1\leq i\leq t$, let ${\bf U}_i=\bigcup_{x\in U_i}V_x$.   Then for any $\e$-regular partition $\calP$ of ${\bf G}$,  there is a set $\calP'\subseteq \calP$ satisfying the following.
\begin{enumerate} 
\item $\left|\bigcup_{P\in \calP'}P\right|\geq (1-\e^{1-s})|V({\bf G})|$.
\item For all $P\in \calP'$, there is $1\leq i\leq t$ such that $|P\cap {\bf U}_i|\geq (1-\e^{1-s})|P|$.
\end{enumerate}
 \end{lemma}
 
 Before proving Lemma \ref{lem:blowup0}, we require a preliminary result (Lemma \ref{lem:int0} below) to help us identify the blowups of $\sim_G$-classes appearing in its conclusion. Recalling Notation \ref{not:01nbrs}, it is important to notice that, given a graph $G=(V,E)$ and a vertex $v\in V$, $v\notin N_{G^1}(v)$ and  $v\notin N_{G^0}(v)$. This yields the following \emph{partition} of $V$:
$$
V =N_{G^1}(v)\sqcup N_{G^0}(v)\sqcup \{v\}.
$$ 
 
 \begin{lemma}\label{lem:int0}
For any graph $G=(U,E)$ and any function $\alpha\colon U\rightarrow \{0,1\}$, the set $ \bigcap_{y\in U}(\{y\}\cup N_{G^{\alpha(y)}}(y))$  is contained in a single $\sim_G$-class.    
 \end{lemma}
 \begin{proof}  
Fix $\alpha\colon U\rightarrow \{0,1\}$ and let $W= \bigcap_{y\in U}(\{y\}\cup N_{G^{\alpha(y)}}(y))$.  Assume towards a contradiction $W$ is not contained in a single $\sim_G$-class. Then there exist $v,v'\in W$  such that $v\nsim_G v'$.  By definition of $\sim_G$, there are $z\in U\setminus \{v,v'\}$ and $\tau\in \{0,1\}$ such that $zv\in E^{\tau}$ and $zv'\in E^{1-\tau}$.  Since $v,v'\in W$, the definition of $W$ implies $\{v,v'\}\subseteq \{z\}\cup N_{G^{\alpha(z)}}(z)$. Since $z$ is distinct from both $v$ and $v'$, this implies $\{v,v'\}\subseteq   N_{G^{\alpha(z)}}(z)$, i.e. $zv$ and $zv'$ are both in $E^{\alpha(z)}$. However, this is impossible, since either $\alpha(z)=\tau$ and $zv'\notin E^{\alpha(z)}$, or $\alpha(z)=1-\tau$ and $zv\notin E^{\alpha(z)}$. Thus we have arrived at a contradiction.   
\end{proof}

 We now prove Lemma \ref{lem:blowup0}.
 
 \vspace{2mm}
 
 \noindent{\bf Proof of Lemma \ref{lem:blowup0}.}  By Definition \ref{def:blowupgraph}, ${\bf G}$ has vertex set of the form $V({\bf G})=\bigsqcup_{x\in U}V_x$, where each $V_x$ has size $n$, and edge set $E({\bf G})$ satisfying 
\begin{align}\label{al:blowup00}
 \bigcup_{xy\in E}K_2[V_x,V_y]\subseteq E({\bf G})\text{ and }  \left(\bigcup_{xy\in {U\choose 2}\setminus E}K_2[V_x,V_y] \right)\cap  E({\bf G})=\emptyset.
 \end{align}
 Note $|V({\bf G})|=|U|n$.  Suppose $\calP$ is an $\e$-regular partition of ${\bf G}$, and define
\begin{align*}
\Sigma_{\reg}&=\{(X,Y)\in \calP^2: (X,Y)\text{ is $\e$-regular with respect to ${\bf G}$}\}\text{ and }\\
\calE_{\reg}&=\bigcup_{(X,Y)\in \Sigma_{\reg}}X\times Y\subseteq V({\bf G})\times V({\bf G}).
\end{align*}
Since $\calP$ is $\e$-regular, we have $|\calE_{\reg}|\geq (1-\e)|V({\bf G})|^2$. Now set 
$$
\calE_{\reg}'=\{x\in V({\bf G}): |N_{\calE_{\reg}}(x)|\geq (1-\e^{s})|V({\bf G})|\}.
$$
Since $|\calE_{\reg}|\geq (1-\e)|V({\bf G})|^2$, Lemma \ref{lem:averaging} implies
 \begin{align}\label{al:blowup0}
 |\calE_{\reg}'|\geq (1-\e^{1-s})|V({\bf G})|.
 \end{align}
 It follows from its definition that $\calE_{\reg}'$ is a union of sets from $\calP$. Let $\calP'\subseteq \calP$ be such that $\calE_{\reg}'=\bigcup_{P\in \calP'}P$.   By construction, we have $|\bigcup_{P\in \calP'}P|=|\calE_{\reg}'|\geq (1-\e^{1-s})|V({\bf G})|$, so (i) holds for this choice of $\calP'$.  The rest of the proof is devoted to showing $\calP'$ satisfies (ii). To this end, for the rest of the proof, we fix $P\in \calP'$. Our goal is to show there is some $1\leq i\leq t$ such that $|P\cap {\bf U}_i|\geq (1-\e^{1-s})|P|$.   

Recall that every $y\in U$ gives rise to a partition $U =N_{G^0}(y)\sqcup N_{G^1}(y)\sqcup\{y\}$.  This induces a corresponding partition of $V({\bf G})$ given by
\begin{align}\label{al:blowup0part}
V({\bf G}) =\left(\bigcup_{x\in   N_{G^1}(y)}V_x\right)\bigsqcup\left(\bigcup_{x\in N_{G^0}(y)}V_x\right)\bigsqcup V_y.
\end{align}
 We next show that $P$ cannot substantially intersect two specific parts of such a partition. 

\begin{claim}\label{cl:qgood1}
For all $y\in U$, 
\begin{align}\label{al:yu}
\min\left\{\left|P\cap \left(\bigcup_{x\in   N_{G^1}(y)}V_x\right)\right|, \left|P\cap \left(\bigcup_{x\in N_{G^0}(y)}V_x\right)\right| \right\}<\e|P|.
\end{align}
\end{claim}
\begin{proof}
Fix $y\in U$, and suppose towards a contradiction that (\ref{al:yu}) is false for $y$. The definition of $\calE_{\reg}$ implies that for all $v,v'\in P$, $N_{\calE_{\reg}}(v)=N_{\calE_{\reg}}(v')$. Further, since $P\in \calP'$, $P\subseteq \calE_{\reg}'$. Consequently, for all $v\in P$, $|V({\bf G})\setminus N_{\calE_{\reg}}(v)|\leq \e^s|V({\bf G})|$.  Combining these observations yields 
\begin{align}
 \label{al:b11}\left|V_y\cap \left(\bigcup_{v\in P}N_{\calE_{\reg}}(v)\right)\right|\geq |V_y|-\left|V({\bf G})\setminus \left(\bigcup_{v\in P}N_{\calE_{\reg}}(v)\right)\right|& \geq |V_y|-\e^{s}|V({\bf G})|\\
&\nonumber= |V({\bf G})|(|U|^{-1} -\e^{s})\\
&\nonumber\geq \e^{s}|V({\bf G})|\\
&\nonumber\geq \e\left|\bigcup_{v\in P}N_{\calE_{\reg}}(v)\right|,
\end{align}
where the equality is because $|V({\bf G})|=|U|n=|U||V_y|$,  the third inequality is because $\e\leq (2|U|)^{-1/s}$, and the last inequality is because $0<s<1$. Since $\bigcup_{v\in P}N_{\calE_{\reg}}(v)$ is a union of sets from $\calP$, (\ref{al:b11}) implies there must be some $P' \in \calP$ such that $P'\subseteq   \bigcup_{v\in P}N_{\calE_{\reg}}(v) $, and such that 
\begin{align}\label{al:vy}
|P' \cap V_y|\geq \e|P'|. 
\end{align}
 By construction, $P\times P'\subseteq \calE_{\reg}$, i.e. $(P,P')\in \Sigma_{\reg}$.  Thus, by (\ref{al:vy}) and our assumption that (\ref{al:yu}) fails for $y$, we have 
 \begin{align*}
\max\Bigg\{&\Bigg|d_{{\bf G}}\Bigg(P\cap \Bigg(\bigcup_{  {x\in N_{G^1}(y) }}V_x \Bigg), P'\cap V_y\Bigg)-d_{{\bf G}}(P,P')\Bigg|, \\
&\nonumber \Bigg|d_{{\bf G}}\Bigg(P\cap \Bigg(\bigcup_{{x\in N_{G^0}(y) }}V_x \Bigg), P'\cap V_y\Bigg)-d_{{\bf G}}(P,P')\Bigg| \Bigg\}\leq \e.
\end{align*}
By the triangle inequality, this yields
\begin{align*} 
\Bigg|d_{{\bf G}}\Bigg(P\cap \Bigg(\bigcup_{  {x\in N_{G^1}(y) }}V_x \Bigg), P'\cap V_y\Bigg)-d_{{\bf G}}\Bigg(P\cap \Bigg(\bigcup_{{x\in N_{G^0}(y) }}V_x \Bigg), P'\cap V_y\Bigg)\Bigg|\leq 2\e<1,
\end{align*}
where the final inequality is because $\e\leq 1/4$ by assumption.  However, this is impossible since, by (\ref{al:blowup00}),  
$$
d_{{\bf G}}\left(P\cap \left(\bigcup_{x\in N_{G^1}(y)}V_x \right), P'\cap V_y\right)=1,
$$
while 
$$
d_{{\bf G}}\left(P\cap \left(\bigcup_{x\in N_{G^0}(y) }V_x \right), P'\cap V_y\right)=0.
$$
This finishes the proof of Claim \ref{cl:qgood1}.
\end{proof}

By Claim \ref{cl:qgood1}, we have that for all $y\in U$, there exists $\beta(y)\in \{0,1\}$ such that 
\begin{align}\label{al:blowup0alpha} 
\Bigg| P\cap \Bigg(\bigcup_{x\in  N_{G^{\beta(y)}}(y)}V_x \Bigg) \Bigg|<\e |P|. 
\end{align}
For each  $y\in U$, let $\alpha(y)$ be such that $\{\alpha(y),\beta(y)\}=\{0,1\}$. Then, recalling (\ref{al:blowup0part}), we have
\begin{align}\label{al:blowup0part2}
P\setminus \Bigg(\bigcup_{x\in  N_{G^{\beta(y)}}(y)}V_x \Bigg)= P\cap \Bigg(V_y\cup \bigcup_{x\in  N_{G^{\alpha(y)}}(y)}V_x \Bigg)=P\cap \Bigg(\bigcup_{x\in \{y\}\cup N_{G^{\alpha(y)}}(y)}V_x\Bigg).
\end{align}
We use the values $\alpha(y)$ for $y\in U$ to define the following subset of $V(G)$. 
$$
W=\bigcap_{y\in U}\Big(\{y\}\cup N_{G^{\alpha(y)}}(y)\Big).
$$
By Lemma \ref{lem:int0}, there exists $1\leq i\leq t$ such that $W\subseteq U_i$, and consequently, $\bigcup_{x\in W}V_x\subseteq {\bf U}_i$. Combining this with  (\ref{al:blowup0part2}) yields the following. 
\begin{align*} 
\left| P\cap {\bf U}_i  \right| \geq \left| P\cap \left(\bigcup_{x\in W}V_x \right)\right| &=\left| P\cap \left(\bigcap_{y\in U}\left( \bigcup_{x\in \{y\}\cup N_{G^{\alpha(y)}}(y)}V_x \right)\right)\right|\\
&=\left|P\setminus \left( \bigcup_{y\in U}\left( \bigcup_{x\in N_{G^{\beta(y)}}(y)}V_x\right)\right)\right|\\
&\geq |P|-\sum_{y\in U} \left|P\cap \left( \bigcup_{x\in N_{G^{\beta(y)}}(y)}V_x\right)\right|\\
&\geq |P|- \e |U||P|\\
&\geq (1-\e^{1-s})|P|,
\end{align*}
where the third inequality is by (\ref{al:blowup0alpha}), and the last inequality is because $\e<|U|^{-1/s}$.   This concludes the proof of (ii).  \qed
 
\vspace{2mm} 

We end this subsection by proving an easy corollary of Lemma \ref{lem:blowup0} which we will use in the case of a graph $G$ where all $\sim_G$-classes have size at most $2$. Such graphs will play an important role in the next subsection (see Definition \ref{def:prime}).  

 \begin{corollary}\label{cor:blowup0}
Suppose $G=(U,E)$ is a graph such that every $\sim_G$-class has size at most $2$. Assume $s\in (0,1)$, $0<\e\leq \min\{1/4,(2|U|)^{-1/s}\}$, and $n\geq 1$ is an integer.  If  ${\bf G}$ is an $n$-blowup of $G$, then any $\e$-regular partition $\calP$ of ${\bf G}$ satisfies  $|\calP|\geq (1-\e^{1-s})^2|U|/2$.  
 \end{corollary}
 \begin{proof}
Let $U_1,\ldots, U_t$ be an enumeration of the $\sim_G$-classes in $G$. Let $V({\bf G})=\bigsqcup_{x\in U}V_x$ as in Definition \ref{def:blowupgraph}, and for each $1\leq i\leq t$, set ${\bf U}_i=\bigcup_{x\in U_i}V_x$.  By assumption, each $U_i$ has size at most $2$, and, consequently, each ${\bf U}_i$ has size at most $2n$.  Let $\calP'\subseteq \calP$ be as in the conclusion of Lemma \ref{lem:blowup0}. Then we know $|\bigcup_{P\in \calP'}P|\geq (1-\e^{1-s})|V({\bf G})|$ and for all  $P\in \calP'$, there is $1\leq g(P)\leq t$ such that $|P\cap {\bf U}_{g(P)}|\geq (1-\e^{1-s})|P|$. Rearranging, this implies  that for all $P\in \calP'$,
\begin{align*}
|P|&\leq \frac{| P\cap {\bf U}_{g(P)}|}{(1-\e^{1-s})} \leq \frac{|{\bf U}_{g(P)}|}{(1-\e^{1-s})} \leq \frac{2 n}{(1-\e^{1-s})} , 
\end{align*}
where the last inequality uses that  $|{\bf U}_{g(P)}|= |U_{g(P)}|n\leq 2n$. Combining with the lower bound for $|\bigcup_{P\in \calP'}P|$, this yields
\begin{align*}
|\calP|\geq |\calP'|\geq \frac{|\bigcup_{P\in \calP'}P|}{\max_{P\in \calP'}|P|}\geq \frac{ (1-\e^{1-s})|V({\bf G})|}{\max_{P\in \calP'}|P|}&\geq \frac{(1-\e^{1-s})^2|V({\bf G})|}{2n}\\
&=\frac{(1-\e^{1-s})^2|U|}{2},
\end{align*}
where the equality uses that $|V({\bf G})|=|U|n$.
 \end{proof}

\subsection{Almost prime graphs}\label{ss:apgraphs}
In this section, we introduce almost prime graphs, which will play a crucial role in the slow growth rates of Theorem \ref{thm:alljumphom}.  

Distinct vertices $x$, $y$ in a graph $G$ satisfying $x\sim_Gy$ are called \emph{twins} (see Definition \ref{def:sim}).  Graphs containing no twins are called \emph{prime}. Prime graphs and related notions have been studied extensively in the literature.  In \cite{prime}, Chudnovsky, Kim, Oum, and Seymour showed that any sufficiently large prime graph contains one of a short list of special induced prime subgraphs (see also \cite{prime1}).  We will need an easier analogue of this result for a more general class of graphs, which we call \emph{almost prime}. 

\begin{definition}\label{def:irr}
A graph $G$ is \emph{almost prime} if every $\sim_G$-class has size at most $2$. 
\end{definition}

We note that a graph $G$ is prime if all its $\sim_G$ classes have size $1$, and thus, any prime graph is automatically almost prime.  On the other hand, there exist almost prime graphs that are not prime (e.g. a perfect matching).  

We next set notation for a special collection of graphs which will play an important role in our results. 
 
\begin{definition}\label{def:prime}
Given $k\geq 1$, let $\Irr(k)$ be the class of all graphs $G=(V,E)$ where $V$ has size $2k$ and $E$ satisfies one of (a)-(c)  for some enumeration $V=\{x_1,\ldots, x_k,y_1,\ldots, y_k\}$.
\begin{enumerate}[(a)]
\item For all $1\leq i,j\leq k$,  $x_iy_j\in E$ if and only if $i\leq j$,
\item For all $1\leq i,j\leq k$,  $x_iy_j\in E$ if and only if $i= j$,
\item For all $1\leq i,j\leq k$,  $x_iy_j\in E$ if and only if $i\neq j$.
\end{enumerate}
\end{definition}

 Note that in Definition \ref{def:prime}, we have assumed $V$ has size $2k$, which means the elements enumerated in the vertex set $\{x_1,\ldots, x_k,y_1,\ldots, y_k\}$ are pairwise distinct.  The notation ``AP" in Definition \ref{def:prime} is meant to indicate a relationship to almost prime graphs.  Indeed, it is not difficult to show all graphs in $\Irr(k)$ are almost prime. 

\begin{fact}\label{fact:irrcorrect}
For all $k\geq 1$, every $G\in \Irr(k)$ is almost prime.
\end{fact}
\begin{proof}
Let $G=(V,E)\in \Irr(k)$. Then, by definition, $|V|=2k$, and there is an enumeration $V=\{x_1,\ldots, x_k,y_1,\ldots, y_k\}$ such that one of the following holds.
\begin{enumerate}[(a)]
\item For all $1\leq i,j\leq k$,  $x_iy_j\in E$ if and only if $i\leq j$,
\item For all $1\leq i,j\leq k$,  $x_iy_j\in E$ if and only if $i= j$,
\item For all $1\leq i,j\leq k$,  $x_iy_j\in E$ if and only if $i\neq j$.
\end{enumerate}
Let $Z\subseteq V$ have size at least $3$.  By the Pigeonhole Principle, there exists $1\leq i< j\leq k$ such that either $\{x_i,x_j\}\subseteq Z$, or $\{y_i,y_j\}\subseteq Z$.  Let us assume $\{x_i,x_j\}\subseteq Z$ (the other case is similar).  Clearly $y_i,y_j\in V\setminus \{x_i,x_j\}$. Since one of (a)-(c) hold, either $x_iy_i\in E$ and $x_jy_i\notin E$ (cases (a) and (b)),  or $x_jy_i\in E$ and $x_iy_i\notin E$ (case (c)).  In all cases, we have shown that  $x_i\nsim_Gx_j$, and thus $Z$ cannot be contained in a single $\sim_G$-class.  This demonstrates that every $\sim_G$-class has size less than $3$, and thus $G$ is almost prime.
\end{proof}

On the other hand, we will show that any sufficiently large almost prime graph contains an induced copy of some element  from $\Irr(k)$. 

\begin{theorem}\label{thm:prime}
For all integers $k\geq 1$, there is an integer $N\geq 1$ such that the following holds.  Suppose $G$ is an almost prime graph on at least $N$ vertices. Then $G$ contains an element of $\Irr(k)$ as an induced subgraph.
\end{theorem} 

Theorem \ref{thm:prime} is proved via Ramsey theoretic arguments standard in model theory, and similar to those appearing in \cite{prime}. We include the proof in the appendix (see Subsection \ref{ss:ramsey1}).   

\subsection{Constant to polynomial jump}\label{ss:contpolygraphs}

This section proves the existence of the jump from constant to polynomial growth rates for graphs.  This jump is characterized via the \emph{almost prime} graphs appearing in ${\bf B}_{\calH}$  (see  Definitions \ref{def:bhgraphs} and \ref{def:irr}).   In particular, we will show that when ${\bf B}_{\calH}$ contains finitely many almost prime graphs up to isomorphism, $M_{\calH}$ is constant, and when ${\bf B}_{\calH}$ contains infinitely many almost prime graphs up to isomorphism, $M_{\calH}$ is bounded below by a polynomial.

We begin by using Theorem \ref{thm:prime} to show that ${\bf B}_{\calH}$ contains infinitely many non-isomorphic almost prime graphs if and only if  it contains graphs from $\Irr(m)$, for all $m\geq 1$ (see Definition \ref{def:prime} for the definition of $\Irr(m)$).
 
 \begin{corollary}\label{cor:bhchar}
Suppose $\calH$ is a hereditary graph property. The following are equivalent.
\begin{enumerate}
 \item ${\bf B}_{\calH}$ contains infinitely many non-isomorphic almost prime graphs.
 \item For all $m\geq 1$, ${\bf B}_{\calH}\cap \Irr(m)\neq \emptyset$.
 \end{enumerate}
 \end{corollary}
 \begin{proof}
 That (2) implies (1) follows from the fact that every $G\in \Irr(m)$ is almost prime by Fact \ref{fact:irrcorrect}, and has $2m$ vertices by the definition of $\Irr(m)$.   
 
  Assume now (1) holds. Fix $m\geq 1$, and let $N=N(m)$ be from Theorem \ref{thm:prime}.  By (1),  ${\bf B}_{\calH}$ contains an almost prime graph $G$ on at least $N$ vertices. By Theorem \ref{thm:prime}, $G$ contains an induced subgraph $G'$ such that $G'\in \Irr(m)$.   By Fact \ref{fact:bhhgraphs}, ${\bf B}_{\calH}$ is hereditary, so $G'\in {\bf B}_{\calH}$.   Thus (2) holds.
\end{proof}
 
 We now use our lower bound lemma, Lemma \ref{lem:blowup0}, to show certain blowups of graphs from $\Irr(m)$ have polynomial lower bounds on the sizes of their regular partitions.

\begin{lemma}\label{lem:lbblowup0}
Fix $0<t<\frac{1}{2}$ and suppose $\e\in (0,1)$ is sufficiently small compared to $2^{-1/t}$.  Set $m=\lfloor\e^{-1+t}/4\rfloor$, and let $n\geq 1$ be an integer.  Suppose $G\in \Irr(m)$ and ${\bf G}$ is an $n$-blowup of $G$.  Then any $\e$-regular partition of ${\bf G}$ has at least $\e^{-1+2t}$ parts.
\end{lemma}
\begin{proof}
Set $s=1-t$,  fix $G\in \Irr(m)$ and let ${\bf G}$ be an $n$-blowup of $G$. By definition of $\Irr(m)$, $|V(G)|=2m$.  Combining with the definition of $m$, we have
$$
|V(G)|=2m=2\left\lfloor\frac{\e^{-1+t}}{4}\right\rfloor\leq \frac{\e^{-1+t}}{2}=\frac{\e^{-s}}{2}.
$$
After rearranging, this implies $\e\leq (2|V(G)|)^{-1/s}$.  Since we assumed $\e$ is sufficiently small, we may also assume $\e\leq 1/4$.  Consequently, the hypotheses of Corollary \ref{cor:blowup0} are satisfied by $\e$, $s$, $G$, and ${\bf G}$. By Corollary \ref{cor:blowup0}, any $\e$-regular partition of ${\bf G}$ has size at least   
\begin{align*}
\frac{(1-\e^{1-s})^2|V(G)|}{2}=(1-\e^t)^2\lfloor\e^{-1+t}/4\rfloor \geq \e^{-1+2t},
\end{align*}
where the last inequality is because $\e$ is sufficiently small.  
\end{proof}

We now combine Corollary \ref{cor:bhchar} with Lemma \ref{lem:lbblowup0} to prove that if ${\bf B}_{\calH}$ contains infinitely many non-isomorphic almost prime graphs,  then $M_{\calH}$ is bounded below by a polynomial.  

\begin{corollary}\label{cor:bhinf}
Suppose $\calH$ is a hereditary graph property and ${\bf B}_{\calH}$ contains infinitely many non-isomorphic almost prime graphs. Then $M_{\calH}(\e)\geq \e^{-1+o(1)}$.
\end{corollary}
\begin{proof}
Fix $0<t<1/2$, and assume $\e\in (0,1)$ is sufficiently small compared to $2^{-1/t}$.  Set $m=\lfloor\e^{-1+t}/4\rfloor$. By Corollary \ref{cor:bhchar}, there exists some $G\in {\bf B}_{\calH} \cap \Irr(m)$. Fix any integer $n\geq 1$. Since $G\in {\bf B}_{\calH} $, there is some ${\bf G}\in \calH$ so that ${\bf G}$ is an $n$-blowup of $G$.  By Lemma \ref{lem:lbblowup0}, any $\e$-regular partition of ${\bf G}$ requires at least $\e^{-1+2t}$ parts.  

This argument shows that for all $0<\gamma<1$, we have that for all sufficiently small $\e$, $M_{\calH}(\e)\geq \e^{-1+\gamma}$.  In other words, $M_{\calH}(\e)\geq \e^{-1+o(1)}$.
\end{proof}

We now shift gears to  proving that when  ${\bf B}_{\calH}$ contains only finitely many almost prime graphs up to isomorphism, $M_{\calH}(\e)$ is a constant function.  In fact, we will show that in this case,  $M_{\calH}(\e)$ is  equal to the maximum number of $\sim_G$-classes appearing in any $G\in {\bf B}_{\calH}$.  This result will require some preliminaries on almost prime graphs. We first observe the following bounds on the vertex set of an almost prime graph $G$, based on the number of $\sim_G$-classes.     

\begin{observation}\label{ob:sizeAP}
Let $G=(V,E)$ be an almost prime graph, and let $\ell$ be the number of $\sim_G$-classes in $G$. Then $\ell\leq |V|\leq 2\ell$. 
\end{observation}

Observation \ref{ob:sizeAP} is immediate from Definition \ref{def:irr}. Our next lemma says that any graph contains an almost prime induced subgraph with the same number of $\sim$-classes.

\begin{lemma}\label{lem:irrgraphs}
Let $G=(V,E)$ be a graph and let $\ell$ be the number of $\sim_G$-classes in $G$. Then  $G$ contains an almost prime induced subgraph $G'$ with $\ell$-many $\sim_{G'}$-classes.   
\end{lemma} 

 The $G'$ in Lemma \ref{lem:irrgraphs} can be built by simply taking one vertex from each $\sim_G$-class of size $1$, and two vertices from every $\sim_G$-class of size larger than $1$. A formal proof appears in Appendix \ref{ss:ap1}.   

Our final preliminary is more substantial. Specifically, we will prove that the class of graphs with at most a fixed number of $\sim$-classes can be characterized by finitely many forbidden induced subgraphs.  Towards stating this result, we set notation for a special class of graphs.  

\begin{definition}\label{def:fc2}
Let $t\geq 1$ be an integer.  Recalling that $\calG^{(2)}$ denotes the class of finite graphs, define  
\begin{align*}
\calF_t=\{F\in \calG^{(2)}: \text{ $F$  is almost prime, $F$ has at least $t$-many $\sim_F$-classes, }&\\
 \text{ and $|V(F)|\leq 4t$}\}&.
\end{align*}
\end{definition}

Clearly we have that for any integer $t\geq 1$, $\calF_t$ contains only finitely many non-isomorphic  graphs.   We will prove that for any integer $C\geq 1$, the class of graphs with at most $C$-many $\sim$-classes is characterized by omitting the elements of $\calF_{C+1}$.

\begin{theorem}\label{thm:irrgraphs2}
Let $C\geq 1$ be an integer.  Recalling $\calG^{(2)}$ denotes the class of all finite graphs, define
$$
\calH_C=\{G\in \calG^{(2)}: G\text{ has at most $C$-many $\sim_G$-classes}\}.
$$
Then $\calH_C=  \Forb(\calF_{C+1})$, where $\calF_{C+1}$ is from Definition \ref{def:fc2}.  
\end{theorem}

The proof of Theorem \ref{thm:irrgraphs2} is not too difficult, but requires several steps.  We include its proof in Appendix \ref{ss:ap1}. We next use Theorem \ref{thm:irrgraphs2} to show that if ${\bf B}_{\calH}\cap \calF_{C+1}$ is empty, then $M_{\calH}$ is bounded above by $C$.

\begin{theorem}\label{thm:bhfiniteg1}
 For any integer $C\geq 1$ and any hereditary graph property  $\calH$ satisfying ${\bf B}_{\calH}\cap \calF_{C+1}=\emptyset$,  then for all sufficiently small $\e>0$, $M_{\calH}(\e)\leq M_{\calH}^{\hom}(\e^3)\leq C$. 
\end{theorem}
\begin{proof}
 Let $\calH_C=\{G\in \calG^{(2)}: G\text{ has at most $C$-many $\sim_G$-classes}\}$. By Theorem \ref{thm:irrgraphs2}, $\calH_C=\Forb(\calF_{C+1})$. Since  $\calF_{C+1}\cap {\bf B}_{\calH}=\emptyset$,  Theorem \ref{thm:blowupthm} implies $\calH$ is close to $\calH_C$.  

We now show $M^{\hom}_{\calH_C}(\e)\leq C$.  Let $\e$ be sufficiently small compared to $C^{-1}$, and let $G=(V,E)\in \calH_C$ with $|V|$ sufficiently large compared to $\e^{-1}$ and $C$.  Since $G\in \calH_C$, $G$ has at most $C$-many $\sim_{G}$-classes. Say $U_1,\ldots, U_t$ is an enumeration of the $\sim_{G}$-classes in $G$, for some $t\leq C$.  We claim $\calU=\{U_1,\ldots, U_t\}$ is an $\e$-homogeneous partition of $G$.  First, we define the collection of very small parts in this partition. 
$$
\calU_{\sm}=\{U_i\in \calU : |U_i|\leq \e^{-8}\}.
$$
Clearly $|\bigcup_{U_i\in \calU_{\sm}}U_i|\leq \e^{-8}t\leq \e^8|V|$, where the last inequality is because $t\leq C$ and $|V|$ is sufficiently large compared to $\e^{-1}$ and $C$.   Now let
$$
\Sigma_{\bg}=\{(U_i,U_j)\in \calU^2: U_i,U_j\notin \calU_{\sm}\}.
$$
Then 
$$
\left|\bigcup_{(U_i,U_j)\in \Sigma_{\bg}}U_i\times U_j\right|\geq \left|\bigcup_{(U_i,U_j)\in \calU^2}U_i\times U_j\right|-2\left|V\right|\left|\bigcup_{U_i\in \calU_{\sm}}U_i\right|\geq |V|^2-2\e^8|V|^2>(1-\e)|V|^2,
$$
 where the last inequality is because $\e$ is sufficiently small. Thus, to show $\calU$ is $\e$-homogeneous with respect to $G$, it suffices to show every $(U_i,U_j)\in \Sigma_{\bg}$ is $\e$-homogeneous with respect to $G$.  
 
Fix $(U_i,U_j)\in \Sigma_{\bg}$.  Since $U_i$ and $U_j$ are $\sim_{G}$ classes, we have that either $i\neq j$ and $d_{G}(U_i,U_j)\in \{0,1\}$, or $i=j$ and $d_{G}(U_i,U_j)\in \{0, 1-\frac{1}{|U_i|}\}$ (recall no vertex can be connected to itself).  As $U_i\notin \calU_{\sm}$, we have in all these cases that $d_{G}(U_i,U_j)\in \{0\}\cup (1-\e^8,1]$.  This implies $(U_i,U_j)$ is $\e$-homogeneous with respect to $G$, as desired.  This completes our verification that $M^{\hom}_{\calH_C}(\e)\leq C$.  

Since $\calH$ is close to $\calH_C$, Proposition \ref{prop:closehomgraphs} and $M^{\hom}_{\calH_C}(\e)\leq C$  imply $M_{\calH}^{\hom}(2\e)\leq C$. By Fact \ref{fact:homub}, we can conclude that for all sufficiently small $\e$, $M_{\calH}(\e)\leq M_{\calH}^{\hom}(\e^3)\leq C$.  
\end{proof}

We next show that $M_{\calH}(\e)$ can be lower bounded by the number of $\sim_G$-classes appearing in an  almost prime element $G$ of ${\bf B}_{\calH}$.  

\begin{theorem}\label{thm:bhfinite2g}
Suppose $C\geq 1$ is an integer and $G$ is an almost prime graph with $C$-many $\sim_G$-classes.  Then for any hereditary graph property  $\calH$ with $G\in {\bf B}_{\calH}$,  $M_{\calH}(\e)\geq C$. 
\end{theorem}
\begin{proof}
 Let $\e$ be sufficiently small compared to $C^{-1}$ and let $n$ be sufficiently large compared to $\e^{-1}$. Let $U_1,\ldots, U_C$ enumerate the $\sim_G$-classes of $G$. By assumption,  there is some ${\bf G}\in \calH$ which is an $n$-blowup of $G$.  Say ${\bf G}$ has vertex set $\bigcup_{x\in V(G)}V_x$ as in Definition \ref{def:blowupgraph}.  For each $1\leq i\leq C$, let ${\bf U}_i=\bigcup_{x\in U_i}V_x$. Note $|V(G)|\leq 2C$ by Observation \ref{ob:sizeAP}. 

Since $\e$ is sufficiently small compared to $C^{-1}$ and $|V(G)|\leq 2C$, we may assume that $\e<\min\{1/4,(2|V(G)|)^{-1/2}\}$.  Consequently, we may apply Lemma \ref{lem:blowup0} to $G$ and ${\bf G}$ with parameters $s=1/2$ and $\e$.  Let $\calP$ be any $\e$-regular partition of ${\bf G}$. By Lemma \ref{lem:blowup0} there exists $\calP'\subseteq \calP$ satisfying $|\bigcup_{P\in \calP'}P|\geq (1-\e^{1/2})|V({\bf G})|$,   such that for all $P\in \calP'$, there is $1\leq g(P)\leq C$ such that $|P\cap {\bf U}_{g(P)}|\geq (1-\e^{1/2})|P|$.  We then have that  
\begin{align*}
n\left|\bigcup_{i\in [C]\setminus \{g(P): P\in \calP'\}}U_i\right|&= \left|\bigcup_{i\in [C]\setminus \{g(P): P\in \calP'\}}{\bf U}_i\right|= \Bigg|V({\bf G})\setminus \Bigg(\bigcup_{P\in \calP'}{\bf U}_{g(P)}\Bigg)\Bigg|.
\end{align*}
This is at most
\begin{align*}
  \Bigg|V({\bf G})\setminus \Bigg(\bigcup_{P\in \calP'}P\Bigg)\Bigg|+\Bigg|\Bigg(\bigcup_{P\in \calP'}P\Bigg)\setminus \Bigg(\bigcup_{P\in \calP'}{\bf U}_{g(P)}\Bigg)\Bigg|&\leq \sqrt{\e}|V({\bf G})|+\sum_{P\in \calP'}|P\setminus {\bf U}_{g(P)}|\\
&\leq \sqrt{\e}|V({\bf G})|+\sum_{P\in \calP'}\sqrt{\e}|P|\\
&\leq 2\sqrt{\e}|V({\bf G})| \\
&= 2\sqrt{\e}|V(G)|n.
\end{align*}
This shows $n|\bigcup_{i\in [C]\setminus \{g(P): P\in \calP'\}}U_i|\leq 2\sqrt{\e}|V(G)|n$. After canceling the $n$ and rearranging, we have
\begin{align}\label{al:cbound}
\left|\bigcup_{i\in [C]\setminus \{g(P): P\in \calP'\}}U_i\right|\leq 2\sqrt{\e}|V(G)|\leq 4\sqrt{\e}C<1,
\end{align}
where the second   inequality uses that $|V(G)|\leq 2C$ (since $G$ is almost prime), and the last inequality is because $\e$ is sufficiently small compared to $C^{-1}$.  Since the first term appearing in (\ref{al:cbound}) is a non-negative integer, it must be $0$. The only way this is possible is if $[C]\setminus \{g(P): P\in \calP'\}=\emptyset$. Thus we have $\{g(P): P\in \calP'\}=[C]$, and consequently $|\calP|\geq  |\{g(P): P\in \calP'\}|=C$.
\end{proof}

We can now prove that if ${\bf B}_{\calH} $ contains finitely many almost prime graphs up to isomorphism, then $M_{\calH}$ is a constant function.  Moreover, we compute said constant explicitly, showing it is equal to the maximum number of  distinct $\sim_G$-classes appearing in any element $G$ of ${\bf B}_{\calH}$.

\begin{lemma}\label{lem:finitegraph}
Suppose $\calH$ is a hereditary graph property and ${\bf B}_{\calH}$ contains finitely many almost prime graphs up to isomorphism.  Then there exists an integer $C>0$ such that 
\begin{align*}
C=M_{\calH}(\e)=M^{\hom}_{\calH}(\e)&=\max\{\ell\in \mathbb{N}^{\geq 1}: \text{there exists $G\in {\bf B}_{\calH}$ with $\ell$-many $\sim_G$-classes}\}\\
&=\max\{\ell \in \mathbb{N}^{\geq 1}: {\bf B}_{\calH}\cap \calF_{\ell}\neq \emptyset\},
\end{align*}
where the $\calF_{\ell}$ are from Definition \ref{def:fc2}.
\end{lemma}
\begin{proof}
We will consider the following three sets of natural numbers. 
\begin{align*}
A_1&=\{\ell\in \mathbb{N}^{\geq 1}: \text{ there is an almost prime }H\in {\bf B}_{\calH} \text{ with $\ell$-many $\sim_H$-classes}\}.\\
A_2&=\{\ell\in \mathbb{N}^{\geq 1}: \text{ there is some }H\in {\bf B}_{\calH} \text{ with $\ell$-many $\sim_H$-classes}\}.\\
A_3&=\{\ell \in \mathbb{N}^{\geq 1}: {\bf B}_{\calH}\cap \calF_{\ell}\neq \emptyset\}.
\end{align*}
Since ${\bf B}_{\calH}$ is hereditary (by Fact \ref{fact:bhhgraphs}), Lemma \ref{lem:irrgraphs} implies $A_1=A_2$.  Fact \ref{fact:bhhgraphs} also implies ${\bf B}_{\calH}$ is nonempty, and consequently, $A_2=A_1$ is nonempty.  By assumption, $A_1$ is finite.  Combining these observations, we see that the integer $C$ below is well defined.
\begin{align*} 
C=\max A_1=\max A_2.  
\end{align*}
Given an integer $t\geq 1$, all elements of $\calF_t$ have at least $t$-many $\sim$-classes. Since $C=\max A_2$, this implies ${\bf B}_{\calH}\cap \calF_t=\emptyset$ for all $t>C$, and thus, $A_3\subseteq \{1,\ldots, C\}$.   By definition, $\calF_1$ contains the trivial graph $G_{\text{triv}}$ with one vertex. Since ${\bf B}_{\calH}$ is hereditary, we have $G_{\text{triv}}\in \calF_1\cap {\bf B}_{\calH}$, and thus, $1\in A_3$.  These observations tell us  the integer $L=\max A_3$ is well defined and satisfies $1\leq L\leq C$. We next show $L=C$. If $C=1$, then this is immediate from $1\leq L\leq C$. If $C>1$, then $C=\max A_2$ tells us ${\bf B}_{\calH}\nsubseteq \calH_{C-1}$, where $\calH_{C-1}$ is as in Theorem \ref{thm:irrgraphs2}. Since Theorem \ref{thm:irrgraphs2} tells us $\calH_{C-1}=\Forb(\calF_C)$ and ${\bf B}_{\calH}$ is hereditary, this implies ${\bf B}_{\calH}\cap \calF_C \neq \emptyset$. This yields that $C\leq L$, and thus $C=L$.   
 
 By Theorem \ref{thm:bhfinite2g} and since $C=\max A_1$,  $M_{\calH}(\e)\geq C$.  By Theorem \ref{thm:bhfiniteg1} and since $C=\max A_3$, we have $M_{\calH}(\e)\leq M_{\calH}^{\hom}(\e^3)\leq C$.  Combining everything together, we have shown that for all sufficiently small $\e>0$, 
$$
M_{\calH}(\e)=M_{\calH}^{\hom}(\e)=C=\max A_1=\max A_2=\max A_3.
$$
  \end{proof}
  
  Corollary \ref{cor:bhinf} and Lemma \ref{lem:finitegraph} immediately imply the existence of a jump from constant to polynomial growth.   
  
  \begin{corollary}\label{cor:constpoly2}
For any hereditary  graph property $\calH$, one of the following holds.  
\begin{enumerate}
\item ${\bf B}_{\calH}$ contains arbitrarily large almost prime graphs. In this case $M_{\calH}(\e)\geq \Omega(\e^{-1+o(1)})$.
\item ${\bf B}_{\calH}$ contains  finitely many non-isomorphic almost prime graphs. In this case, $M_{\calH}(\e)=M_{\calH}^{\hom}(\e)=C$, where $C\in \mathbb{N}^{\geq 1}$ is equal to the maximum number of $\sim$-classes appearing in any element of ${\bf B}_{\calH}$. 
\end{enumerate}
\end{corollary}

We now  prove our main result about $M_{\calH}$ in the graphs case, Theorem \ref{thm:alljumphom}.

\vspace{2mm}

\noindent{\bf Proof of Theorem \ref{thm:alljumphom}.}
Suppose first $\calH$ has infinite VC-dimension. Then by Corollary \ref{cor:fastgraphs}(1),  $\Tw(\Omega(\e^{-2}))\leq M_{\calH}(\e)\leq \Tw(O(\e^{-4}))$, so (1) holds.

Suppose now $\calH$ has finite VC-dimension, say $\VC(\calH)=k$. If ${\bf B}_{\calH}$ contains infinitely many non-isomorphic almost prime graphs, then Corollaries \ref{cor:fastgraphs}(2) and \ref{cor:bhinf} imply there is a constant $C=C(k)>0$ such that
$$
\e^{-1+o(1)}\leq M_{\calH}(\e)\leq M_{\calH}^{\hom}(\e^3)\leq C\e^{-6k-3},
$$
 so (2) holds.

We are left with the case where $\calH$ has finite VC-dimension and ${\bf B}_{\calH}$ contains only finitely many non-isomorphic almost prime graphs.  By Lemma \ref{lem:finitegraph}, there is a constant $C>0$ such that $M_{\calH}(\e)=M_{\calH}^{\hom}(\e)=C$, so (3) holds.
\qed

\section{Background on VC-dimension in $3$-graphs}\label{sec:weakreg}

In this section, we cover two versions of VC-dimension for $3$-graphs and their connections to homogeneous partitions.  We begin by defining what is simply called the VC-dimension of a $3$-graph, which is the natural analogue of Definition \ref{def:vc}.  

\begin{definition}\label{def:vchg}
Suppose $H=(V,E)$ is a $3$-graph.  The \emph{VC-dimension of $H$} is the VC-dimension of the set system $(V,\calF)$, where 
$$
\calF=\left\{N_H(uv): uv\in {V\choose 2}\right\}.
$$ 
\end{definition}

Definition \ref{def:vchg}  extends to hereditary $3$-graph properties in the natural way.

\begin{definition}\label{def:vc3graph}
 For a hereditary $3$-graph property $\calH$, the \emph{VC-dimension of $\calH$} is 
$$
\VC(\calH)=\sup\{\VC(H):H\in \calH\}\in \mathbb{N}\cup \{\infty\}.
$$
When $\VC(\calH)<\infty$, we say $\calH$ \emph{has finite VC-dimension}, and when $\VC(\calH)=\infty$, we say $\calH$ \emph{has infinite VC-dimension}.
\end{definition}

We next define a special collection of finite $3$-graphs which we will use to characterize when $\VC(\calH)=\infty$.
 
\begin{definition}\label{def:ukhat} Given $k\geq 1$, let $\widehat{\PS(k)}$ denote the class of all $3$-graphs $H$ such that the vertex set $V(H)$ admits an indexing (possibly with repetitions)  of the form
$$
V(H)=\{a_i,b_i: i\in [k]\}\cup \{c_S: S\subseteq [k]\},
$$
 such that $\{a_i,b_i: i\in [k]\}\cap \{c_S: S\subseteq [k]\}=\emptyset$, and such that the edge set $E(H)$ satisfies
$$
\{a_ib_ic_S: S\subseteq [k], i\in S \}\subseteq E(H)\text{ and }  \{a_ib_ic_S : S\subseteq [k], i\in [k]\setminus S\}\cap E(H)=\emptyset.
$$  
\end{definition}

Note that Definition \ref{def:ukhat} says nothing about triples in ${V(H)\choose 3}$ not of the form $a_ib_ic_S$ for some $i\in [k]$ and $S\subseteq [k]$.  This means $\widehat{\PS(k)}$ will contain many non-isomorphic $3$-graphs.  On the other hand, every element in $\widehat{\PS(k)}$ has at most $2k+2^k$ vertices, so $\widehat{\PS(k)}$ contains only finitely many $3$-graphs up to isomorphism. Observe that in Definition \ref{def:ukhat}, the two sets $\{a_i,b_i:i\in [k]\}$ and $\{c_S:S\subseteq [k]\}$ are assumed to be disjoint.  This disjointness requirement will be convenient in later proofs.  We will use the following characterization of hereditary $3$-graph properties with infinite VC-dimension, stated in terms of $\widehat{\PS(k)}$. 

\begin{proposition}\label{prop:equiv22}
For any hereditary $3$-graph property $\calH$, the following are equivalent.
\begin{enumerate}
\item $\calH$ has infinite VC-dimension.
\item For all  integers $k\geq 1$, $ \calH  \cap \widehat{\PS(k)}\neq \emptyset$. 
\end{enumerate}
\end{proposition}

The proof of Proposition \ref{prop:equiv22} consists of standard exercises involving the Sauer-Shelah lemma and is included in the appendix (see Appendix \ref{ss:appclose}).

 As mentioned in Section \ref{sec:graphs}, work of Alon--Fischer--Newman and Lov\'{a}sz--Szegedy first showed that graphs of bounded VC-dimension have homogeneous partitions  with extremely efficient bounds. Extensions of these results to hypergraphs were later proved independently by Chernikov--Starchenko \cite{Chernikov.2016zb} and  Fox--Pach--Suk \cite{Fox.2017bfo}.\footnote{The theorems proved in both \cite{Chernikov.2016zb} and \cite{Fox.2017bfo} are stronger and more detailed than what is stated in Theorem \ref{thm:foxpachsuk}.  In particular, both results apply to $r$-graphs for any $r\geq 2$, and give  explicit expressions for $C$ in terms of the dual VC-dimension of the hypergraph.  The theorem in \cite{Fox.2017bfo} obtains equipartitions and a better bound for the constant $C$ than \cite{Chernikov.2016zb}.  On the other hand, the partition obtained in \cite{Chernikov.2016zb}, while not necessarily equitable,  is definable (a property of central interest to model theorists).}

\begin{theorem}[Chernikov--Starchenko \cite{Chernikov.2016zb}, Fox--Pach--Suk \cite{Fox.2017bfo}]\label{thm:foxpachsuk}
For all integers $k\geq 0$ there exists $C=C(k)>0$ so that the following holds.  Suppose $H$ is a $3$-graph with VC-dimension at most $k$.  Then $H$ has an $\e$-homogeneous partition of size at most $\e^{-C}$.
\end{theorem}

 Theorem \ref{thm:foxpachsuk} immediately implies that if $\VC(\calH)<\infty$, then $M_{\calH}^{\hom}(\e)$ is bounded above by a polynomial in $\e^{-1}$.   In light of Proposition \ref{prop:closehom}, this will also hold for any $\calH$ which is close to a property $\calH'$ with $\VC(\calH')<\infty$. This motivates the following definition.

\begin{definition}\label{def:clvc}
Suppose $\calH$ is a hereditary $3$-graph property.  We say $\calH$ is \emph{close to finite VC-dimension} if $\calH$ is close to some $\calH'$ with $\VC(\calH')<\infty$.  Otherwise, we say $\calH$ is \emph{far from finite VC-dimension}.
\end{definition}

In contrast to the graphs setting, a hereditary $3$-graph property may be close to finite VC-dimension without actually having finite VC-dimension.  For this reason, we require a combinatorial characterization of when Definition \ref{def:clvc} holds.  To our knowledge, no such result appears explicitly in the literature, so we will provide one here.   

\begin{proposition}\label{prop:equiv2}
For any hereditary $3$-graph property $\calH$, the following are equivalent.
\begin{enumerate}
\item $\calH$ is close to finite VC-dimension.
\item There exists an  integer $k\geq 1$ such that ${\bf B}_{ \calH  }\cap \widehat{\PS(k)}=\emptyset$.   
\item ${\bf B}_{\calH}$ has finite VC-dimension.
\end{enumerate}
\end{proposition}
\begin{proof}
That (2) and (3) are equivalent follows immediately from Proposition \ref{prop:equiv22} (and implicitly Fact \ref{fact:bhhgraphs}).  We now show (1) and (2) are equivalent.  Suppose first (1) holds. Then there is some hereditary $3$-graph property $\calH'$ with $\VC(\calH')<\infty$ such that $\calH$ is close to $\calH'$.  By Proposition \ref{prop:equiv22}, there is $k\geq 1$ such that $\calH'\cap \widehat{\PS(k)}=\emptyset$.  By Theorem \ref{thm:blowupthm}, $\widehat{\PS(k)}\cap {\bf B}_{\calH}=\emptyset$, so (2) holds. 

 Assume now (2) holds. Then there is some $k\geq 1$ such that ${\bf B}_{\calH}\cap \widehat{\PS(k)}=\emptyset$.  By Theorem \ref{thm:blowupthm}, $\calH$ is close to $\Forb(\widehat{\PS(k)})$. By Proposition \ref{prop:equiv22}, $\Forb(\widehat{\PS(k)})$ has finite VC-dimension, so by definition, $\calH$ is close to finite VC-dimension.\end{proof}
 
 The reader may naturally wonder at this stage  if the three conditions appearing in Proposition \ref{prop:equiv2} are also equivalent to $M_{\calH}^{\hom}<\infty$.  It turns out that this is not the case. The existence of homogeneous partitions can be extended to a wider class of $3$-graphs defined in terms of a different generalization of VC-dimension called ``slicewise VC-dimension."  Slicewise VC-dimension is defined in terms of the following notion of a ``slice graph," also called a ``link graph" in the literature.

\begin{definition}\label{def:slicegraph}
Suppose $H=(V,E)$ is a $3$-graph with at least $2$ vertices.  For each $x\in V$,  the \emph{slice graph of $H$ at $x$}, denoted $H_x$, is the graph with vertex set $V\setminus \{x\}$ and edge set $\{yz\in {V\setminus \{x\}\choose 2}: xyz\in E\}$.
\end{definition} 

We now give an analogue of VC-dimension for $3$-graphs, distinct from Definition \ref{def:vc3graph}, which is defined in terms of slice graphs.

\begin{definition}\label{def:slvc}
If $H=(V,E)$ is a $3$-graph with at least $2$ vertices, then the \emph{slicewise VC-dimension (SVC-dimension) of $H$ } is 
$$
\SVC(H):=\max\{\VC(H_x): x\in V\}.
$$
If $H$ is a $3$-graph with only $1$ vertex, set $\SVC(H)=0$ by convention.  
\end{definition}

 It is an easy exercise to see that for any $3$-graph $H$, $\SVC(H)\leq \VC(H)$.  On the other hand, there exist easily defined $3$-graphs with arbitrarily large VC-dimension and SVC-dimension 1 (see, e.g., the examples used to prove Proposition 2.28 in \cite{Terry.2021b}).  Thus, bounding the SVC-dimension of a $3$-graph is a genuinely weaker assumption than bounding its VC-dimension.

We extend Definition \ref{def:slvc} to hereditary properties in analogy to Definition \ref{def:vc3graph}.

\begin{definition}\label{def:slvchp}
For a hereditary $3$-graph property $\calH$, the \emph{slicewise VC-dimension of $\calH$} is 
 $$
 \SVC(\calH)=\sup\{\SVC(H):H\in \calH\}\in \mathbb{N}\cup \{\infty\}.
 $$
 We will also refer to  $\SVC(\calH)$ as the \emph{SVC-dimension of $\calH$}.  When $\SVC(\calH)<\infty$, we say $\calH$ \emph{has finite SVC-dimension}, and when $\SVC(\calH)=\infty$, we say $\calH$ \emph{has infinite SVC-dimension}.
\end{definition}

Work of Wolf and the author \cite{Terry.2021b} and, independently, Chernikov and Towsner \cite{Chernikov.2020}, showed that $3$-graphs with bounded SVC-dimension admit homogeneous partitions.  

\begin{theorem}\label{thm:slvcwow}
For all integers $k\geq 0$ and $\e\in (0,1)$, there is an integer $M=M(\e,k)\geq 1$ so that any $3$-graph $H$ with $\SVC(H)\leq k$ admits an $\e$-homogeneous partition with at most $M$ parts.
\end{theorem}

Theorem \ref{thm:slvcwow} immediately implies that $M_{\calH}^{\hom}<\infty$ whenever $\SVC(\calH)<\infty$.  However, the original proofs of this theorem did not produce efficient bounds on $M_{\calH}^{\hom}$. The proof by Chernikov and Towsner is non-quantitative, producing no explicit bounds.  The proof by the author and Wolf  is quantitative, but produces a wowzer-style bound for $M(\e,k)$ due to an application of a strong version of the hypergraph regularity lemma. A crucial ingredient in the current paper is the following result from part 1 \cite{Terry.2024a},\footnote{Definitions \ref{def:slicegraph} and \ref{def:slvc} differ slightly from those used in \cite{Terry.2024a}, however it is easy to check the $\SVC(H)$ in Definition \ref{def:slvc}  differs by at most $1$ compared to the $\SVC(H)$ defined in \cite{Terry.2024a}.} which gives a quantitative improvement on Theorem \ref{thm:slvcwow}.\footnote{See \cite{GSW} for a further  improvement by Gishboliner, Shapira, and Wigderson.}

\begin{theorem}[Theorem 1.4 of \cite{Terry.2024a}]\label{thm:slvc}
For all integers $k\geq 0$ there exists $C=C(k)>0$ so that for all sufficiently small $\e>0$, the following holds. Suppose $H$ is a sufficiently large $3$-graph with $\SVC(H)\leq k$. Then there exists an $\e$-homogeneous partition of $H$ of size at most $2^{2^{\e^{-C}}}$.
\end{theorem}

Theorem \ref{thm:slvc} implies that if $\calH$ has finite SVC-dimension, then $M_{\calH}^{\hom}$ is bounded above by a double exponential.  With Proposition \ref{prop:closehom} in mind, one can foresee that  the same conclusion will hold for any hereditary $3$-graph property which is  \emph{close to finite SVC-dimension}  in the following sense.  

\begin{definition}\label{def:clsvc}
Suppose $\calH$ is a hereditary $3$-graph property.  We say $\calH$ is \emph{close to finite SVC-dimension} if $\calH$ is close to some $\calH'$ with $\SVC(\calH')<\infty$.  Otherwise, we say $\calH$ is \emph{far from finite SVC-dimension}.
\end{definition}

It turns out that being close to finite SVC-dimension is necessary and sufficient to guarantee the existence of homogeneous partitions, a fact which was proved in \cite{Terry.2021b}.\footnote{Related results  in a different formalism were proved independently in \cite{Chernikov.2020}.}\footnote{Definitions \ref{def:slicegraph} and \ref{def:slvc} differ slightly from those used in \cite{Terry.2021b}. A straightforward exercise using the Sauer-Shelah lemma shows the $\SVC(H)$ from \cite{Terry.2021b} can be bounded above and below as a function of the $\SVC(H)$ in Definition \ref{def:slvc}. The result we will use from \cite{Terry.2021b} depends only on whether $\SVC(\calH)$ is finite or infinite, and is therefore unaffected by these differences.}  We will need this fact, in addition to combinatorial characterizations of when a hereditary property is close to finite SVC-dimension. Before stating these results, we must set some notation. 

We first define special sets of $3$-graphs, built by adjoining $n$ new vertices to a graph $G$ in a certain way.

\begin{definition}\label{def:otimes}
Suppose $G=(V, E)$ is a graph and $n\geq 1$ is an integer. Define $n\otimes G$ to be the class of $3$-graphs $H$ with vertex set of the form $V(H)=V\sqcup C$, where $C$ is a set of $n$ new vertices, and with edge set $E(H)$ satisfying the following.
\begin{align*}
 \left\{uwc: uw\in E, c\in C \right\}\subseteq E(H)\text{ and } \left\{uwc: uw\in {V\choose 2}\setminus E, c\in C\right\}\cap E(H)=\emptyset. 
\end{align*}
\end{definition}

Note that for any $3$-graph $H$ and $x\in V(H)$, $H\in 1\otimes H_x$.  On the other hand, if $H\in 1\otimes G$, then there must be some vertex $x\in V(H)$ such that $H_x\cong G$.  In light of these remarks, we have the following.

\begin{fact}\label{fact:svcbip}
Suppose $\calH$ is a hereditary $3$-graph property. Then 
$$
\SVC(\calH)=\sup\{\VC(G): (1\otimes G)\cap \calH\neq \emptyset\}.
$$
\end{fact}

Combining Fact \ref{fact:svcbip} with Fact \ref{fact:bipvc} yields several criteria equivalent to $\SVC(\calH)=\infty$.  One such criterion we will use is stated in terms of the following collection of special $3$-graphs.

\begin{definition}\label{def:psotimes}
Given $k\geq 1$, let $1\otimes \PS(k)=\bigcup_{G\in \PS(k)}1\otimes G$.
\end{definition}

\begin{fact}\label{fact:svcbip2}
Suppose $\calH$ is a hereditary $3$-graph property. Then the following are equivalent.
\begin{enumerate}
\item $\SVC(\calH)=\infty$.
\item For all $k\geq 1$, $\calH\cap (1\otimes \PS(k))\neq \emptyset$.
\end{enumerate}
\end{fact}

We provide a proof of Fact \ref{fact:svcbip2} from Facts \ref{fact:bipvc}  and \ref{fact:svcbip} in the appendix.  We end this section by stating the necessary equivalences for when a hereditary $3$-graph property $\calH$ is far from finite SVC-dimension.  This result is essentially Theorem 2.34 from \cite{Terry.2021b}, altered to suit the applications in this paper. 

\begin{theorem}\label{thm:vdischom}
Suppose $\calH$ is a hereditary $3$-graph property.  The following are equivalent.
\begin{enumerate}
\item For some $\e\in (0,1)$, $M_{\calH}^{\hom}(\e)=\infty$.\footnote{This property is called $\vdisc_3$-homogeneity in \cite{Terry.2021b}.}
\item $\calH$ is far from finite $\SVC$-dimension. 
\item ${\bf B}_{\calH}$ has infinite $\SVC$-dimension.
\item For all $k\geq 1$, ${\bf B}_{\calH}\cap (1\otimes \PS(k))\neq \emptyset$.
\item One of the following holds.
\begin{enumerate}
\item For all $n\geq 1$ and every bipartite graph $G$, $(n\otimes G)\cap  {\bf B}_{\calH}\neq \emptyset$. 
\item For all $n\geq 1$ and every co-bipartite graph $G$, $(n\otimes G)\cap  {\bf B}_{\calH}\neq \emptyset$.
\item For all $n\geq 1$ and every split graph $G$, $(n\otimes G)\cap  {\bf B}_{\calH}\neq \emptyset$.
\end{enumerate}
\end{enumerate}
\end{theorem}

The equivalence of (1) and (2) follows from Theorem 2.34 in \cite{Terry.2021b}.  As the remaining equivalences are slightly different from those appearing in \cite{Terry.2021b}, we provide a proof of Theorem \ref{thm:vdischom} in the appendix (see Appendix \ref{ss:appclose}).

\section{Jump from Almost Exponential to Tower}\label{sec:exptower}

In this section, we prove that there exists a jump from almost exponential to tower growth for hereditary $3$-graph properties.  Before doing so, we require two lemmas.  First, we will use the fact that the common refinement of a regular partition with a partition of size at most $2$ will still be fairly regular. 

\begin{lemma}\label{lem:slicingcor}
Let $H=(V,E)$ be a $3$-graph, and let $\calP$ be an $\e$-regular partition of $H$.  Suppose  $\calP'$ is a partition of $V$ with at most $2$ parts, and let $\calQ$ be the common refinement of $\calP$ with $\calP'$.   Then $|\bigcup_{(X,Y,Z)\in \Sigma}X\times Y\times Z|\geq (1-7\e^{1/2})|V|^3$, where 
$$
\Sigma=\{(X,Y,Z)\in \calQ^3: (X,Y,Z)\text{ is $2\e^{1/2}$-regular with respect to $H$}\}.
$$ 
\end{lemma}

Lemma \ref{lem:slicingcor} is a standard fact (see Lemma 6.7 in \cite{Fox.2014} for a similar statement in the graphs setting).  We provide a proof in Section \ref{ss:standardlemmas} of the appendix for the sake of completeness.  

Our next lemma shows that given a graph $G$, there are certain values of $n$ such that a regular partition of a $3$-graph from $n\otimes G$ induces a regular partition of the graph $G$  (see Definition \ref{def:otimes}).

\begin{lemma}\label{lem:timeslem}
Let $\e\in (0,1)$ be sufficiently small, and let $n\geq 1$ be an  integer. Suppose $G$ is a graph with $|V(G)|= 2n$ and  $H$ is a $3$-graph satisfying $H\in n\otimes G$.  Then if $H$ has an $\e$-regular partition of size $t$, $G$ has a  $48\e^{1/2}$-regular partition of size at most $t$. 
\end{lemma}
\begin{proof}
By Definition \ref{def:otimes}, we may assume $H$ has vertex set $V(G)\sqcup C$, where $C$ is a set of $n$ vertices, and edge set $E(H)$ satisfying
\begin{align}\label{al:timeslem1}
\{uwc: uw\in E(G), c\in C\}\subseteq E(H)\text{ and }\left\{uwc: uw\in {V(G)\choose 2}\setminus E(G), c\in C\right\}\cap E(H)=\emptyset.
\end{align}
 Note that our assumptions imply $|V(G)|=2n$ and $|V(H)|=3n$. Suppose $\calP=\{V_1,\ldots, V_t\}$ is an $\e$-regular partition of $H$ of size $t$.  Let $\calP'$ be the common refinement of $\calP$ with the partition $V(H)=V(G)\sqcup C$, and define
$$
\calQ=\{X\in \calP': X\subseteq V(G)\}.
$$
Clearly $\calQ$ is a partition of $V(G)$ and $|\calQ|\leq |\calP|=t$. We show that $\calQ$ is $48\e^{1/2}$-regular with respect to $G$.  We first observe that by Lemma \ref{lem:slicingcor}, if we set 
\begin{align*}
\Sigma_{\reg}&=\{(X,Y,Z)\in (\calP')^3: (X,Y,Z)\text{ is $2\e^{1/2}$-regular with respect to $H$}\}\text{ and }\\
\calE_{\reg}&=\bigcup_{(X,Y,Z)\in \Sigma_{\reg}}X\times Y\times Z,
\end{align*}
then $|\calE_{\reg}|\geq (1-7\e^{1/2})|V(H)|^3$. Consequently, 
\begin{align}\label{al:sreg}
\nonumber|\calE_{\reg}\cap (V(G)\times V(G)\times C)|\geq |V(G)|^2|C|-7\e^{1/2} |V(H)|^3&= \Bigg(1-\frac{7\cdot 27\e^{1/2}}{4}\Bigg)|V(G)|^2|C|\\
&>(1-48\e^{1/2})|V(G)|^2|C|,
\end{align}
where the equality uses that $|V(G)|=2n$, $|C|=n$, and $|V(H)|=3n$.    Since $C$ is a disjoint union of elements from $\calP'$, (\ref{al:sreg}) implies there exists some nonempty $Z\in \calP'$ such that $Z\subseteq C$ and such that
\begin{align}\label{al:timeslem2}
|\calE_{\reg}\cap (V(G)\times V(G)\times Z)|\geq (1-48\e^{1/2})|V(G)|^2|Z|.
\end{align}
Set $\Theta=\{(X,Y)\in \calQ^2: X\times Y\times Z \subseteq \calE_{\reg}  \}$.  By definition of $\Theta$ and (\ref{al:timeslem2}), we have 
\begin{align*} 
\left|\bigcup_{(X,Y)\in \Theta}X\times Y\right||Z|= |\calE_{\reg}\cap (V(G)\times V(G)\times Z)|\geq (1-48\e^{1/2})|V(G)|^2|Z|.
\end{align*}
Dividing by $|Z|$ yields that $|\bigcup_{(X,Y)\in \Theta}X\times Y | \geq (1-48\e^{1/2})|V(G)|^2$. It thus suffices to show every $(X,Y)\in \Theta$ is $48\e^{1/2}$-regular with respect to $G$.   We will show something stronger, namely that every $(X,Y)\in \Theta$ is $2\e^{1/2}$-regular with respect to $G$.

Fix $(X,Y)\in \Theta$.  Let  $X'\subseteq X$ and $Y'\subseteq Y$ satisfy $|X'|\geq 2\e^{1/2}|X|$ and $|Y'|\geq 2\e^{1/2}|Y|$. By definition of $\Theta$, $(X,Y,Z)\in \Sigma_{\reg}$, and consequently,
\begin{align}\label{al:dclose}
|d_H(X,Y,Z)-d_H(X',Y',Z)|\leq 2\e^{1/2}.
\end{align}
By (\ref{al:timeslem1}), and since $X,Y\subseteq V(G)$ while $Z\subseteq C$, we have 
\begin{align*}
d_G(X,Y)&=\frac{|\overline{E(G)}\cap (X\times Y)|}{|X||Y|}=\frac{|\overline{E(H)}\cap (X\times Y\times Z)|}{|X||Y||Z|}=d_H(X,Y,Z).
\end{align*}
 The same computation using $X',Y'$ in place of $X,Y$ shows $d_G(X',Y')=d_H(X',Y',Z)$.  Thus (\ref{al:dclose}) implies $|d_G(X,Y)-d_G(X',Y')|\leq 2\e^{1/2}$. This completes the proof.
\end{proof}

We can now prove a sufficient condition for $M_{\calH}$ to be bounded above and below by a tower.

\begin{theorem}\label{thm:fastest}
Suppose $\calH$ is a hereditary $3$-graph property and $\calH$ is far from finite SVC-dimension.  Then $\Tw(\Omega(\e^{-1}))\leq M_{\calH}(\e)\leq \Tw(6\e^{-4})$. 
\end{theorem}
\begin{proof}
Suppose $\calH$ is a hereditary $3$-graph property which is far from finite SVC-dimension. The upper bound for $M_{\calH}(\e)$ holds by Theorem \ref{thm:chung} (see also the remark following the statement of Theorem \ref{thm:chung}).   

For the lower bound, let $c$ be as in Corollary \ref{cor:bipgowers}, let $\e>0$ be sufficiently small, and let $N$ be sufficiently large.  Let $G_1=(U\cup W, E_1)$, $G_2=(U\cup W, E_2)$, and $G_3=(U\cup W, E_3)$ be from Corollary \ref{cor:bipgowers} for the parameter $48\e^{1/2}$ and satisfying $|U|=|W|=n\geq N$.  By Theorem \ref{thm:vdischom}, there is some $1\leq i\leq 3$ such that $(n\otimes G_i)\cap  {\bf B}_{\calH}\neq \emptyset$. Recalling ${\bf B}_{\calH}\subseteq \calH$ (by Fact \ref{fact:bhhgraphs}), we can thus fix some $H\in (n\otimes G_i)\cap \calH$. Let $t_{G_i}$ be the size of the smallest $48\e^{1/2}$-regular partition of $G_i$, and let $t_H$ be the size of the smallest $\e$-regular partition of $H$.  By our choice of $G_i$ from Corollary \ref{cor:bipgowers}, $t_{G_i}\geq \Tw(c\e^{-1}/(48)^2)$.   By Lemma \ref{lem:timeslem}, $t_{G_i}\leq t_H$. This shows
$$
M_{\calH}(\e)\geq t_H\geq t_{G_i} \geq  \Tw(c\e^{-1}/(48)^2)=\Tw(\Omega( \e^{-1})).
$$
\end{proof}

We can now characterize the jump to tower.

\begin{theorem}\label{thm:sl1}
For any hereditary $3$-graph property, $\calH$, one of the following holds.
\begin{enumerate}
\item $\calH$ is far from finite SVC-dimension, and 
$$
\Tw(\Omega(\e^{-1}))\leq M_{\calH}(\e)\leq \Tw(6\e^{-4}).
$$
\item $\calH$ is close to finite SVC-dimension, and for some $K>0$, 
$$
M_{\calH}(\e)\leq M_{\calH}^{\hom}(\e^4)\leq 2^{2^{\e^{-K}}}.
$$
\end{enumerate}
\end{theorem}
\begin{proof}
Part (1) follows immediately from Theorem \ref{thm:fastest}.  We now prove (2). Assume $\calH$ is close to finite SVC-dimension. By definition, there is some hereditary $3$-graph property $\calH'$ and an integer $k\geq 1$ such that $\calH$ is close to $\calH'$  and $\SVC(\calH')<k$. Let $C=C(k)$ be as in Theorem \ref{thm:slvc}. By Theorem \ref{thm:slvc}, $M^{\hom}_{\calH'}(\e)\leq 2^{2^{\e^{-C}}}$.  By Proposition \ref{prop:closehom}, we have $M_{\calH}^{\hom}(2\e)\leq M^{\hom}_{\calH'}(\e)$.  Combining these observations with Fact \ref{lem:homup3graphs} yields that for all sufficiently small $\e>0$,
$$
M_{\calH}( \e)\leq M_{\calH}^{\hom}(\e^4)\leq M^{\hom}_{\calH'}(\e^4/2)\leq 2^{2^{(\e^4/2)^{-C}}}.
$$
This implies that for some $K>0$, $M_{\calH}(\e)\leq M_{\calH}^{\hom}(\e^4)\leq 2^{2^{\e^{-K}}}$.
\end{proof}

\section{Lower Bound Lemma for $3$-graphs}\label{sec:blowuplemma}
In this section we prove a $3$-graph analogue of Lemma \ref{lem:blowup0}. This result, Lemma \ref{lem:12blowup},  will be the key ingredient in the lower bounds in the exponential, polynomial, and constant ranges of Theorem \ref{thm:weakhom}.  We will use the following $3$-graph analogue of the canonical equivalence relation from Definition \ref{def:sim}.

\begin{definition}\label{def:simhg}
Suppose $H=(V,E)$ is a $3$-graph.  Given $x,x'\in V$, define $x\sim_H x'$ if and only if 
$$
N_H(x)\cap {V\setminus \{x,x'\}\choose 2}=N_H(x')\cap {V\setminus \{x,x'\}\choose 2}. 
$$
\end{definition}

It is an exercise to show that $\sim_H$ forms an equivalence relation on the vertex set of $H$, all classes of which are cliques or anticliques. For some of our $3$-graph applications, we will require the following version of Definition \ref{def:simhg}, which is relativized to two distinguished subsets.

\begin{definition}\label{def:simhg2}
Suppose $H=(U,E)$ is a $3$-graph and $X,Y\subseteq U$ are nonempty subsets satisfying either $X=Y$ or $X\cap Y=\emptyset$. 

Given $x,x'\in X$, define $x\sim_H^{Y} x'$ if and only if $N_H(x)\cap {Y\setminus \{x,x'\}\choose 2}=N_H(x')\cap {Y\setminus \{x,x'\}\choose 2}$. 
\end{definition}

In other words, $x\sim_H^Yx'$ means $x$ and $x'$ ``look alike" with respect to pairs of vertices from $Y$.  The assumption that $X=Y$ or $X\cap Y=\emptyset$ in Definition \ref{def:simhg2} implies $\sim_H^{Y}$ is always an equivalence relation on the set $X$.  We refer to its equivalence classes as the \emph{$\sim_H^{Y}$ classes in $X$}.  

Note that when $X=Y=U$, the relation $\sim_H^{Y}$ from Definition \ref{def:simhg2} is simply the equivalence relation $\sim_H$ from Definition \ref{def:simhg}. Otherwise, $\sim_H^{Y}$ is a possibly new  equivalence relation on the set $X$.  

Lemma \ref{lem:12blowup} will use Definition \ref{def:simhg2} in its statement, along with the more detailed notion of blowup from part (3) of Definition \ref{def:blowupgraph}.  In particular, we recall that if  $H=(U, E)$ is a $3$-graph and $\mathbf{H}$ is a blowup of $H$, with  vertex set of the form  $V(\mathbf{H})=\bigsqcup_{u\in U}V_u$ as in Definition \ref{def:blowupgraph}(1), then, given integers $n_1,n_2\geq 1$ and sets $X,Y\subseteq U$, we say that ${\bf H}$ is an \emph{$(n_1,n_2;X,Y)$-blowup of $H$} if for each $x\in X$, $|V_x|=n_1$, and for each $y\in Y$, $|V_y|=n_2$.   We can now state Lemma \ref{lem:12blowup}.

\begin{lemma}[Lower Bound Lemma for $3$-graphs]\label{lem:12blowup}
Suppose $H=(U,E)$ is a $3$-graph,  $X,Y\subseteq U$ are nonempty sets satisfying either   $X=Y=U$ or $X\sqcup Y=U$, and $X_1,\ldots, X_t$ is an enumeration of the $\sim_H^{Y}$-classes in $X$.  Assume $n_1,n_2\in \mathbb{N}^{\geq 1}$ and $\e\in (0,1)$ are such that the following hold.  
\begin{enumerate} 
\item $|X|n_1=|Y|n_2$, and
\item $0<\e  <(\frac{1}{16|Y|})^{8}$.
\end{enumerate}
Suppose ${\bf H}$ is an $(n_1,n_2; X,Y)$-blowup of $H$ with vertex set $V({\bf H})=\bigsqcup_{u\in U}V_u$ as in Definition \ref{def:blowupgraph}. For each $1\leq i\leq t$, let ${\bf X}_i=\bigcup_{x\in X_i}V_x$, and set ${\bf X}:={\bf X}_1\cup \cdots \cup {\bf X}_t=\bigcup_{x\in X}V_x$.  For any $\e$-regular partition $\calP$ of ${\bf H}$, there is a set $\calP'\subseteq \calP$ such that the following hold.
\begin{enumerate}[(i)]
\item $|\bigcup_{P\in \calP'}P\cap {\bf X} |\geq (1-2\e^{1/8})|{\bf X}|$.
\item For all $P\in \calP'$, there is   $1\leq i\leq t$ such that $|P\cap {\bf X}_i|\geq (1-\e^{1/4})|P\cap {\bf X}|$.  
\end{enumerate}
\end{lemma}

We now briefly compare Lemma \ref{lem:12blowup} to the graph theoretic analogue used earlier in the paper, Lemma \ref{lem:blowup0}.   Lemma \ref{lem:12blowup} is most similar to Lemma \ref{lem:blowup0} in the case where $X=Y=U$. In this case, the set ${\bf X}$ is simply $V({\bf H})$, and the sets $X_i$ are simply the $\sim_H$-classes of $H$. After this translation, we see that the conclusions of Lemma \ref{lem:12blowup} are the expected $3$-graph analogues of what appears in Lemma \ref{lem:blowup0} in that case.  We require here the additional option of allowing $X$ and $Y$ to partition $U$ specifically to deal with the exponential lower bound in Section \ref{sec:polyexp}, a jump which does not exist in the graphs case. In that specific case, the set $X$ will be very large compared to $\e^{-1}$ (namely exponential in $\e^{-1}$). In that regime, it will be crucial the hypotheses of Lemma \ref{lem:12blowup} only require $\e$ to be small compared to $|Y|^{-1}$, and not  $|X|^{-1}$.

 Before proving Lemma \ref{lem:12blowup}, we require one preliminary. In particular, Lemma \ref{fact:int3graphs} below is the appropriate $3$-graph analogue  of Lemma \ref{lem:int0}, and will be used to identify the blowups of $\sim_H^Y$-classes appearing in conclusion (ii) of Lemma \ref{lem:12blowup}.  The reader may wish to review Notation \ref{not:01nbrs}, and observe that, in this notation, given a $3$-graph $H=(V,E)$ and $yz\in {V\choose 2}$, we have  the following partition of $V(H)$:
\begin{align}\label{not:hub}
V(H)=N_{H^1}(yz)\sqcup N_{H^0}(yz)\sqcup \{y,z\}.
\end{align}

\begin{lemma}\label{fact:int3graphs}
 Suppose $H=(U,E)$ is a $3$-graph, and $X,Y\subseteq U$ satisfy either $X=Y=U$ or $X\sqcup Y=U$. Let $X_1,\ldots, X_t$ be an enumeration of the $\sim^{Y}_H$-classes in $X$.  
 
For any $\alpha\colon {Y\choose 2}\rightarrow \{0,1\}$, the set $X\cap \bigcap_{yz\in {Y\choose 2}}(\{y,z\}\cup N_{H^{\alpha(yz)}}(yz))$ is contained in $X_i$ for some $1\leq i\leq t$.  
 \end{lemma}
 \begin{proof} 
Fix $\alpha\colon {Y\choose 2}\rightarrow \{0,1\}$. To ease notation, set 
$$
W=X\cap  \bigcap_{yz\in {Y\choose 2}}(\{y,z\}\cup N_{H^{\alpha(yz)}}(yz)).
$$
  Suppose towards a contradiction $W$ is not a subset of $X_i$ for any $1\leq i\leq t$.  Then there exist $x,x'\in W$  such that $x\nsim^Y_H x'$.  By definition of $\sim^Y_H$, there is $yz\in {Y\setminus\{x,x'\}\choose 2}$ and $\tau\in \{0,1\}$   such that $xyz\in E^{\tau}$ and $x'yz\in E^{1-\tau}$. Since $x,x'\in W$,  $\{x,x'\}\subseteq \{y,z\}\cup N_{H^{\alpha(yz)}}(yz)$. Since $y,z\notin \{x,x'\}$, this implies $\{x,x'\}\subseteq   N_{H^{\alpha(yz)}}(yz)$. However, this is impossible, since either $\alpha(yz)=\tau$ and $x'\notin N_{H^{\alpha(yz)}}(yz)$, or $\alpha(yz)=1-\tau$ and $x\notin N_{H^{\alpha(yz)}}(yz)$.  Thus we have arrived at a contradiction.   
\end{proof}

We now prove Lemma \ref{lem:12blowup}, following a similar strategy to the proof of Lemma \ref{lem:blowup0}.
\vspace{2mm}

\noindent{\bf Proof of Lemma \ref{lem:12blowup}.} In addition to setting ${\bf X}=\bigcup_{x\in X}V_x$, we set ${\bf Y}=\bigcup_{y\in Y}V_y$.  We begin with a few observations. Since either $X=Y=U$ or $X\sqcup Y=U$ hold, we know that either ${\bf X}={\bf Y}=V({\bf H})$  or ${\bf X}\sqcup {\bf Y}=V({\bf H})$.  Consequently, if we let $\calR$ be the partition of $V({\bf H})$ generated by the sets ${\bf X}$ and ${\bf Y}$, then either $\calR=\{{\bf X},{\bf Y}\}$ or $\calR=\{{\bf X}\}=\{{\bf Y}\}$. In particular, $\calR$ has size at most $2$.  

By definition of ${\bf H}$ and assumption (1), we have $|{\bf X}|=|X|n_1=|Y|n_2=|{\bf Y}|$, and thus
\begin{align}\label{al:hub}
|V({\bf H})|\leq |X|n_1+|Y|n_2=|{\bf X}|+|{\bf Y}|=2|{\bf X}|=2|{\bf Y}|.
\end{align} 
Define
\begin{align*}
\calQ=\{P\cap {\bf X}: P\in \calP\text{ and }P\cap {\bf X}\neq \emptyset\}\cup \{P\cap {\bf Y}: P\in \calP\text{ and }P\cap {\bf Y}\neq \emptyset\}.
\end{align*}
Observe $\calQ$  is the common refinement of $\calP$ with $\calR$, so by Lemma \ref{lem:slicingcor}, if we set 
\begin{align*}
\Sigma_{\reg}&=\{(X,Y,Z)\in \calQ^3: (X,Y,Z)\text{ is $2\e^{1/2}$-regular with respect to ${\bf H}$}\},
\end{align*}
then $|\bigcup_{(X,Y,Z)\in \Sigma_{\reg}}X\times Y\times Z|\geq (1-7\e^{1/2})|V({\bf H})|^3$. To ease notation, let 
$$
\calE_{\reg}=\bigcup_{(X,Y,Z)\in \Sigma_{\reg}}X\times Y\times Z.
$$
Note that by above, $|\calE_{\reg}|\geq (1-7\e^{1/2})|V({\bf H})|^3$. Now set
$$
\calE_{\reg}'=\{(x,y)\in V({\bf H})^2: |N_{\calE_{\reg}}(x,y)|\geq (1-7\e^{1/4})|V({\bf H})|\}.
$$
 By Lemma \ref{lem:averaging}, and since $|\calE_{\reg}|\geq (1-7\e^{1/2})|V({\bf H})|^3$, we have  $|\calE_{\reg}'|\geq (1-\e^{1/4})|V({\bf H})|^2$.  Observe that by its definition,  $\calE_{\reg}'$ is a union of sets of the form $X\times Y$ with $X,Y\in \calQ$. Now define  
 $$
\calE_{\reg}''=\{x\in V({\bf H}): |N_{\calE_{\reg}'}(x)|\geq (1-\e^{1/8})|V({\bf H})|\}.
$$
  By Lemma \ref{lem:averaging}, and since $|\calE_{\reg}'|\geq (1-\e^{1/4})|V({\bf H})|^2$, we have $|\calE_{\reg}''|\geq (1-\e^{1/8})|V({\bf H})|$. Observe that by its definition, $\calE_{\reg}''$ is a union of elements from $\calQ$. Define
\begin{align*}
{\bf X}''=\calE_{\reg}''\cap {\bf X}.
\end{align*}
By our lower bound on $|\calE''_{\reg}|$, we have  
\begin{align}\label{al:agood}
|{\bf X}''|\geq |{\bf X}|-|V({\bf H})\setminus \calE_{\reg}''|\geq |{\bf X}|-\e^{1/8}|V({\bf H})|\geq |{\bf X}|(1-2\e^{1/8}),
\end{align}
where the last inequality is because $|V({\bf H})|\leq 2|{\bf X}|$  by (\ref{al:hub}). Since every set in $\calQ$ is either contained in ${\bf X}$ or disjoint from ${\bf X}$, and since $\calE_{\reg}''$ is a union of sets from $\calQ$, we have that ${\bf X}''$ is a union of elements from $\calQ$.   Let $\calX\subseteq \calQ$ be such that ${\bf X}''=\bigcup_{Q\in \calX}Q$.  By construction of $\calQ$, there is a set $\calP'\subseteq \calP$ so that $\calX=\{P\cap {\bf X}: P\in \calP'\}$, and thus
$$
\bigcup_{P\in \calP'}P\cap {\bf X}=\bigcup_{Q\in \calX} Q={\bf X}''.
$$
We show this $\calP'$ satisfies the desired conclusions of Lemma \ref{lem:12blowup}.  By definition of $\calP'$ and (\ref{al:agood}), we already know that $|\bigcup_{P\in \calP'}P\cap {\bf X}|=|{\bf X}''|\geq (1-2\e^{1/8})|{\bf X}|$, so (i) holds.  The rest of the proof is devoted to showing (ii).  To this end, we fix $Q\in \calX$ for the rest of the proof. Our goal is to show there exists an index $1\leq i\leq t$ such  that $|Q\cap {\bf X}_i|\geq (1-\e^{1/4})|Q|$.  

Recalling the notation appearing in (\ref{not:hub}), we observe that every $yz\in {Y\choose 2}$ gives rise to a natural partition of $X$:
$$
X=\left( N_{H^1}(yz)\cap X\right)\cup \left(N_{H^0}(yz)\cap X\right)\cup \left(\{y,z\}\cap X\right).
$$
This induces a corresponding partition of ${\bf X}$ as follows.
\begin{align}\label{al:partitionob}
{\bf X}=\left(\bigcup_{x\in X\cap N_{H^1}(yz)}V_x\right)\bigsqcup\left(\bigcup_{x\in X\cap N_{H^0}(yz)}V_x\right)\bigsqcup\left((V_y\cup V_z)\cap {\bf X}\right).
\end{align}
We next show that $Q$ cannot substantially intersect two specific pieces of such a partition of ${\bf X}$.    

\begin{claim}\label{cl:qgood}
For all $yz\in {Y\choose 2}$,    
\begin{align}\label{al:uy3}
\min\left\{\left|Q\cap \left(\bigcup_{x\in X\cap N_{H^{0}}(yz)}V_x \right)\right|,\left|Q\cap \left(\bigcup_{x\in X\cap N_{H^{1}}(yz)}V_x \right)\right|\right\}<2\e^{1/2}|Q|.
\end{align}
\end{claim}
\begin{proof}
Fix $yz\in {Y\choose 2}$, and assume towards a contradiction that (\ref{al:uy3}) fails for $yz$. We will use below that $V_y\cup V_z\subseteq {\bf Y}$.  Since $Q\in \calX$,  we have $Q\subseteq \calE_{\reg}''$, which implies that for all $v\in Q$, $|V({\bf H})\setminus N_{\calE'_{\reg}}(v)|\leq \e^{1/8}|V({\bf H})|$.  We also know by definition of $\calE'_{\reg}$ that for all $v,v'\in Q$, $N_{\calE'_{\reg}}(v)=N_{\calE'_{\reg}}(v')$.  Combining these observations yields the following.  
\begin{align}
\nonumber\left|V_y\cap \left(\bigcup_{v\in Q}N_{\calE'_{\reg}}(v)\right)\cap{\bf Y}\right|=\left|V_y\cap \left(\bigcup_{v\in Q}N_{\calE'_{\reg}}(v)\right)\right|&\geq |V_y|-\left|V({\bf H})\setminus \left(\bigcup_{v\in Q}N_{\calE'_{\reg}}(v)\right)\right|\\
&\nonumber\geq |V_y|-\e^{1/8}|V({\bf H})|\\
&\nonumber=|{\bf Y}||Y|^{-1}-\e^{1/8}|V({\bf H})|\\
&\nonumber\geq |{\bf Y}|(|Y|^{-1} -2\e^{1/8})\\
&\nonumber\geq 2\e^{1/2}|{\bf Y}|\\
&\label{al:b1}\geq 2\e^{1/2}\left|\left(\bigcup_{v\in Q}N_{\calE'_{\reg}}(v)\right)\cap {\bf Y}\right|,
\end{align}
where the second equality is because $|{\bf Y}|=|Y|n_2=|Y||V_y|$,  third inequality is since $|V({\bf H})|\leq 2|{\bf Y}|$  by (\ref{al:hub}), and the fourth inequality is because $\e<(16|Y|)^{-8}$.  Since every set in $\calQ$ is either contained in ${\bf Y}$ or disjoint from ${\bf Y}$, and since $\bigcup_{v\in Q}N_{\calE'_{\reg}}(v)$ is a union of sets from $\calQ$, $(\bigcup_{v\in Q}N_{\calE'_{\reg}}(v))\cap {\bf Y}$ is also a union of sets from $\calQ$. Therefore, (\ref{al:b1}) implies there must be some $Q' \in \calQ$ such that $Q'\subseteq  (\bigcup_{v\in Q}N_{\calE'_{\reg}}(v))\cap {\bf Y}$, and such that 
$$
|Q' \cap V_y|\geq 2\e^{1/2}|Q'|.
$$
Note we have $Q\times Q'\subseteq \calE'_{\reg}$. By definition of $\calE_{\reg}'$, we have that for all $(v,v')\in Q\times Q'$, $|V({\bf H})\setminus N_{\calE_{\reg}}(v,v')|\leq  7\e^{1/4}|V({\bf H})|$.  Further, by definition of $\calE_{\reg}$, we know that for all $(v,v'), (v'',v''')\in Q\times Q'$, $N_{\calE_{\reg}}(v,v')=N_{\calE_{\reg}}(v'',v''')$.  Combining these observations yields 
\begin{align}
 \nonumber\left|V_z\cap \left(\bigcup_{(v,v')\in Q\times Q'}N_{\calE_{\reg}}(v,v')\right)\cap{\bf Y}\right|&=\left|V_z\cap \left(\bigcup_{(v,v')\in Q\times Q'}N_{\calE_{\reg}}(v,v')\right)\right|\\
 &\nonumber\geq |V_z|-\left| V({\bf H})\setminus \left(\bigcup_{(v,v')\in Q\times Q'}N_{\calE_{\reg}}(v,v')\right)\right|\\
&\nonumber\geq |V_z|-7\e^{1/4}|V({\bf H})|\\
&\nonumber=|{\bf Y}||Y|^{-1}-7\e^{1/4}|V({\bf H})|\\
&\nonumber\geq |{\bf Y}|(|Y|^{-1}-14\e^{1/4})\\
&\nonumber\geq 2\e^{1/2}|{\bf Y}|\\
&\label{al:c1}\geq 2\e^{1/2}\left|\left(\bigcup_{(v,v')\in Q\times Q'}N_{\calE_{\reg}}(v,v')\right)\cap{\bf Y}\right|,
\end{align}
where the second equality uses that $|{\bf Y}|=|Y|n_2=|Y||V_z|$, the third inequality uses that $|V({\bf H})|\leq 2|{\bf Y}|$ holds by (\ref{al:hub}), and the fourth inequality uses $\e<(16|Y|)^{-8}$. Since every set in $\calQ$ is either contained in ${\bf Y}$ or disjoint from ${\bf Y}$, and since $\bigcup_{(v,v')\in Q\times Q'}N_{\calE_{\reg}}(v,v')$ is a union of sets from $\calQ$, we have that $ (\bigcup_{(v,v')\in Q\times Q'}N_{\calE_{\reg}}(v,v'))\cap {\bf Y}$ is also a union of sets from $\calQ$. Therefore, (\ref{al:c1}) implies there is some $Q''\in \calQ$ such that $Q''\subseteq(\bigcup_{(v,v')\in Q\times Q'}N_{\calE_{\reg}}(v,v'))\cap {\bf Y}$, and such that  
$$
|Q''\cap V_z|\geq 2\e^{1/2}|Q''|. 
$$
By construction, we have $Q\times Q'\times Q''\subseteq \calE_{\reg}$, i.e. $(Q,Q',Q'')\in \Sigma_{\reg}$.  Combining this with $|Q'\cap V_y|\geq 2\e^{1/2}|Q'|$, $|Q''\cap V_z|\geq 2\e^{1/2}|Q''|$, and our assumption that (\ref{al:uy3}) fails for $yz$, we have 
\begin{align*} 
\max\Bigg\{&\left|d_{{\bf H}}\left(Q\cap \left(\bigcup_{x\in  X\cap N_{H^1}(yz)}V_x \right), Q'\cap V_y, Q''\cap V_z\right)-d_{{\bf H}}\left(Q,Q',Q''\right)\right|, \\
&\nonumber\left|d_{{\bf H}}\left(Q\cap \left(\bigcup_{x\in X\cap N_{H^0}(yz)}V_x \right), Q'\cap V_y, Q''\cap V_z\right)-d_{{\bf H}}\left(Q,Q',Q''\right)\right|\Bigg\}\leq 2\e^{1/2}.
\end{align*}
By the triangle inequality, this yields
\begin{align*} 
 & \Bigg|d_{{\bf H}}\left(Q\cap \left(\bigcup_{x\in  X\cap N_{H^1}(yz)}V_x \right), Q'\cap V_y, Q''\cap V_z\right)\\
&-d_{{\bf H}}\left(Q\cap \left(\bigcup_{x\in X\cap N_{H^0}(yz)}V_x \right), Q'\cap V_y, Q''\cap V_z\right)\Bigg|\leq 4\e^{1/2}<1,
\end{align*}
where the final inequality uses $\e<(16|Y|)^{-8}$ and $|Y|\geq 1$. However, this is impossible, since our assumption that ${\bf H}$ is a blowup of $H$ implies
$$
d_{{\bf H}}\left(Q\cap \left(\bigcup_{x\in X\cap N_{H^1}(yz)}V_x \right), Q'\cap V_y, Q''\cap V_z\right)=1,
$$
while 
$$
d_{{\bf H}}\left(Q\cap \left(\bigcup_{x\in X\cap N_{H^0}(yz)}V_x \right), Q'\cap V_y, Q''\cap V_z\right)=0.
$$
This finishes the proof of Claim \ref{cl:qgood}.
\end{proof}

 Claim \ref{cl:qgood} implies that for each $yz\in {Y\choose 2}$, there is $\beta(yz)\in \{0,1\}$  such that 
\begin{align}\label{al:partitionob3}
\left| Q\cap \left(\bigcup_{x\in  X\cap N_{H^{\beta(yz)}}(yz)}V_x \right) \right|<2\e^{1/2}|Q|.
\end{align}
Let $\alpha(yz)$ be such that $\{\alpha(yz),\beta(yz)\}=\{0,1\}$.  Note that by (\ref{al:partitionob}), each $yz\in {Y\choose 2}$ induces a partition  of $Q$ given by
\begin{align}\label{al:partitionob2}
Q=\left( Q\cap  \bigcup_{x\in X\cap N_{H^{\alpha(yz)}}(yz) }V_x\right)\bigsqcup \left( Q\cap  \bigcup_{x\in X\cap N_{H^{\beta(yz)}}(yz) }V_x\right)\bigsqcup \left( Q\cap \bigcup_{x\in X\cap \{y,z\}}V_x\right).
\end{align} 
The values $\alpha(yz)$ for $yz\in {Y\choose 2}$ naturally give  rise to the following subset of $X$.
$$
W=\bigcap_{yz\in {Y\choose 2}}\left(\left(\{y,z\}\cup N_{H^{\alpha(yz)}}(yz)\right)\cap X\right).
$$
By Lemma \ref{fact:int3graphs},  there is $1\leq i\leq t$ such that $W\subseteq X_i$, and consequently, $\bigcup_{x\in W}V_x\subseteq {\bf X}_i$. Combining this with (\ref{al:partitionob3}) and (\ref{al:partitionob2}), we have 
\begin{align}\label{al:12blowupa}
\nonumber|Q\cap {\bf X}_i|\geq \left|Q\cap \left(\bigcup_{x\in W}V_x\right)\right|&=\left|Q\cap\left( \bigcap_{yz\in {Y\choose 2}}\left(\bigcup_{x\in X\cap (\{y,z\}\cup N_{H^{\alpha(yz)}}(yz))}V_x\right)\right)\right|\\
&\nonumber\geq |Q|-\sum_{yz\in {Y\choose 2}} \left|Q\cap \left( \bigcup_{x\in X\cap N_{H^{\beta(yz)}}(yz)}V_x\right)\right|\\
&\nonumber\geq |Q|-|Y|^22 \e^{1/2} |Q|\\
&\geq (1-\e^{1/4})|Q|,
\end{align}
where the last inequality is because $\e<(16|Y|)^{-8}$.  This completes our proof of  Lemma \ref{lem:12blowup}.
\qed

\vspace{2mm}

We will use Lemma \ref{lem:12blowup} to produce the lower bounds in the constant, polynomial, and exponential ranges of Theorem \ref{thm:weak}.  The exact lower bounds we are able to obtain in the polynomial and exponential ranges are limited by the relationship between $\e$ and $|Y|$ appearing in assumption (2) of Lemma \ref{lem:12blowup}.  We have been unable to improve this relationship beyond a power of $8$.  Finding a way to avoid using Lemma \ref{lem:slicingcor} would improve the power to $4$.  If that were possible, then performing the first application of Lemma \ref{lem:averaging} more carefully, as is done in the graphs version (Lemma \ref{lem:blowup0}), could further improve the power to $2-o(1)$.  However, our current proof strategy could not do better than this due to the second application of Lemma \ref{lem:averaging}.  

We end this section by proving an easy corollary of Lemma \ref{lem:12blowup} which applies when the sets $X_1,\ldots, X_t$ appearing in its statement all have size $1$.

\begin{corollary}\label{cor:12blowup}
Let $H=(U,E)$ be a $3$-graph, and let $X,Y\subseteq U$ be nonempty sets satisfying either $X=Y=U$ or $X\sqcup Y=U$. Suppose every $\sim_H^Y$-class in $X$ has size $1$. Assume $n_1,n_2\in \mathbb{N}^{\geq 1}$ and $\e\in (0,1)$ are such that 
\begin{enumerate} 
\item $|X|n_1=|Y|n_2$, and
\item $0<\e  <(\frac{1}{16|Y|})^8$.
\end{enumerate}
If  ${\bf H}$ is an $(n_1,n_2; X,Y)$-blowup of $H$, then any $\e$-regular partition $\calP$ of ${\bf H}$ satisfies
$$
|\calP|\geq (1-2\e^{1/8})(1-\e^{1/4})|X|.
$$ 
\end{corollary}
\begin{proof}
By assumption, the $\sim_H^Y$-classes in $X$ are exactly the $1$-element subsets of $X$.  Enumerate these as $X_1,\ldots, X_t$.  Let $V({\bf H})=\bigsqcup_{x\in V(H)}V_x$ as in Definition \ref{def:blowupgraph}, where for each $x\in X$, $|V_x|=n_1$ and for each $y\in Y$, $|V_y|=n_2$.  Let ${\bf X}=\bigcup_{x\in X}V_x$, and for each $1\leq i\leq t$, let ${\bf X}_i=\bigcup_{x\in X_i}V_x$.  Note $|{\bf X}|=|X|n_1$, and each ${\bf X}_i$ has size $n_1$ (since each $X_i$ has size $1$). By Lemma \ref{lem:12blowup}, there is a set $\calP'\subseteq \calP$ such that 
\begin{align}\label{al:cor1}
\left|\bigcup_{P\in \calP'}P\cap {\bf X}\right|\geq (1-2\e^{1/8})|{\bf X}|,
\end{align}
and such that for all $P\in \calP'$, there is   $1\leq g(P)\leq t$ satisfying $|P\cap {\bf X}_{g(P)}|\geq (1-\e^{1/4})|P\cap {\bf X}|$.  Rearranging, we have that
\begin{align}\label{al:cor2}
\text{ for all $P\in \calP'$, $|P\cap {\bf X}|\leq \frac{|P\cap {\bf X}_{g(P)}|}{1-\e^{1/4}}\leq \frac{| {\bf X}_{g(P)}|}{1-\e^{1/4}}=\frac{n_1}{1-\e^{1/4}}$},
\end{align}
where the equality uses that each of ${\bf X}_1,\ldots,{\bf X}_t$ has size $n_1$. Combining (\ref{al:cor1}) and (\ref{al:cor2}) yields
$$
|\calP|\geq \frac{|\bigcup_{P\in \calP'}P\cap {\bf X}|}{\max_{P\in \calP'}|P\cap {\bf X}|}\geq \frac{(1-2\e^{1/8})(1-\e^{1/4})|{\bf X}|}{n_1}=(1-2\e^{1/8})(1-\e^{1/4})|X|,
$$
where the last equality uses that $|{\bf X}|=|X|n_1$.
\end{proof}

\section{Jump from Polynomial to Exponential}\label{sec:polyexp}
In this section we prove the existence of a jump from polynomial to exponential  growth.  This jump will be characterized by whether or not a hereditary $3$-graph property $\calH$ is  close to finite VC-dimension.   In particular, we will show that if $\calH$ is close to finite VC-dimension, then $M_{\calH}$ is bounded above by a polynomial, and if $\calH$ is far from finite VC-dimension, then $M_{\calH}$ is bounded below by a single exponential.  
 
 In light of Theorem \ref{thm:sl1}, it will suffice to prove the existence of the jump between polynomial and exponential growth rates for hereditary $3$-graph properties $\calH$ which are close to finite SVC-dimension (as all other $\calH$ have tower type growth). 

\begin{theorem}\label{thm:exp}
 Suppose $\calH$ is a hereditary $3$-graph property which is close to finite SVC-dimension. Then one of the following holds. 
\begin{enumerate}
\item $\calH$ is close to finite VC-dimension. In this case there exists $C>0$ such that 
$$
M_{\calH}(\e)\leq M_{\calH}^{\hom}(\e^4)\leq \e^{-C}.
$$
\item $\calH$ is far from finite VC-dimension. In this case, there exists $C>0$ such that 
$$
2^{\Omega(\e^{-1/8})}\leq M_{\calH}(\e)\leq M_{\calH}^{\hom}(\e^4)\leq 2^{2^{\e^{-C}}}.
$$ 
\end{enumerate}
\end{theorem}
\begin{proof}
Fix  a hereditary $3$-graph property $\calH$ which is close to finite SVC-dimension.  

Suppose first $\calH$ is close to finite VC-dimension. Then there exists a hereditary $3$-graph property $\calH'$ with $\VC(\calH')<\infty$, such that $\calH$ is close to $\calH'$.  Let $k=\VC(\calH')$.  By Theorem \ref{thm:foxpachsuk}, there exists $D=D(k)$ such that $M_{\calH'}^{\hom}(\e)\leq \e^{-D}$ holds for all sufficiently small $\e$.  Combining this with Fact \ref{lem:homup3graphs}  and Proposition \ref{prop:closehom}, we have
\begin{align*}
M_{\calH}(\e)\leq M_{\calH}^{\hom}(\e^4)\leq M_{\calH'}^{\hom}(\e^4/2)\leq (\e^4/2)^{-D}.
\end{align*}  
Clearly this implies there exists $C=C(k)>0$ such that $M_{\calH}(\e)\leq M_{\calH}^{\hom}(\e^4)\leq \e^{-C}$.  This completes the proof of (1).

Suppose now $\calH$ is far from finite VC-dimension.   The existence of $C >0$ such that  
$$
M_{\calH}(\e)\leq M_{\calH}^{\hom}(\e^4)\leq 2^{2^{\e^{-C }}}
$$
follows from our assumption that $\calH$ is close to finite SVC-dimension and  Theorem \ref{thm:sl1}(2).  We have left to prove an exponential lower bound on $M_{\calH}$.  Let $\e$ be sufficiently small, and set $K=\lfloor (33 \e^{1/8})^{-1}\rfloor$.

Since $\calH$ is far from finite VC-dimension, Proposition \ref{prop:equiv2} implies the existence of some $H\in \widehat{\PS(K)}\cap{\bf B}_{ \calH }$.  By Definition \ref{def:ukhat}, we may assume $H$ has vertex set of the form
$$
V(H)=X\sqcup Y,
$$
 where $X=\{c_S: S\subseteq [K]\}$ and $Y=\{a_i,b_i: i\in [K]\}$, and edge set $E(H)$ satisfying  
 \begin{align}\label{al:exp1}
 \text{$\{a_ib_ic_S: i\in S\}\subseteq E(H)$ and $\{a_ib_ic_S: i\notin S\}\cap E(H)=\emptyset$.  }
 \end{align}
Note $|Y|\leq 2K$ holds by definition, and (\ref{al:exp1}) implies $|X|=2^K$.   Let $n_1$ be sufficiently large and divisible by $|Y|$, and set $n_2=|Y|^{-1}|X|n_1$.  Let $X_1,\ldots, X_t$ be an enumeration of the $\sim_H^Y$-equivalence classes in $X$ (see Definition \ref{def:simhg2}).  It is an exercise to check that (\ref{al:exp1}) implies $|X_i|=1$ for all $1\leq i\leq t$.

Since $H\in {\bf B}_{ \calH }$,  $\calH$ contains an $m$-blowup of $H$ for all $m\geq 1$.     Since $\calH$ is hereditary and $X$ and $Y$ are disjoint, this implies that there is an $(n_1,n_2; X,Y)$-blowup ${\bf H}$ of $H$ with ${\bf H}\in \calH$ (see Definition \ref{def:blowupgraph}(3)).   

We check the hypotheses of Corollary \ref{cor:12blowup} are satisfied by ${\bf H}$ and $H$, with respect to the sets $X$ and $Y$, and the parameter $\e$.  First, $X\cup Y=V(H)$ is a partition by construction.  Note
$$
|Y|\leq 2K=2\lfloor (33 \e^{1/8})^{-1}\rfloor\leq 2(33 \e^{1/8})^{-1}<\frac{1}{16\e^{1/8}}.
$$
Consequently, $\e<\Big(\frac{1}{16 |Y|}\Big)^8$.  By our choices for $n_1$ and $n_2$, we have 
$$
|Y|n_2=|Y|(|Y|^{-1}|X|n_1)=|X|n_1. 
$$
Finally, we recall (\ref{al:exp1}) implies each $X_i$ has size $1$.  This concludes our verification that the hypotheses of Corollary \ref{cor:12blowup} are satisfied.  Corollary \ref{cor:12blowup} then implies any $\e$-regular partition of ${\bf H}$ satisfies 
\begin{align*}
|\calP|\geq (1-2\e^{1/8})(1-\e^{1/4})|X|=(1-2\e^{1/8})(1-\e^{1/4})2^K\geq 2^{ \e^{-1/8}/34 },
\end{align*}
where the last inequality uses that $K=\lfloor (33\e^{1/8})^{-1}\rfloor$ and that $\e$ is sufficiently small.  The argument above shows that $M_{ \calH }(\e)\geq 2^{\Omega(\e^{-1/8})}$, which completes our proof of (2).  
\end{proof}

We end this section by showing the general form of lower bound in Theorem \ref{thm:exp} cannot be improved beyond a single exponential.\footnote{The fact that this lower bound cannot be improved beyond a single exponential can  now be easily deduced from the main theorem of \cite{GSW}.  We have chosen to keep Theorem \ref{thm:exp} in this paper as it was an important part of the narrative before \cite{GSW} appeared.}
 
\begin{theorem}
There is a hereditary $3$-graph property $\calH$ with $\SVC(\calH)<\infty$, so that $\calH$ is far from  finite VC-dimension, and so that for some $K,K'>0$.
$$
2^{\e^{-K}}\leq M_{\calH}(\e)\leq 2^{\e^{-K'}}.
 $$
\end{theorem}
\begin{proof}
For all $k,n\geq 1$, let $H(k,n)$ be a $3$-partite $3$-graph with vertex set 
$$
X_1\sqcup \cdots \sqcup X_k\sqcup Y_1\sqcup \cdots \sqcup  Y_k\sqcup \bigsqcup_{S\subseteq [k]}Z_S,
$$
 where for each $i\in [k]$ and $S\subseteq[k]$, $|X_i|=|Y_i|=|Z_S|=n$, and with edge set
$$
E(H(k,n))=\bigcup_{S\subseteq[k]}\bigcup_{i\in S}K_3[X_i,Y_i,Z_S].
$$
Note $H(k,n)$ is an $n$-blowup of an element from $\widehat{\PS(k)}$.  Let $\calH$ be the hereditary $3$-graph property obtained by taking the closure of $\{ H(k,n): k,n\in \mathbb{N}^{\geq 1}\}$ under isomorphisms and induced sub-$3$-graphs.  Since $\calH$ contains $n$-blowups of elements from $\widehat{\PS(k)}$ for every $k$ and $n$, Proposition \ref{prop:equiv2} implies $\calH$ is far from finite VC-dimension.  We leave it as an exercise to the reader to see that on the other hand, $\calH$ has finite slicewise VC-dimension.  By Theorem \ref{thm:exp}, there is some $K$ so that $2^{\e^{-K}}\leq M_{\calH}(\e)$.

The rest of the proof is devoted to the stated upper bound for $M_{\calH}(\e)$.  Fix a sufficiently small $\e>0$, and suppose $H\in \calH$ has $|V(H)|$ sufficiently large.  By definition of $\calH$, we may assume there are $n,k$ so that $V(H)=A\sqcup B\sqcup C$, where $A=A_1\sqcup \cdots \sqcup A_k$, $C=C_1\sqcup \ldots \sqcup C_k$, and $B=\bigsqcup_{S\subseteq [k]}B_S$, and for each $i\in [k]$ and $S\subseteq [k]$, $0\leq |A_i|,|C_i|,|B_S|\leq n$, and where 
\begin{align}\label{al:noimprove}
E(H)=\bigcup_{S\subseteq[k]}\bigcup_{i\in S}K_3[A_i,C_i,B_S].
\end{align}
We first show $H$ has a  $2\e$-homogeneous partition of size at most $2^{\e^{-3}}$.  If one of $A$, $B$, or $C$ is empty, then $H$ has no edges, so $\{V(H)\}$ is such a partition.  Assume now each of $A$, $B$, and $C$ are nonempty. Let $\calP_A=\{A_1,\ldots, A_k\}$, $\calP_C=\{C_1,\ldots, C_k\}$, and $\calP_B=\{B_S: S\subseteq[k]\}$.  Let 
\begin{align*}
\calP^{\bg}_A=\{X\in \calP_A: |X|\geq \e^{2}|A|\}\text{ and }\calP^{\bg}_C=\{X\in \calP_C: |X|\geq \e^{2}|C|\}.
\end{align*}
 Note each of $\calP^{\bg}_A$ and $\calP^{\bg}_C$ have size at most $ \e^{-2}$.  Let $\{i_1,\ldots, i_t\}\subseteq [k]$ be a minimal subset of $[k]$ satisfying 
 $$
 \calP^{\bg}_A\subseteq \{A_{i_1},\ldots, A_{i_t}\}\text{ and } \calP^{\bg}_C\subseteq \{C_{i_1},\ldots, C_{i_t}\}.
 $$
Clearly $t\leq |\calP^{\bg}_A|+|\calP^{\bg}_C|\leq 2 \e^{-2}$. Now define
$$
A_{\sm}:=\bigcup_{j\in [k]\setminus \{i_1,\ldots, i_t\} }A_j \text{ and }C_{\sm}:=\bigcup_{j\in [k]\setminus \{i_1,\ldots, i_t\} }C_j. 
$$
Observe that we have no information about the sizes of $A_{\sm}$ and $C_{\sm}$.  For each subset $T\subseteq \{i_1,\ldots, i_t\}$, let 
$$
B_T'=\bigcup_{S\subseteq [k]: S\cap \{i_1,\ldots, i_t\}=T }B_S.
$$
Note that $\{B_T':T\subseteq \{i_1,\ldots, i_t\}\}$ is a partition of $B$ with at most $2^{t}\leq 2^{2\e^{-2}}$ many parts. Define now 
$$
\calQ:=\{B_T': T\subseteq \{i_1,\ldots, i_t\}\text{ and }B_T'\neq \emptyset \}\cup \{A_{i_1},\ldots, A_{i_t}\}\cup  \{C_{i_1},\ldots, C_{i_t}\} \cup \{A_{\sm},C_{\sm}\}.
$$
Then by construction, $\calQ$ is a partition of $V(H)$ and
$$
|\calQ|\leq 2^{2\e^{-2}}+4\e^{-2}+2\leq 2^{\e^{-3}},
$$
where the last inequality is because $\e$ is sufficiently small. We show $\calQ$  is $2\e$-homogeneous with respect to $H$.  Let 
$$
\Sigma_{\hom}=\{(X,Y,Z)\in \calQ^3: d_H(X,Y,Z)\in [0,\e)\cup (1-\e,1]\},
$$
and set $\calE_{\hom}=\bigcup_{(X,Y,Z)\in \Sigma_{\hom}}X\times Y\times Z$.  We note to the reader that $\Sigma_{\hom}$ is purposefully defined to contain the $\e$-homogeneous triples, not the $2\e$-homogeneous triples.  We will show $|\calE_{\hom}|\geq (1-2\e)|V(H)|^3$, which suffices to prove $\calQ$ is $2\e$-homogeneous.  We begin with a claim.

\begin{claim}\label{cl:shom}
 $|\calE_{\hom}\cap (A\times C\times B)|\geq  (1-2\e) |A||B||C|$.
 \end{claim}
\begin{proof}
We first show that 
\begin{align}\label{al:u}
|\calE_{\hom}\cap (A_{\sm}\times C\times B)|\geq (1-\e)|A_{\sm}||C||B|.
\end{align}
Observe that by our choice of $\{i_1,\ldots, i_t\}$, the definition of $A_{\sm}$, and (\ref{al:noimprove}),
\begin{align*}
|(A_{\sm}\times C\times B)\cap \overline{E(H)}|\leq |B|\sum_{j\in [k]\setminus \{i_1,\ldots, i_t\} } |A_j||C_j| &\leq |B|\sum_{j\in [k]\setminus \{i_1,\ldots, i_t\} } |A_j|\e^2|C| \\
&= \e^2|B||C|\sum_{j\in [k]\setminus \{i_1,\ldots, i_t\} } |A_j| \\
&=\e^2|A_{\sm}||C||B|.
\end{align*}
Combining this with the fact that $A_{\sm}\times C\times B$ is a disjoint union of sets of the form $X\times Y\times Z$ for $X,Y,Z\in \calQ$, we see that Lemma \ref{lem:averaging} implies $|\calE_{\hom}\cap (A_{\sm}\times C\times B)|\geq (1-\e)|A_{\sm}||C||B|$, so (\ref{al:u}) holds. A similar argument shows that  
$$
|\calE_{\hom}\cap (A\times C_{\sm}\times B)|\geq (1-\e)|A||C_{\sm}||B|.
$$
Consider now a  triple $(X,Y,Z)\in \calQ^3$ with $X\subseteq A$, $Y\subseteq C$ and $Z\subseteq B$ such that $X\times Y\times Z$ is disjoint from $A_{\sm}\times C\times B$ and $A\times C_{\sm}\times B$.  By construction, any such $(X,Y,Z)$ has the form $(A_{i_u},C_{i_v},B_T')$ for some $1\leq u,v\leq t$ and $T\subseteq \{i_1,\ldots, i_t\}$.  If $u\neq v$, then (\ref{al:noimprove}) implies $d_H( A_{i_u},C_{i_v},B_T')=0$.  If $u=v$, then (\ref{al:noimprove}) implies  either $i_u=i_v\in T$ and $d_H(A_{i_u},C_{i_v},B_T')=1$, or $i_u=i_v\notin T$ and $d_H(A_{i_u},C_{i_v},B_T')=0$. In either case, we have $(A_{i_u},C_{i_v},B_T')\in \Sigma_{\hom}$. Consequently, 
\begin{align*}
|\calE_{\hom}\cap (A\times C\times B)|&\geq |A||C||B|-\e|A_{\sm}||C|| B|-\e|A|| C_{\sm}||B|\\ 
&\geq (1-2\e)|A||C||B|.
\end{align*}
This concludes the proof of Claim \ref{cl:shom}.
\end{proof}
Clearly Claim \ref{cl:shom} implies that for any permutation $\sigma\colon \{A,B,C\}\rightarrow \{A,B,C\}$, we have $|\calE_{\hom}\cap (\sigma(A)\times \sigma(B)\times \sigma(C))|\geq (1-2\e)|A||B||C|$.  On the other hand, let us call a triple $(X,Y,Z)\in \calQ^3$ \emph{non-crossing} if at least two of the sets $X$, $Y$, $Z$ are contained in one of $A$, $B$, or $C$. Since $H$ is tripartite, any non-crossing triple from $\calQ^3$ has density $0$, and is thus in $\Sigma_{\hom}$.  Combining these observations, we can conclude that  
\begin{align*}
|\calE_{\hom}|&\geq  |V(H)|^3- \sum_{\sigma\colon \{A,B,C\}\rightarrow \{A,B,C\}\text{ a permutation}}2\e|\sigma(A)\times \sigma(B)\times \sigma(C)| \\
&\geq (1-2\e)|V(H)|^3,
\end{align*}
where the last inequality uses that $V(H)=A\sqcup B\sqcup C$.  This concludes our verification that $\calQ$ is a $2\e$-homogeneous partition of $H$.  By Proposition \ref{prop:2.21}, $\calQ$ is $(2\e)^{1/4}$-regular, so we have shown $M_{\calH}((2\e)^{1/4})\leq |\calQ|\leq 2^{\e^{-3}}$. Clearly this implies there exists $K'>0$ such that for all sufficiently small $\e>0$, $M_{\calH}(\e)\leq 2^{\e^{-K'}}$. 
\end{proof}

\section{Jump from Constant to Polynomial}\label{sec:constpoly}

This section contains the jump from constant to polynomial growth for $3$-graphs.   We will first cover preliminaries about almost prime $3$-graphs in Subsection \ref{ss:ap}. In Subsection \ref{ss:constjump} we prove the existence of the desired jump, after which we prove our main result, Theorem \ref{thm:weakhom}.

\subsection{Almost prime $3$-graphs}\label{ss:ap}

 A \emph{prime $3$-graph} is naturally defined to be a $3$-graph $H$ in which all $\sim_H$-classes have size at most $2$.  In this section, we will work with the more general notion of an \emph{almost} prime $3$-graph, which we now define. 

\begin{definition}\label{def:3graphirred}
A $3$-graph $H$ is \emph{almost prime} if every $\sim_H$-class has size at most $3$.  
\end{definition}

The following sets of $3$-graphs will play an important role in Theorem \ref{thm:weakhom}, analogous to the role played by the graphs from Definition \ref{def:prime} in Theorem \ref{thm:alljumphom}.

\begin{definition}\label{def:prime3}
Given $k\geq 1$, let $\widehat{\Irr(k)}$  be the set of all $3$-graphs $H=(V,E)$ such  that there exists some labeling of the vertices  $V=\{a_1,\ldots, a_k,b_1,\ldots, b_k, c_1,\ldots, c_k\}$ (possibly with repetitions) satisfying 
$$
\{a_1,\ldots, a_k\}\cap \{b_1,\ldots, b_k,c_1,\ldots, c_k\}=\emptyset \text{ and }|\{a_i,b_i,c_i\}|=3\text{ for each }i\in [k],
$$
and such that one of (a)-(c) holds.
\begin{enumerate}[(a)]
\item $a_ib_jc_j\in E$ if and only if $i=j$.
\item $a_ib_jc_j\in E$ if and only if $i\neq j$.
\item $a_ib_jc_j\in E$ if and only if $i\leq j$.
\end{enumerate} 
\end{definition}

In analogy to Theorem \ref{thm:prime}, we will show that any sufficiently large almost prime $3$-graph contains an element of $\widehat{\Irr(k)}$ as an induced sub-$3$-graph.
 
\begin{theorem}\label{thm:prime3}
For all integers $m\geq 1$ there is an integer $N\geq 1$ such that the following holds. Suppose $H$ is an almost prime $3$-graph with at least $N$ vertices. Then $H$ contains an element of $\widehat{\Irr(m)}$ as an induced sub-$3$-graph.
\end{theorem}

The proof of Theorem \ref{thm:prime3} is similar to the proof of Theorem \ref{thm:prime} and appears in the appendix (see Appendix \ref{ss:ramsey2}).  

While Theorem \ref{thm:prime3} suffices for the purposes of our main theorems, it is  not a full analogue of Theorem \ref{thm:prime} because the $3$-graphs in $\widehat{\Irr(m)}$ are not necessarily themselves almost prime.  Consider a $3$-graph $H$ with $3m$ vertices  $\{a_1,\ldots, a_m,b_1,\ldots, b_m, c_1,\ldots, c_m\}$ and edge set $\{a_ib_ic_j: i,j\in [m]\}$. Then $H\in \widehat{\Irr(m)}$ because its edge set satisfies condition (a) of Definition \ref{def:prime3}. However, as long as $m>1$, $H$ is not almost prime because all of the vertices $c_1,\ldots, c_m$ are in the same $\sim_H$-class.  On the other hand,  $H$ does contain  an obvious almost prime induced sub-$3$-graph, namely $H[\{a_1,\ldots, a_m,b_1,\ldots, b_m,c_1\}]$. 

In general, one could produce a fuller analogue of Theorem \ref{thm:prime} by combining Theorem \ref{thm:prime3} with a proof that for all $k\geq 1$, there is $m\geq 1$ so that every element of $\widehat{\Irr(m)}$ contains an almost prime induced sub-$3$-graph on at least $k$ vertices.   However, such a result would be significantly more technical than Theorem \ref{thm:prime3}, and is not necessary for our main results.

\subsection{Proof of the jump}\label{ss:constjump}

In analogy to Section \ref{sec:graphs}, we will show that when ${\bf B}_{\calH}$ contains only finitely many non-isomorphic almost prime $3$-graphs, $M_{\calH}$ is constant, and when  ${\bf B}_{\calH}$ contains infinitely many non-isomorphic almost prime $3$-graphs,  $M_{\calH}$ is at least polynomial.   

Our first step is to use Lemma \ref{lem:12blowup} to show that when ${\bf B}_{\calH}\cap \widehat{\Irr(m)}\neq \emptyset$ for  arbitrarily large $m$,  $M_{\calH}$ is at least polynomial.

\begin{theorem}\label{thm:vcfar3}
Suppose $\calH$ is a hereditary $3$-graph property and assume that for all $m\geq 1$, ${\bf B}_{\calH}\cap \widehat{\Irr(m)}\neq \emptyset$.   Then $M_{\calH}(\e)\geq \Omega(\e^{-1/8})$.
\end{theorem}
\begin{proof}
 Let $\e>0$ be sufficiently small, and set $K=\lfloor (33\e^{1/8})^{-1} \rfloor$.  By assumption, there is some $H\in {\bf B}_{\calH}\cap \widehat{\Irr(K)}$.  By definition of $\widehat{\Irr(K)}$, the vertex set of $H$ can be written (possibly with some repetitions) as 
 $$
 \{a_1,\ldots, a_K,b_1,\ldots, b_K, c_1,\ldots, c_K\},
 $$
such that for each $i\in [K]$, $|\{a_i,b_i,c_i\}|=3$, such that 
$$
 \{a_1,\ldots, a_K\}\cap \{b_1,\ldots, b_K, c_1,\ldots, c_K\}=\emptyset,
$$
and such that one of (a)-(c) holds.
\begin{enumerate}[(a)]
\item $a_ib_jc_j\in E(H)$ if and only if $i=j$.
\item $a_ib_jc_j\in E(H)$ if and only if $i\neq j$.
\item $a_ib_jc_j\in E(H)$ if and only if $i\leq j$.
\end{enumerate} 
 Let $X=\{a_1,\ldots, a_K\}$ and $Y=\{b_1,\ldots, b_K,c_1,\ldots, c_K\}$.   It is not difficult to see that, because one of (a)-(c) holds, we must have $|X|=K$.  Let $n_1$ be sufficiently large and divisible by $|Y|$, and set $n_2=n_1|X||Y|^{-1}$.  Let $X_1,\ldots, X_t$ be an enumeration of the $\sim^Y_H$-classes in $X$.  It is not hard to check that, because one of (a)-(c) holds, each $X_i$ has size $1$.
 
 Since $H\in {\bf B}_{\calH}$, there exists an $(n_1,n_2; X,Y)$-blowup ${\bf H}$ of $H$ satisfying ${\bf H}\in \calH$.  Let $V({\bf H})=\bigsqcup_{u\in V(H)}V_u$ be as in Definition \ref{def:blowupgraph}.  Set ${\bf X}=\bigcup_{x\in X}V_x$ and for each $1\leq i\leq t$, let ${\bf X}_i=\bigcup_{x\in X_i}V_x$.  Note $|{\bf X}|=|X|n_1=Kn_1$, and each ${\bf X}_i$ has size $n_1$ (since each $X_i$ has size $1$).
 
 We check the hypotheses of Corollary \ref{cor:12blowup} are satisfied by $H$ and ${\bf H}$ with respect to the sets $X$ and $Y$, and the parameter $\e$.  We have already argued each of the sets $X_1,\ldots, X_t$ has size $1$. By definition of $n_1$ and $n_2$, $|X|n_1=|Y|n_2$. Note 
 $$
 |Y|\leq 2K=2\lfloor (33\e^{1/8})^{-1}\rfloor \leq   2 (33\e^{1/8})^{-1} <\frac{1}{16 \e^{1/8}}.
 $$
Consequently, we have $\e<(\frac{1}{16|Y|})^8$. This concludes our verification of the hypotheses of Corollary \ref{cor:12blowup}.  Corollary \ref{cor:12blowup} then implies that any $\e$-regular partition $\calP$ of ${\bf H}$ satisfies 
\begin{align*}
|\calP|\geq  (1-2\e^{1/8})(1-\e^{1/4}) |X|=(1-2\e^{1/8})(1-\e^{1/4})\lfloor (33\e^{1/8})^{-1}\rfloor\geq  \e^{-1/8}/34,
\end{align*}
where the equality uses $|X|=K$, and the last inequality is because $\e$ is sufficiently small.  The argument above shows $M_{\calH}(\e)\geq \Omega(\e^{-1/8})$, as desired.
\end{proof}

  Using Theorem \ref{thm:prime3}, we next  show that  when ${\bf B}_{\calH}$ contains infinitely many non-isomorphic almost prime $3$-graphs, the  hypotheses of Theorem \ref{thm:vcfar3} are satisfied.   
 
 \begin{proposition}\label{prop:equiv3}
Suppose $\calH$ is a hereditary $3$-graph property and ${\bf B}_{\calH}$ contains infinitely many non-isomorphic almost prime $3$-graphs.
Then for all $m\geq 1$, ${\bf B}_{\calH}\cap \widehat{\Irr(m)}\neq \emptyset$.  
\end{proposition}
\begin{proof}
Given $m\geq 1$, let $N=N(m)$ be as in Theorem \ref{thm:prime3}.  By assumption, ${\bf B}_{\calH}$ contains an almost prime $3$-graph $H$ on at least $N$ vertices. By Theorem \ref{thm:prime3}, $H$ contains an induced sub-$3$-graph $H'$ such that $H'\in  \widehat{\Irr(m)}$.  Since ${\bf B}_{\calH}$ is hereditary (by Fact \ref{fact:bhhgraphs}), $H'\in {\bf B}_{\calH}$. Thus ${\bf B}_{\calH}\cap  \widehat{\Irr(m)}\neq \emptyset$, as desired.
\end{proof}
 
Recall from Subsection \ref{ss:ap} that $\widehat{\Irr(m)}$ contains $3$-graphs which are not themselves almost prime. For this reason, it is not immediate  that the converse of Proposition \ref{prop:equiv3} holds.  However, we will be able to deduce this from our results later on (see Proposition \ref{prop:equiv4}).  

We can now combine Proposition \ref{prop:equiv3} with Theorem \ref{thm:vcfar3} to show that when ${\bf B}_{\calH}$ contains arbitrarily large almost prime $3$-graphs, $M_{\calH}$ is at least polynomial.  

\begin{corollary}\label{cor:vcfar3}
Suppose $\calH$ is a hereditary $3$-graph property and assume ${\bf B}_{\calH}$ contains infinitely many non-isomorphic almost prime $3$-graphs.  Then  $M_{\calH}(\e)\geq \Omega(\e^{-1/8})$.
\end{corollary}
\begin{proof}
Fix a hereditary $3$-graph property  $\calH$, and assume ${\bf B}_{\calH}$ contains infinitely many non-isomorphic almost prime $3$-graphs.   By Proposition \ref{prop:equiv3}, we have that for all $m\geq 1$, there is some $H\in {\bf B}_{\calH}\cap  \widehat{\Irr(m)}$. By Theorem \ref{thm:vcfar3}, $M_{\calH}(\e)\geq \Omega(\e^{-1/8})$.
\end{proof}

Our next task is to show that when ${\bf B}_{\calH}$ contains only finitely many almost prime $3$-graphs up to isomorphism, $M_{\calH}$ is constant.   We will need some auxiliary results about almost prime $3$-graphs. First, we will use the $3$-graph analogue of Observation \ref{ob:sizeAP}, which follows immediately from the fact all $\sim_H$-classes in an almost prime $3$-graph $H$ contain at most $3$ vertices (see Definition \ref{def:3graphirred}).

\begin{observation}\label{ob:irr3graphs}
Let $H=(V,E)$ be an almost prime $3$-graph and let $\ell$ be the number of $\sim_H$-classes in $H$. Then $\ell \leq |V(H)|\leq 3\ell$. 
\end{observation}

Our next lemma tells us that if a $3$-graph $H$ has exactly $\ell$-many $\sim_H$-classes, then we can find an almost prime induced sub-$3$-graph witnessing this. 

\begin{lemma}\label{lem:irr3graphs2}
Suppose $H=(V,E)$ is a $3$-graph and $\ell$ is the number of $\sim_H$-classes in $H$. Then $H$ contains an induced almost prime sub-$3$-graph $H'$ with $\ell$-many $\sim_{H'}$-classes.   
\end{lemma}

The proof of Lemma \ref{lem:irr3graphs2} is straightforward, although slightly harder than its graph theoretic analogue, Lemma \ref{lem:irrgraphs}. Indeed, one can construct the desired $H'$ by keeping one vertex from each $\sim_H$-class of size $1$, two vertices from each $\sim_H$-class of size $2$, and three vertices from any larger $\sim_H$-classes.  A formal proof appears in Appendix \ref{ss:ap2}. 

The last preliminary we need is a theorem showing the class of $3$-graphs with at most $C$-many $\sim$-classes can be characterized by finitely many forbidden sub-$3$-graphs.  To state this result (Theorem \ref{thm:irrgraphs2hg} below), we first define a set of special $3$-graphs.  We will use the same notation from the graph theoretic analogue (Definition \ref{def:fc2}). No confusion should arise from this choice, as we consider only $3$-graphs in this section.  The appearance of the Ramsey number in Definition \ref{def:fc} reflects the increased complexity of the $3$-graph setting.

\begin{definition}\label{def:fc}
Let $K$ be the $6$-color Ramsey number $R(3,3,3,3,3,3)$, and let $N=3K+9$. Recalling that $\calG^{(3)}$ denotes the class of finite $3$-graphs, define the following for each integer $t\geq 1$.
\begin{align*}
\calF_{t}=\{F\in \calG^{(3)}: \text{ $F$  is almost prime, $F$ has at least $t$-many $\sim_F$-classes, }&\\
 \text{ and $|V(F)|\leq tN$}\}&.
\end{align*}
\end{definition}

Observe that every element of $\calF_{t}$ has at most $tN$ vertices, so $\calF_{t}$ contains only finitely many non-isomorphic $3$-graphs.  

\begin{theorem}\label{thm:irrgraphs2hg}
For all integers $C\geq 1$, if 
\begin{align*}
 \calH_C=\{H\in \calG^{(3)}: \text{$H$ has at most $C$-many $\sim_H$-classes}\},
 \end{align*}
then $\calH_C=\Forb(\calF_{C+1})$, where $\calF_{C+1}$ is from Definition \ref{def:fc}.
\end{theorem}

The proof of Theorem \ref{thm:irrgraphs2hg} appears in Appendix \ref{ss:ap2}.  It is similar to, but significantly trickier than, the proof of the graph analogue, Theorem \ref{thm:irrgraphs2}.  We observe that it is very easy to characterize the class $\calH_C$ from Theorem \ref{thm:irrgraphs2hg} via infinitely many forbidden $3$-graphs: one simply forbids all $3$-graphs with more than $C$-many $\sim$-classes. The purpose of Theorem \ref{thm:irrgraphs2hg} is to show this can be done by forbidding \emph{finitely many} non-isomorphic $3$-graphs.  

We will use Theorem \ref{thm:irrgraphs2hg}  in conjunction with Theorem \ref{thm:blowupthm}, to show that, given an integer $C\geq 1$, $M_{\calH}(\e)\leq C$ if and only if $\calF_{C+1}\cap {\bf B}_{\calH}=\emptyset$ (see Theorem \ref{thm:bhfinite1} below). Since $\calF_{C+1}$ contains finitely many $3$-graphs up to isomorphism,  the condition $\calF_{C+1}\cap {\bf B}_{\calH}=\emptyset$ translates into finitely many forbidden $3$-graphs in $\calH$ (via the definition of ${\bf B}_{\calH}$).  Thus, Theorem \ref{thm:irrgraphs2hg} allows us to show that when $M_{\calH}(\e)$ is bounded above by a constant, there are finitely many forbidden $3$-graphs responsible.\footnote{One could instead characterize the class $\calH_C$ in Theorem \ref{thm:irrgraphs2hg} via an infinite collection of forbidden sub-$3$-graphs, then apply the more general version of Theorem \ref{thm:blowupthm} for infinite families of forbidden sub-$3$-graphs (see the discussion following the statement of Theorem \ref{thm:blowupthm}).  This approach would characterize the constant growth classes via infinitely many forbidden $3$-graphs, rather than finitely many.  }

We now prove that if ${\bf B}_{\calH}$ avoids $\calF_{C+1}$ from Theorem \ref{thm:irrgraphs2hg}, then $M_{\calH}(\e)\leq C$. 

\begin{theorem}\label{thm:bhfinite1}
Suppose $C\geq 1$ is an integer. Then for any hereditary $3$-graph property  $\calH$ satisfying ${\bf B}_{\calH}\cap \calF_{C+1}=\emptyset$,  $M_{\calH}(\e)\leq M_{\calH}^{\hom}(\e^4)\leq C$. 
\end{theorem}
\begin{proof}
Let $\calH_C=\Forb(\calF_{C+1})$. By Theorem \ref{thm:irrgraphs2hg}, $\calH_C$ is equal to the class of all finite $3$-graphs $H$ with at most $C$-many $\sim_H$-classes. Since ${\bf B}_{\calH}\cap \calF_{C+1}=\emptyset$, Theorem \ref{thm:blowupthm} implies $\calH$ is close to $\calH_C$.

Fix $\e>0$ sufficiently small compared to $C^{-1}$. We show $M^{\hom}_{\calH_C}(\e)\leq C$. Let $H=(U, E)$ be a sufficiently large element of $\calH_C$.  Since $H\in \calH_C$, $H$ has at most $C$-many $\sim_{H}$-classes.  Let  $U_1,\ldots, U_t$ enumerate the $\sim_{H}$-classes, where $t\leq C$.  One can now show $\{U_1,\ldots, U_t\}$  is an $\e$-homogeneous partition with respect to $H$ via a similar argument to that at the end of the proof of Lemma \ref{lem:finitegraph}.  The arguments are sufficiently close  that we omit the details here.  From this we can conclude $M^{\hom}_{\calH_C}(\e)\leq C$. 

Since $\calH$ is close to $\calH_C$, Proposition \ref{prop:closehom} implies $M^{\hom}_{\calH}(\e)\leq M_{\calH_C}^{\hom}(\e/2)\leq C$.   Combining with  Fact \ref{lem:homup3graphs}, we have that for all sufficiently small $\e$, $M_{\calH}(\e)\leq M_{\calH}^{\hom}(\e^4)\leq C$.   
\end{proof}

We next show that $M_{\calH}(\e)$ can be lower bounded by the number of $\sim$-classes appearing in an  almost prime element of ${\bf B}_{\calH}$.  

\begin{theorem}\label{thm:bhfinite2}
Suppose $C\geq 1$ is an integer and $H=(U,E)$ is an almost prime $3$-graph with $C$-many $\sim_H$-classes.  If $\calH$ is a hereditary $3$-graph property and $H\in {\bf B}_{\calH}$, then for all sufficiently small $\e>0$, $M_{\calH}(\e)\geq C$. 
\end{theorem}
\begin{proof}
Let $U_1,\ldots, U_C$ enumerate the $\sim_H$-classes of $H$.    Since $H$ is almost prime, each $U_i$ has size $1$, $2$, or $3$, and $C\leq |V(H)|\leq 3C$ (by Observation \ref{ob:irr3graphs}).  Fix $\e$ sufficiently small compared to $C^{-1}$ and $n$ sufficiently large.

Since $H\in {\bf B}_{\calH}$, there exists an $n$-blowup ${\bf H}$ of $H$ with ${\bf H}\in \calH$.  We may assume ${\bf H}$ has vertex set $V({\bf H})=\bigcup_{u\in U}V_u$, where $|V_u|=n$ for all $u\in U$, and edge set $E({\bf H})$ satisfying
$$
\bigcup_{u_1u_2u_3\in E(H)}K_3[V_{u_1},V_{u_2},V_{u_3}]\subseteq E({\bf H})\text{ and }\left(\bigcup_{u_1u_2u_3\in {U\choose 3}\setminus E(H)}K_3[V_{u_1},V_{u_2},V_{u_3}]\right)\cap E({\bf H})=\emptyset.
$$
Note $Cn\leq |V({\bf H})|\leq 3Cn$. For each $1\leq i\leq C$, set ${\bf U}_i=\bigcup_{u\in U_i}V_u$. Observe that $V({\bf H})={\bf U}_1\sqcup \ldots \sqcup {\bf U}_C$, and each ${\bf U}_i$ has size $n$, $2n$, or $3n$.  

 Set $X=Y=V(H)$, and note ${\bf H}$ is an $(n,n; X,Y)$-blowup of $H$. We now check the hypotheses of Lemma \ref{lem:12blowup} hold for $H$ and ${\bf H}$ with respect to the sets $X$ and $Y$ and the parameter $\e$.  Clearly $X=Y$ and $n|X|=n|Y|$.   Since $\e$ is sufficiently small compared to $C^{-1}$ and $|Y|=|V(H)|\leq 3C$, we have $\e<(\frac{1}{16|Y|})^8$. This finishes our verification of the hypotheses of Lemma \ref{lem:12blowup}.  
 
 Assume $\calP$ is an  $\e$-regular partition of ${\bf H}$. Lemma \ref{lem:12blowup} implies the following holds, where ${\bf X}=\bigcup_{x\in X}V_x=V({\bf H})$: there exists $\calP'\subseteq \calP$ such that 
\begin{align}\label{al:pxlb}
\left| \bigcup_{P\in \calP'}P\cap {\bf X}\right|=\left|\bigcup_{P\in \calP'}P\right|\geq (1-2\e^{1/8})|V({\bf H})|,
\end{align}
and such that for all $P\in \calP'$, there is $1\leq g(P)\leq C$ such that  
\begin{align}\label{al:pubound}
\Big| P\cap {\bf U}_{g(P)}\Big|\geq (1-\e^{1/4})|P|.
\end{align} 
Combining (\ref{al:pxlb}), (\ref{al:pubound}),  the fact each $U_i$ has size at least $1$, and the fact each ${\bf U}_i$ has size  $n|U_i|$, we have the following.
\begin{align*}
\nonumber n| [C]\setminus \{g(P): P\in \calP'\}|\leq n\left|\bigcup_{i\in [C]\setminus \{g(P): P\in \calP'\}}U_i\right|&= \left| \bigcup_{i\in [C]\setminus\{g(P): P\in \calP'\}}{\bf U}_i\right|\\
&\leq \left|V({\bf H})\setminus \left(\bigcup_{P\in \calP'}P\right)\right| + \sum_{P\in \calP'}|P\setminus {\bf U}_{g(P)}|\\
&\leq   2\e^{1/8}|V({\bf H})| + \sum_{P\in \calP'}\e^{1/4}|P| \\
&\leq (2\e^{1/8}+\e^{1/4})|V({\bf H})|\\
&\leq (2\e^{1/8}+\e^{1/4})3Cn,
\end{align*}
where the last inequality uses that $|V({\bf H})|\leq 3Cn$.  Canceling $n$ yields 
\begin{align}\label{al:pxlb1}
| [C]\setminus \{g(P): P\in \calP'\}|\leq ( 2\e^{1/8}+\e^{1/4})3C<1,
\end{align}
where the last inequality uses that   $\e$ is sufficiently small compared to $C^{-1}$.  Since the left-hand side of (\ref{al:pxlb1}) must be a non-negative integer, it must be $0$.  Thus $|\{g(P): P\in \calP'\}|=C$, and consequently, $|\calP|\geq |\calP'|\geq C$.  This shows $M_{\calH}(\e)\geq C$.
\end{proof}

Using the results above, we next show that when ${\bf B}_{\calH}$ contains finitely many almost prime $3$-graphs up to isomorphism, $M_{\calH}$ is a constant equal to the maximum number of $\sim$-classes appearing in any element of ${\bf B}_{\calH}$.

\begin{theorem}\label{thm:bhfinite3}
 Suppose $\calH$ is a hereditary $3$-graph property and ${\bf B}_{\calH}$ contains finitely many almost prime $3$-graphs  up to isomorphism.   Then there exists a positive integer $C$ such that for all sufficiently small $\e>0$,  
 \begin{align*}
C= M_{\calH}(\e)=M_{\calH}^{\hom}(\e)&= \max\{\ell\in \mathbb{N}^{\geq 1}: \text{ there is }H\in {\bf B}_{\calH}\text{ with $\ell$-many $\sim_H$-classes}\}\\
 &=\max\{\ell\in \mathbb{N}^{\geq 1}: {\bf B}_{\calH}\cap \calF_{\ell}\neq \emptyset\},
\end{align*}
where $\calF_{\ell}$ is as in Definition \ref{def:fc}. 
\end{theorem}
\begin{proof}
We will use throughout the proof that ${\bf B}_{\calH}$ is a hereditary $3$-graph property (by Fact \ref{fact:bhhgraphs}), and is thus nonempty and closed under induced sub-$3$-graphs. We begin by defining three subsets of $\mathbb{N}^{\geq 1}$.
\begin{align*}
A_1&=\{\ell\in \mathbb{N}^{\geq 1}: \text{ there is an almost prime }H\in {\bf B}_{\calH} \text{ with $\ell$-many $\sim_H$-classes}\}.\\
A_2&=\{\ell\in \mathbb{N}^{\geq 1}: \text{ there is some }H\in {\bf B}_{\calH} \text{ with $\ell$-many $\sim_H$-classes}\}.\\
A_3&=\{\ell \in \mathbb{N}^{\geq 1}: {\bf B}_{\calH}\cap \calF_{\ell}\neq \emptyset\}.
\end{align*}
 Since ${\bf B}_{\calH}$ is  nonempty, $A_2\neq \emptyset$.  Lemma \ref{lem:irr3graphs2} and the fact ${\bf B}_{\calH}$ is hereditary imply $A_1=A_2$.  By assumption, ${\bf B}_{\calH}$ contains only finitely many almost prime $3$-graphs  up to isomorphism, so $A_1=A_2$ is finite. Thus, the following integer is well defined.
 \begin{align*}
C =\max A_1=\max A_2.  
\end{align*}
Given any integer $t\geq 1$, all elements of $\calF_t$ have at least $t$-many $\sim$-classes by definition. Since $C=\max A_2$, we can deduce that ${\bf B}_{\calH}\cap \calF_t=\emptyset$ for all $t>C$.  Thus $A_3\subseteq \{1,\ldots, C\}$.  Note that by definition, $\calF_1$ contains the trivial $3$-graph $H_{\text{triv}}$ with one vertex.  Since ${\bf B}_{\calH}$ is hereditary, $H_{\text{triv}}\in {\bf B}_{\calH}\cap \calF_1$, and consequently, $1\in A_3$.  These observations tell us the integer $L=\max A_3$ is well defined and satisfies $1\leq L\leq C$.  If $C=1$, we can immediately conclude $C=L=1$.  If $C>1$, then $C=\max A_2$ and Theorem \ref{thm:irrgraphs2hg} imply ${\bf B}_{\calH}\cap \calF_C\neq \emptyset$. This yields $L\geq C$, and thus, $C=L$.

Since ${\bf B}_{\calH}\subseteq \Forb(\calF_{C+1})$, Theorem \ref{thm:bhfinite1}  implies $M_{\calH}(\e)\leq M_{\calH}^{\hom}(\e^4)\leq C$.  On the other hand, since $C=\max A_1$, there exists an almost prime $H\in {\bf B}_{\calH}$ with $C$ many $\sim_H$-classes, so by Theorem \ref{thm:bhfinite2}, $M_{\calH}(\e)\geq C$.  Combining everything together, we have that for all sufficiently small $\e$, $M_{\calH}(\e)=M_{\calH}^{\hom}(\e)=C=L$.  
\end{proof}

We can now strengthen Proposition \ref{prop:equiv3} to an equivalence. 

\begin{proposition}\label{prop:equiv4}
Suppose $\calH$ is a hereditary $3$-graph property. Then the following are equivalent.
\begin{enumerate}
\item ${\bf B}_{\calH}$ contains infinitely many non-isomorphic almost prime $3$-graphs.
\item For all $m\geq 1$, ${\bf B}_{\calH}\cap \widehat{\Irr(m)}\neq \emptyset$.  
\end{enumerate}
\end{proposition}
\begin{proof}
That (1) implies (2) was proved in Proposition \ref{prop:equiv3}.  Assume now (2) holds.  By Theorem \ref{thm:vcfar3}, $M_{\calH}(\e)\geq \Omega(\e^{-1/8})$.  Since $M_{\calH}$ is not asymptotically a constant function,  Theorem \ref{thm:bhfinite3} implies  (1) holds. 
\end{proof}

We now put things together to prove our main theorem, Theorem \ref{thm:weakhom}.

\vspace{2mm}

\noindent{\bf Proof of Theorem \ref{thm:weakhom}.}
Fix $\calH$ a hereditary $3$-graph property.  Suppose first $\calH$ is far from finite SVC-dimension. By Theorem \ref{thm:sl1}, 
$$
\Tw(\Omega(\e^{-1}))\leq M_{\calH}(\e)\leq \Tw(6\e^{-4}),
$$
 so (1) holds.  Suppose now that $\calH$ is close to finite SVC-dimension.   If $\calH$ is far from finite VC-dimension,  then by Theorem \ref{thm:exp}(2), there is $C>0$ such that 
 $$
 2^{\Omega(\e^{-1/8})}\leq M_{\calH}(\e)\leq M_{\calH}^{\hom}(\e^4)\leq 2^{2^{\e^{-C}}},
 $$
  so (2) holds.  Suppose now $\calH$ is close to finite VC-dimension. If ${\bf B}_{\calH}$ contains infinitely many non-isomorphic almost prime $3$-graphs, then Corollary \ref{cor:vcfar3} and Theorem \ref{thm:exp}(1) imply there is $C>0$ such that
$$
\Omega(\e^{-1/8})\leq M_{\calH}(\e)\leq M_{\calH}^{\hom}(\e^4)\leq \e^{-C},
$$
so (3) holds. Finally, if ${\bf B}_{\calH}$ contains finitely many non-isomorphic almost prime $3$-graphs, then by Theorem  \ref{thm:bhfinite3}, there is a constant $C$ so that $M_{\calH}(\e)=M_{\calH}^{\hom}(\e)=C$ for all sufficiently small $\e>0$, so    (4) holds. 
\qed

\appendix

\section{Proof of Lemma \ref{lem:slicingcor} }\label{ss:standardlemmas} 
 In this section, we prove  Lemma \ref{lem:slicingcor} (restated below as Lemma \ref{lem:slicingapp}).   To prove this result, we require the following ``slicing lemma," which says that large sub-triples of regular triples are still somewhat regular.  
 
\begin{proposition}\label{prop:slicing3} Suppose $0<\e\leq \gamma<1$. Let $H=(V , E)$ be a $3$-graph,  let $A,B,C\subseteq V$ be nonempty sets, and assume $(A,B,C)$ is $\e$-regular with respect to $H$.  For any $A'\subseteq A$,  $B'\subseteq B$, and $C'\subseteq C$ satisfying $ |A'| \geq \gamma |A|$, $|B'| \geq \gamma|B|$, and $|C'| \geq \gamma|C|$, the triple $(A',B',C')$ is $2\gamma^{-1}\e $-regular with respect to $H$ and $|d_H(A',B',C')-d_H(A,B,C)|\leq \e$.  
\end{proposition}
 
We leave the proof of Proposition \ref{prop:slicing3} as an exercise for the reader, as it is essentially identical to the well-known analogue in the graphs case (see Lemma 3.1 in \cite{Alon.2000}).  

\begin{lemma}\label{lem:slicingapp}
Let $\e\in (0,1)$, let $H=(V,E)$ be a $3$-graph, and let $\calP$ be an $\e$-regular partition of $H$.  Suppose  $\calP'$ is a partition of $V$ with at most $2$ parts, and let $\calQ$ be the common refinement of $\calP$ with $\calP'$.   Then $|\bigcup_{(X,Y,Z)\in \Sigma}X\times Y\times Z|\geq (1-7\e^{1/2})|V|^3$, where 
$$
\Sigma=\{(X,Y,Z)\in \calQ^3: (X,Y,Z)\text{ is $2\e^{1/2}$-regular with respect to $H$}\}.
$$ 
\end{lemma}
\begin{proof}
Fix an enumeration $\calP=\{X_1,\ldots, X_t\}$ and let 
$$
\Sigma_{\reg}=\{(X_i,X_j,X_k)\in \calP^3: (X_i,X_j,X_k)\text{ is $\e$-regular with respect to $H$}\}.
$$
Since $\calP$ is $\e$-regular, $|\bigcup_{(X_i,X_j,X_k)\in \Sigma_{\reg}}X_i\times X_j\times X_k|\geq (1-\e)|V(H)|^3$.  Since $\calP'$ has size at most $2$, for each $i\in [t]$, we can write $X_i=X_i^1\cup X_i^2$, where each of $X_i^1,X_i^2$ is either in $\calQ$ or equal to the empty set.  By definition, $\calQ$ consists of all nonempty sets of the form $X_i^{\alpha}$.   We next set aside the elements in $\calQ$ which have shrunk too much compared to the original sets in $\calP$.  In particular, define
$$
\calQ_{\sm}=\{X_i^{\alpha}:i\in [t],1\leq \alpha\leq 2,\text{ and } 0<|X_i^{\alpha}|< \e^{1/2}|X_i|\}.
$$
Observe that 
\begin{align}\label{al:slicingcor1}
\sum_{Q\in \calQ_{\sm}}|Q|\leq \sum_{i=1}^t 2\e^{1/2}|X_i|=2\e^{1/2}|V(H)|.
\end{align}
We next consider the set of triples from $\calQ^3$ which come from a triple in $\Sigma_{\reg}$, and which avoid $\calQ_{\sm}$.  Specifically, we set
$$
\Sigma_{\reg}'=\{(X_i^{\alpha_1},X_j^{\alpha_2},X_k^{\alpha_3})\in \calQ^3: (X_i,X_j,X_k)\in \Sigma_{\reg}\text{ and }X_i^{\alpha_1},X_j^{\alpha_2},X_k^{\alpha_3}\in \calQ\setminus \calQ_{\sm}\}.
$$
By Proposition \ref{prop:slicing3}, every $(X_i^{\alpha_1},X_j^{\alpha_2},X_k^{\alpha_3})\in \Sigma_{\reg}'$ is $2\e^{1/2}$-regular with respect to $H$.   Thus, it suffices to show $|\bigcup_{(X,Y,Z)\in \Sigma_{\reg}'}X\times Y\times Z|\geq (1-7\e^{1/2})|V(H)|^3$. Observe,
\begin{align*}
\left|\bigcup_{(X,Y,Z)\in \Sigma_{\reg}'}X\times Y\times Z\right|&\geq \left|\bigcup_{(X,Y,Z)\in \Sigma_{\reg}}X\times Y\times Z\right|-3|V(H)|^2\sum_{Q\in \calQ_{\sm}}|Q|\\
&\geq (1-\e)|V(H)|^3 -6\e^{1/2}|V(H)|^3\\
&\geq (1-7\e^{1/2})|V(H)|^3,
\end{align*}
where the second inequality uses (\ref{al:slicingcor1}) and the fact that $\calP$ is $\e$-regular.  
\end{proof}

\section{Proofs of Theorems \ref{thm:irrgraphs2} and \ref{thm:irrgraphs2hg}}

In this section, we first prove Theorem \ref{thm:irrgraphs2}, which shows the class of finite graphs with at most $C$-many $\sim$-classes is characterized by omitting a certain finite collection of induced subgraphs.  We then prove the $3$-graph companion, Theorem \ref{thm:irrgraphs2hg}.  

\subsection{Proof of Theorem \ref{thm:irrgraphs2}}\label{ss:ap1}

We will require several new auxiliary lemmas in addition to those stated in Subsection \ref{ss:contpolygraphs}. First, it is easy to show that the equivalence relation $\sim_G$ on a graph $G$  coarsens when passing to induced subgraphs.

\begin{lemma}\label{lem:coarsen}
Let $G=(V,E)$ be a graph, and let $U_1,\ldots, U_t$ be an enumeration of the $\sim_G$-classes in $G$.  Assume $G'=(V',E')$ is an induced subgraph of $G$.  Then every $\sim_{G'}$-class in $G'$ is a union of sets from $\{V'\cap U_i: 1\leq i\leq t\}$.  
\end{lemma} 
 \begin{proof}
It suffices to show that if $v,v'\in V'$ and $v\sim_Gv'$, then $v\sim_{G'}v'$.  This is immediate from Definition \ref{def:sim}.
\end{proof}

Next, we prove Lemma \ref{lem:irrgraphs} (repeated as Lemma \ref{lem:irrgraphsrepeat} below), which says that we can always find an almost prime induced subgraph of $G$ with the same number of equivalence classes as $G$.

\begin{lemma}\label{lem:irrgraphsrepeat}
Let $G=(V,E)$ be a graph and let $\ell$ be the number of $\sim_G$-classes in $G$. Then  $G$ contains an almost prime induced subgraph $G'$ with $\ell$-many $\sim_{G'}$-classes.   
\end{lemma}
\begin{proof}
Let $U_1,\ldots, U_{\ell}$ enumerate the $\sim_G$-classes of $G$.  For each $1\leq i\leq \ell$, define a subset $U'_i\subseteq U_i$ as follows.  If $|U_i|=1$, set $U_i'=U_i$. Otherwise, let $U_i'$ be any two-element subset of $U_i$.   Since each $U_i'$ has size at most $2$ by construction, it suffices to show  $U_1',\ldots, U_{\ell}'$  are exactly the $\sim_{G'}$-classes of $G':=G[U_1'\cup \ldots \cup U_{\ell}']$.  We will use below that for each $1\leq i\leq \ell$, if $|U_i'|<|U_i|$, then $|U_i'|=2$.

By Lemma \ref{lem:coarsen}, we know that each $U_i'$ is contained in a single $\sim_{G'}$-class.  Thus it suffices to show that for each $1\leq i\neq j\leq \ell$, and each $u\in U_i'$ and $u'\in U_j'$, we have $u\nsim_{G'}u'$. 

To this end, fix $1\leq i\neq j\leq \ell$, along with $u\in U'_i$ and $u'\in U'_j$.  Since $u\nsim_G u'$,  there is some $v\in (N_G(u)\Delta N_G(u'))\setminus \{u,u'\}$.  By relabeling if necessary, we may assume $uv\in E$ and $u'v\notin E$.  If $v\in V(G')$, this immediately implies $u\nsim_{G'}u'$, and we are done.  Assume now $v\notin V(G')$.  
\begin{claim}\label{cl:irrgraphs1}
There is $x\in V(G')\setminus \{u,u'\}$ such that $x\sim_G v$. 
\end{claim}
\begin{proof} Let $1\leq k\leq \ell$ be such that $v\in U_k$.  Since $v\notin V(G')$, $|U_k'|<|U_k|$, so it must be the case that $|U_k'|=2$.  Since $u\nsim_G u'$, at most one of $u$ or $u'$ can be in $U'_k$, so $|U_k'\setminus \{u,u'\}|\geq 1$. We can now take $x$ to be any element of $U_k'\setminus \{u,u'\}$.  
\end{proof}
We can now quickly finish the proof of Lemma \ref{lem:irrgraphsrepeat}. Let $x$ be as in Claim \ref{cl:irrgraphs1}.  Since $x\sim_G v$ and $\{u,u'\}\cap  \{x,v\}=\emptyset$, we have $ux\in E$ and $u'x\notin E$ (since $uv\in E$ and $u'v\notin E$). This proves $u\nsim_{G'}u'$, as desired.  
\end{proof}

Our next lemma shows that deleting a $\sim_G$-class from a graph $G$ cannot decrease the total number of equivalence classes by too much.  In our proof, we will use Notation \ref{not:01nbrs}. 

\begin{lemma}\label{lem:deletesim}
Suppose $\ell \geq 4$ is an integer, $G=(U,E)$ is a graph, and $U_1,\ldots, U_{\ell}$ is any enumeration of the $\sim_G$-classes of $G$.  Then $G'=G[U\setminus U_1]$ has at least $(\ell-1)/2$-many $\sim_{G'}$-classes.
\end{lemma}
\begin{proof}
Let $V_1,\ldots, V_t$ enumerate the $\sim_{G'}$-classes of $G'$, and suppose towards a contradiction $t<(\ell-1)/2$. By Lemma \ref{lem:coarsen} and the definition of $G'$, each $V_i$ is a union of sets from $\{U_2,\ldots, U_{\ell}\}$.  Since $t<(\ell-1)/2$, we have by the Pigeonhole Principle that for some $1\leq i_*\leq t$, $V_{i_*}$ contains at least three distinct $\sim_G$-classes.  Let $2\leq \alpha_1,\alpha_2,\alpha_3\leq \ell$ be pairwise distinct indices such that $U_{\alpha_1}\cup U_{\alpha_2}\cup U_{\alpha_3}\subseteq V_{i_*}$.  

Fix $u_{\alpha_1}\in U_{\alpha_1}$, $u_{\alpha_2}\in U_{\alpha_2}$, and $u_{\alpha_3}\in U_{\alpha_3}$. For each $1\leq i<j\leq 3$, we have $u_{\alpha_i}\nsim_Gu_{\alpha_j}$. Consequently, there exist  $a_{ij}\in V(G)\setminus \{u_{\alpha_i},u_{\alpha_j}\}$ and $\tau(ij)\in \{0,1\}$ such that $a_{ij}u_{\alpha_i}\in E^{\tau(ij)}$ and such that $a_{ij}u_{\alpha_j}\in E^{1-\tau(ij)}$.  Since $u_{\alpha_i}\sim_{G'}u_{\alpha_j}$, we must have $a_{ij}\in V(G)\setminus V(G')$.  Let $A=\{a_{ij}: 1\leq i<j\leq 3\}$. By the preceding observation, $A\cap V(G')=\emptyset$.  By definition of $G'$, this implies $A\subseteq U_1$.  Since $u_{\alpha_1},u_{\alpha_2},u_{\alpha_3}\in V(G')$, this implies $A\cap \{u_{\alpha_1},u_{\alpha_2},u_{\alpha_3}\}=\emptyset$. 

Since $A$ is contained in a single $\sim_G$-class and $A\cap \{u_{\alpha_1},u_{\alpha_2},u_{\alpha_3}\}=\emptyset$, we have that for each $1\leq i\leq 3$,  there is $\sigma(i)\in \{0,1\}$ such that $u_{\alpha_i} a\in E^{\sigma(i)}$ for all $a\in A$.  By the Pigeonhole Principle, there are $1\leq i< j\leq 3$ such that $\sigma(i)=\sigma(j)$. This means $a_{ij}u_{\alpha_i}$ and $a_{ij}u_{\alpha_j}$ are both in $E^{\sigma(i)}=E^{\sigma(j)}$. However, this is impossible, since either $\sigma(i)=\tau(ij)$ and $a_{ij}u_{\alpha_j}\notin E^{\sigma(i)}$, or $\sigma(i)=1-\tau(ij)$ and $a_{ij}u_{\alpha_i}\notin E^{\sigma(i)}$.
\end{proof}

Our next lemma says that any graph $G$ with more than $C$-many $\sim_G$-classes contains a small almost prime induced subgraph witnessing this. 

\begin{lemma}\label{lem:irrgraphs2}
Let $C\geq 1$ be an integer. Suppose $G=(V,E)$ is a graph with more than $C$-many $\sim_G$-classes.  Then there exists an almost prime induced subgraph $G'$ of $G$, such that $G'$ has more than $C$-many $\sim_{G'}$-classes and such that $|V(G')|\leq 4C+2$.   
\end{lemma}
\begin{proof}
Let $\ell$ be the number of $\sim_G$-classes in $G$.  We first construct an integer $t\geq 1$ and sequence of graphs $G_0,\ldots, G_t$ and integers $\ell_0,\ldots, \ell_t$ via an inductive process as follows.  

Step 0: Set $G_0=G$ and let $\ell_0=\ell$.  If $\ell_0\leq 2C+1$,  set $t=0$ and end the construction. Otherwise, we have $2C+1<\ell_0\leq \ell-0$. Go to the next step.

Step $m+1$: Suppose $m\geq 0$, and assume by induction we have constructed an induced subgraph $G_m$ of $G$ with $\ell_m$-many $\sim_{G_m}$-classes, for some $\ell_m$ satisfying $2C+1<\ell_m\leq \ell-m$.  Let $U$ be any $\sim_{G_m}$-class of $G_m$, set $G_{m+1}=G_m[V(G_m)\setminus U]$, and let $\ell_{m+1}$ be the number of $\sim_{G_{m+1}}$-classes.  The definition of $G_{m+1}$ and Lemma \ref{lem:coarsen} imply $\ell_{m+1}\leq \ell_m-1\leq \ell-(m+1)$, where the last inequality is by our induction hypothesis $\ell_m\leq \ell-m$. On the other hand, by Lemma \ref{lem:deletesim}, $\ell_{m+1}\geq (\ell_m-1)/2>C$, where the last inequality is by our induction hypothesis $2C+1<\ell_m$.  If $\ell_{m+1}\leq 2C+1$, set $t=m+1$ and end the construction. Otherwise, we have $2C+1<\ell_{m+1}\leq \ell-(m+1)$. Go to the next step.  

This construction will clearly halt after at most $\ell-2C$ steps.  At the end, we will have found $G_t$, an induced subgraph of $G$, with $\ell_t$-many $\sim_{G_t}$-classes, for some $\ell_t$ satisfying $C<\ell_t\leq 2C+1$.  By Lemma \ref{lem:irrgraphsrepeat}, $G_t$ contains an induced, almost prime subgraph $G'$ with $\ell_t$-many $\sim_{G'}$-classes. By Observation \ref{ob:sizeAP}, $|V(G')|\leq 2\ell_t\leq 2(2C+1)=4C+2$. Thus $G'$ satisfies the desired conclusions. 
\end{proof}

 We now prove Theorem \ref{thm:irrgraphs2}, which we repeat below for the convenience of the reader.
 
\begin{theorem}\label{thm:irrgraphs2repeat}
Let $C\geq 1$ be an integer.  Recalling $\calG^{(2)}$ denotes the class of all finite graphs, define
$$
\calH_C=\{G\in \calG^{(2)}: G\text{ has at most $C$-many $\sim_G$-classes}\}.
$$
Then $\calH_C=  \Forb(\calF_{C+1})$, where $\calF_{C+1}$ is from Definition \ref{def:fc2}.  
\end{theorem}
\begin{proof}
 Suppose $G$ is a finite graph and $G\notin \Forb(\calF_{C+1})$. Then $G$ contains an element $F$ of $\calF_{C+1}$ as an induced subgraph. By definition of $\calF_{C+1}$, $F$ has at least $(C+1)$-many $\sim_F$-classes.  Since $F$ is an induced subgraph of $G$, Lemma \ref{lem:coarsen} implies  $G$ must have at least $(C+1)$-many $\sim_G$-classes, so $G\notin \calH_C$.

Suppose now $G$ is a finite graph and $G\notin \calH_C$.  By definition of $\calH_C$, $G$ has more than $C$-many $\sim_G$-classes.  By Lemma \ref{lem:irrgraphs2}, $G$ contains an almost prime induced subgraph $G'$ satisfying $|V(G')|\leq 4C+2$ and such that $G'$ has more than $C$-many $\sim_{G'}$-classes.  Setting $K=C+1$, we see that $G'$ has at least $K$-many $\sim_{G'}$-classes, and $|V(G')|\leq 4C+2\leq 4K$.  Thus $G'\in \calF_{K}=\calF_{C+1}$.  Since $G'$ is an induced subgraph of $G$, this implies $G\notin \Forb(\calF_{C+1})$.  
\end{proof}

\subsection{Proof of Theorem \ref{thm:irrgraphs2hg}}\label{ss:ap2}

In this subsection, we prove the $3$-graph analogue of Theorem \ref{thm:irrgraphs2}. This requires several lemmas in analogy to the preceding subsection.  We will use Notation \ref{not:01nbrs}  throughout.  

First, we will use the fact that, given a $3$-graph $H$, and an  induced sub-$3$-graph $H'$ of $H$, every $\sim_{H'}$-class in $H'$ is a union of intersections of $\sim_H$-classes with $V(H')$.

\begin{lemma}\label{lem:unionclasses}
Suppose $H=(V,E)$ is a $3$-graph and $U_1,\ldots, U_t$ is an enumeration of the $\sim_H$-classes in $H$.  Assume $H'=(V',E')$ is an induced sub-$3$-graph of $H$. Then every $\sim_{H'}$-class  is a union of sets from $\{U_i\cap V': 1\leq i\leq t\}$.
\end{lemma}

We leave the proof of Lemma \ref{lem:unionclasses} as an exercise, as it is very similar to Lemma \ref{lem:coarsen}.  We now prove Lemma \ref{lem:irr3graphs2}, repeated below for convenience. 

\begin{lemma}\label{lem:irr3graphs2repeat}
Suppose $H=(V,E)$ is a $3$-graph and $\ell$ is the number of $\sim_H$-classes in $H$. Then $H$ contains an induced almost prime sub-$3$-graph $H'$ with $\ell$-many $\sim_{H'}$-classes.   
\end{lemma}
\begin{proof}
Let $U_1,\ldots, U_{\ell}$ enumerate the $\sim_H$-classes of $H$.  We define a subset $U'_i$ of $U_i$ for each $1\leq i\leq \ell$ as follows.  If $|U_i|\leq 2$, set $U_i'=U_i$, and otherwise, let $U_i'$ be any three-element subset of $U_i$.  Note that for any $i\in [\ell]$, if $|U_i'|<|U_i|$, then $|U_i'|=3$.

Since each $U_i'$ has size at most $3$, it suffices to show $U_1',\ldots, U_{\ell}'$ are exactly the $\sim_{H'}$-classes of $H':=H[U_1'\cup \ldots \cup U_{\ell}']$.   By Lemma \ref{lem:unionclasses}, each $U_i'$ is contained in a single $\sim_{H'}$-class.  Thus it suffices to show that for each $1\leq i\neq j\leq \ell$ and each $u\in U'_i$ and $u'\in U_j'$, $u\nsim_{H'}u'$. 

Fix $1\leq i\neq j\leq \ell$, along with $u\in U'_i$ and $u'\in U'_j$.  Since $u\nsim_H u'$,  there exist $x\neq y\in V(H)\setminus \{u,u'\}$ and $\tau\in \{0,1\}$ such that $xyu\in E^{\tau}$ and $xyu'\in E^{1-\tau}$.

\begin{claim}\label{cl:findsim}There exist vertices $x'\neq y' \in V(H')\setminus \{u,u'\}$ satisfying $x'\sim_{H}x$  and  $y'\sim_Hy$.
\end{claim}
\begin{proof}
If $x$ and $y$ are both in $V(H')$, just let $x'=x$ and $y'=y$.  This leaves us with the case where at least one of $x$ or $y$ is in $V(H)\setminus V(H')$.  After possibly relabeling, we may assume $x\in V(H)\setminus V(H')$.

Let $1\leq k\leq \ell$ be such that $x\in U_k$.  Since $x\notin V(H')$, we have $|U_k'|<|U_k|$, so by construction, $|U'_k|=3$.  Since $u$ and $u'$ are in distinct $\sim_H$-classes, at most one of them is in $U_k'$, and consequently, $|U_k'\setminus \{u,u'\}|\geq |U_k'|-1\geq 2$.  Let $x_1$ and $x_2$ be two distinct elements of $U_k'\setminus \{u,u'\}$.   We now choose $x'$ and $y'$ in cases.  

If $y\sim_H x$, set $x'=x_1$ and $y'=x_2$. 

If $y\nsim_H x$ and $y\in V(H')$, set $y'=y$ and $x'=x_1$.  
 
We are left with the case where $y\nsim_H x$ and $y\in V(H)\setminus V(H')$. Let $1\leq s\leq \ell$ be such that $y\in U_s$.  Since $y\notin V(H')$, $|U_s'|<|U_s|$, so by construction, $|U'_s|=3$.  Since $s\neq k$, $y\in U_s$, and $x_1,x_2\in U_k$,  we have $U_s'\cap \{x_1,x_2\}=\emptyset$.  Combining with the fact $u\nsim_Hu'$, we have $|U_s'\setminus \{u,u',x_1,x_2\}|=|U_s'\setminus \{u,u'\}|\geq |U_s'|-1\geq 2$. We can now choose $y'$ to be any element of $U_s'\setminus \{u,u',x_1,x_2\}$, and set $x'=x_1$.
\end{proof}

We can now quickly finish the proof of Lemma \ref{lem:irr3graphs2repeat}. Let $x',y'$ be as in Claim \ref{cl:findsim}. By construction, $x'\neq y'\in V(H')\setminus \{u,u'\}$, and $x'\sim_H x$ and $y'\sim_H y$. Thus $x'y'u\in E^{\tau}$ and $x'y'u'\in E^{1-\tau}$.  This proves $u\nsim_{H'}u'$, as desired. 
\end{proof}

We next prove a $3$-graph analogue of Lemma \ref{lem:deletesim}, which will tell us that for any $3$-graph $H$, there exists a $\sim_H$-class we can delete without decreasing the number of $\sim$-classes too drastically.  This lemma is significantly more difficult than its graph theoretic analogue. 

 \begin{lemma}\label{lem:deletesimhg}
Let $K$ be the $6$-color Ramsey number $R(3,3,3,3,3,3)$, and let $L=K+2$.   Suppose $\ell>L+1$ is an integer, $H=(U,E)$ is a $3$-graph, and $U_1,\ldots, U_{\ell}$ is an enumeration of the $\sim_H$-classes in  $H$.  Then there exists some $1\leq i\leq \ell$ such that $H'=H[U\setminus U_i]$ has at least $(\ell-1)/L$-many $\sim_{H'}$-classes.
\end{lemma}
\begin{proof}
Suppose towards a contradiction there exists a $3$-graph $H=(U,E)$, with $\sim_H$-classes $U_1,\ldots, U_{\ell}$, such that for all $1\leq i\leq \ell$, $H_i=H[U\setminus U_i]$ has less than $(\ell-1)/L$-many $\sim_{H_i}$-classes.

Let $H_1=H[U\setminus U_1]$, and let $V_1,\ldots, V_t$ be an enumeration of the $\sim_{H_1}$-classes in $H_1$.   By definition of $H_1$ and Lemma \ref{lem:unionclasses},  each $V_i$ is a union of sets from $\{U_2,\ldots, U_{\ell}\}$.  By assumption, $t<(\ell-1)/L<(\ell-1)/2$, so by the Pigeonhole Principle, there is some $V_{i_*}$ containing two distinct $\sim_H$-classes.  After relabeling if necessary, let us assume that $U_{2}\cup U_{3} \subseteq V_{i_*}$.  

Now let $H_2=H[U\setminus U_2]$, and let $W_1,\ldots, W_s$ enumerate the $\sim_{H_2}$-classes in $H_2$.  By  definition of $H_2$ and Lemma \ref{lem:unionclasses}, each $W_j$ is a union of sets from $\{U_1,U_3,\ldots, U_{\ell}\}$.  By assumption, $s<(\ell-1)/L$, so by the Pigeonhole Principle, there must be some  $W_{j_{*}}$ containing at least $L$ many sets from $\{U_1,U_3,\ldots, U_{\ell}\}$.    Since $L-2\geq K$, $W_{j_{*}}$ contains at least  $K$ many sets from $\{U_1,U_3,\ldots, U_{\ell}\}\setminus  \{U_1,U_3\}$. Thus we may choose $\{\beta_1,\ldots, \beta_{K}\}\subseteq \{4,\ldots, \ell\}$ a set of $K$ distinct indices such that $U_{\beta_1}\cup \ldots \cup U_{\beta_{K}}\subseteq W_{j_*}$.   

For each $1\leq i\leq K$, fix $u_{\beta_i}\in U_{\beta_i}$.   For each $1\leq i<j\leq K$, we have $u_{\beta_i}\nsim_Hu_{\beta_j}$. Consequently, there exist $\tau_{ij}\in \{0,1\}$ and $a_{ij}\neq b_{ij}\in V(H)\setminus \{u_{\beta_i},u_{\beta_j}\}$ such that 
\begin{align}
\label{al:tij}&a_{ij}b_{ij}u_{\beta_i}\in E^{\tau_{ij}}\text{ and }\\
\label{al:tij2}&a_{ij}b_{ij}u_{\beta_j}\in E^{1-\tau_{ij}}.
\end{align}
 On the other hand, for each $1\leq i<j\leq K$, $u_{\beta_i}\sim_{H_2}u_{\beta_j}$, so (\ref{al:tij}) and (\ref{al:tij2}) imply one of $a_{ij}$ or $b_{ij}$ must be in $U_2$.  After relabeling if necessary, let us assume $a_{ij}\in U_{2}$ for all $1\leq i<j\leq K$.  After this relabeling, we have the following fact, which we will use later in the proof.
\begin{align}\label{al:aij}
\text{ For all $1\leq i<j\leq K$ and $1\leq i'<j'\leq K$, $a_{ij}\sim_Ha_{i'j'}$.}
\end{align}
We now show that at least one of the $b_{ij}$ must avoid $U_1$, $U_2$, and $U_3$.

\begin{claim}\label{cl:bint}
The set $B=\{ b_{ij}: 1\leq i<j\leq K \}$  intersects some $\sim_H$-class from $\{U_4,\ldots, U_{\ell}\}$. 
\end{claim}
\begin{proof} Suppose towards a contradiction $B\subseteq U_1\cup U_2\cup U_3$.  Then we can define a $6$-coloring of the edges of the complete graph on vertex set $[K]$ as follows.
\begin{align*}
{[K]\choose 2}&=\Big\{ij\in {[K]\choose 2}: i<j, \tau_{ij}=0, b_{ij}\in U_1\Big\} \cup \Big\{ij\in {[K]\choose 2}: i<j, \tau_{ij}=1, b_{ij}\in U_1\Big\} \\
&\cup \Big\{ij\in {[K]\choose 2}: i<j, \tau_{ij}=0, b_{ij}\in U_2\Big\} \cup \Big\{ij\in {[K]\choose 2}: i<j, \tau_{ij}=1, b_{ij}\in U_2\Big\} \\
&\cup \Big\{ij\in {[K]\choose 2}: i<j, \tau_{ij}=0, b_{ij}\in U_3\Big\} \cup \Big\{ij\in {[K]\choose 2}: i<j, \tau_{ij}=1, b_{ij}\in U_3\Big\}.
\end{align*}
By our choice of $K$, there exists a monochromatic triangle. In other words, there exist $\sigma\in \{0,1\}$, $\mu\in \{1,2,3\}$, and $1\leq i_1<i_2<i_3\leq K$ such that $\tau_{i_1i_2}=\tau_{i_1i_3}=\tau_{i_2i_3}=\sigma$ and such that $b_{i_1i_2},b_{i_1i_3},b_{i_2i_3}\in U_{\mu}$. By (\ref{al:tij}) and (\ref{al:tij2}),  this tells us
\begin{align}
\label{al:aibj1}&a_{i_1i_2}b_{i_1i_2}u_{\beta_{i_2}}\in  E^{1-\sigma}\text{ and }\\
\label{al:aibj2}&a_{i_2i_3}b_{i_2i_3}u_{\beta_{i_2}}\in  E^{\sigma}.
\end{align}
We also know by construction that
\begin{align}\label{al:u2}
\text{ $u_{\beta_{i_2}}\notin \{a_{i_1i_2},b_{i_1i_2},a_{i_2i_3},b_{i_2i_3}\}$, $a_{i_1i_2}\neq b_{i_1i_2}$, and $a_{i_2i_3}\neq b_{i_2i_3}$. }
\end{align}
Clearly (\ref{al:aibj1}) and (\ref{al:aibj2}) imply   $\{a_{i_1i_2},b_{i_1i_2}\}\neq \{a_{i_2i_3},b_{i_2i_3}\}$, and consequently, we have that either $a_{i_1i_2}\neq b_{i_2i_3}$ or $a_{i_2i_3}\neq b_{i_1i_2}$.

Suppose first $a_{i_1i_2}\neq b_{i_2i_3}$.  Then, recalling $a_{i_1i_2}\sim_H a_{i_2i_3}$ (by  (\ref{al:aij})), we have by (\ref{al:u2}) and  (\ref{al:aibj2}) that $a_{i_1i_2}b_{i_2i_3}u_{\beta_{i_2}}\in E^{\sigma}$.  However, since $a_{i_1i_2}b_{i_1i_2}u_{\beta_{i_2}}\in E^{1-\sigma}$ (by (\ref{al:aibj1})), this implies $b_{i_1i_2}\nsim_H b_{i_2i_3}$, contradicting that $b_{i_1i_2},b_{i_2i_3}\in U_{\mu}$. 
 
 We are left with the case that  $a_{i_2i_3}\neq b_{i_1i_2}$. In this case, $a_{i_1i_2}\sim_H a_{i_2i_3}$,  (\ref{al:u2}), and (\ref{al:aibj1}) imply $a_{i_2i_3}b_{i_1i_2}u_{\beta_{i_2}}\in E^{1-\sigma}$. However, since  $a_{i_2i_3}b_{i_2i_3}u_{\beta_{i_2}}\in E^{\sigma}$ (by (\ref{al:aibj2})), this implies $b_{i_1i_2}\nsim_H b_{i_2i_3}$, contradicting that $b_{i_1i_2},b_{i_2i_3}\in U_{\mu}$.  This completes our proof of Claim \ref{cl:bint}.
\end{proof}

By Claim \ref{cl:bint}, there exist $1\leq i<j\leq K$ such that  $b_{ij}\in V(H)\setminus( U_1\cup U_{2}\cup U_{3})$. Fix any $u_3\in U_3$.  We now observe a certain collection of vertices are distinct.  Since $b_{ij}\notin U_3$, we know $b_{ij}\neq u_3$.  Recalling that $a_{ij}\in U_2$, we know $a_{ij}\neq u_3$ is also true. Recall that by construction, $b_{ij},a_{ij},u_{\beta_i},u_{\beta_j}$ are pairwise distinct.  Finally,  $\beta_i,\beta_j>3$ holds by construction, so, since $u_{\beta_i}\in U_{\beta_i}$ and $u_{\beta_j}\in U_{\beta_j}$, we have that $u_{\beta_i},u_{\beta_j}, a_{ij},u_3$ are pairwise distinct.  Combining all these together, we see 
\begin{align}\label{al:vdis}
\text{the vertices $b_{ij},a_{ij},u_{\beta_i},u_{\beta_j}, u_3$ are pairwise distinct. }
\end{align}

We now recall that $U_2$ and $U_3$ are contained in a single $\sim_{H_1}$-class. Consequently, since $a_{ij}\in U_2$ and $u_3\in U_3$,  we have that $a_{ij}\sim_{H_1}u_3$. Thus, by (\ref{al:vdis}) and because $u_{\beta_i},u_{\beta_j},b_{ij}\in V(H)\setminus U_1$, we have that there exist $\rho,\rho'\in \{0,1\}$ such that $a_{ij}b_{ij}u_{\beta_i}, u_3b_{ij}u_{\beta_i}\in E^{\rho}$ and $ a_{ij}b_{ij}u_{\beta_j}, u_3b_{ij}u_{\beta_j}\in E^{\rho'}$.  By (\ref{al:tij}) and (\ref{al:tij2}),  $\rho=\tau_{ij}$ and $\rho'=1-\tau_{ij}$.  Thus $u_3b_{ij}u_{\beta_i}\in E^{\tau_{ij}}$ while $u_3b_{ij}u_{\beta_j}\in E^{1-\tau_{ij}}$.  Since $u_3\neq b_{ij}\in V(H)\setminus U_2$, this implies $u_{\beta_i}\nsim_{H_2}u_{\beta_j}$, a contradiction. 
\end{proof}

We now prove that if a $3$-graph $H$ has more than $C$-many $\sim_H$-classes, then we can find an induced sub-$3$-graph of bounded size witnessing this.  

\begin{lemma}\label{lem:irrgraphs2hg}
Let $K$ be the $6$-color Ramsey number $R(3,3,3,3,3,3)$,  let $L=K+2$, and let $N=3L+3$.  Suppose $C\geq 1$ is an integer and $H=(V,E)$ is a $3$-graph with more than $C$-many $\sim_H$-classes.  Then there exists an almost prime induced sub-$3$-graph $H'$ of $H$ with more than $C$-many $\sim$-classes and at most $CN$ vertices.     
\end{lemma}
\begin{proof}
Let $\ell>C$ be an integer, and let $H=(V,E)$ be a $3$-graph with $\ell$-many $\sim_H$-classes. We first construct an integer $t$ and a sequence of $3$-graphs $H_0,\ldots, H_t$ and integers $\ell_0,\ldots, \ell_t$ inductively as follows.

Step 0: Set $H_0=H$ and let $\ell_0=\ell$.  If $\ell_0\leq CL+1$,  set $t=0$ and end the construction. Otherwise, we have $CL+1<\ell_0 \leq \ell-0$.  Go to the next step.

Step $m+1$: Suppose $m\geq 0$ is an integer, and assume by induction we have constructed an induced sub-$3$-graph $H_m$ of $H$ with $\ell_m$-many $\sim_{H_m}$-classes, for some integer $\ell_m$ satisfying $CL+1<\ell_m\leq \ell-m$.  By Lemma \ref{lem:deletesimhg}, there exists a $\sim_{H_m}$-class $U$ such that, if we set $H_{m+1}=H_m[V(H_m)\setminus U]$, and let $\ell_{m+1}$ be the number of $\sim_{H_{m+1}}$-classes in $H_{m+1}$, then $\ell_{m+1}\geq (\ell_m-1)/L$.  By definition of $H_{m+1}$ and Lemma \ref{lem:unionclasses}, $\ell_{m+1}\leq \ell_m-1\leq \ell-(m+1)$, where the last inequality is by our induction hypothesis $\ell_m\leq \ell-m$. On the other hand, by our choice of $H_{m+1}$, $\ell_{m+1}\geq (\ell_m-1)/L>C$, where the last inequality is by our induction hypothesis $CL+1<\ell_m$.  If $\ell_{m+1}\leq CL+1$, set $t=m+1$ and end the construction. Otherwise, $CL+1<\ell_{m+1}\leq \ell_m-1\leq \ell-(m+1)$.  Go to the next step.  

Clearly this construction will halt after at most $\ell-(CL+1)$ steps.  At the end, we will have found $H_t$, an induced sub-$3$-graph of $H$, with $\ell_t$-many $\sim_{H_t}$-classes for some $\ell_t$ satisfying $C<\ell_t\leq CL+1$.  By Lemma \ref{lem:irr3graphs2repeat}, $H_t$ contains an induced, almost prime sub-$3$-graph $H'$ with $\ell_t$-many $\sim_{H'}$-classes. By Observation \ref{ob:irr3graphs}, $|V(H')|\leq 3\ell_t\leq 3CL+3\leq CN$ vertices.  This $H'$ satisfies the desired conclusions. 
\end{proof}

 We now prove  Theorem \ref{thm:irrgraphs2hg} (repeated below as Theorem \ref{thm:irrgraphs2hgrepeat}), which shows the class of finite $3$-graphs with at most $C$-many $\sim$-classes is characterized by omitting $\calF_{C+1}$ from Definition \ref{def:fc}.  

\begin{theorem}\label{thm:irrgraphs2hgrepeat}
Let $C\geq 1$ be an integer. Recalling that $\calG^{(3)}$ denotes the class of all finite $3$-graphs, define 
$$
 \calH_C=\{H\in \calG^{(3)}: \text{$H$ has at most $C$-many $\sim_H$-classes}\}.
 $$
 Then $\calH_C=\Forb(\calF_{C+1})$, where $\calF_{C+1}$ is from Definition \ref{def:fc}.   
\end{theorem}
\begin{proof}
 Suppose $H$ is a finite $3$-graph and $H\notin \Forb(\calF_{C+1})$. Then $H$ contains an element $F$ of $\calF_{C+1}$ as an induced sub-$3$-graph.  By definition of $\calF_{C+1}$, $F$ has more than $C$-many $\sim_F$-classes.  Lemma \ref{lem:unionclasses} implies the number of $\sim_H$-classes in $H$ is at least the number of $\sim_F$-classes in $F$, so  $H\notin \calH_C$.

Suppose conversely $H$ is a finite $3$-graph and $H\notin \calH_C$. Then $H$ has more than $C$-many $\sim_H$-classes. By Lemma \ref{lem:irrgraphs2hg}, $H$ contains an almost prime induced sub-$3$-graph $H'$ such that $H'$ has more than $C$-many $\sim_{H'}$-classes and such that $|V(H')|\leq CN$, where $N$ is as in Lemma \ref{lem:irrgraphs2hg}.  Since $H'$ has at least $C+1$-many $\sim_{H'}$-classes and $|V(H')|\leq CN$, we have $H'\in \calF_{C+1}$ by definition of $\calF_{C+1}$. This shows $H\notin \Forb(\calF_{C+1})$.
\end{proof}

\section{Proofs of Theorems \ref{thm:prime} and \ref{thm:prime3}}
 
This section contains the proofs of a Ramsey-theoretic result for almost prime graphs (Theorem \ref{thm:prime}), along with an analogous result for almost prime $3$-graphs (Theorem \ref{thm:prime3}).  Both proofs will heavily rely on Notation \ref{not:01nbrs}. 

\subsection{Proof of Theorem \ref{thm:prime}}\label{ss:ramsey1}  We begin by stating a lemma which, given a sequence of pairs of vertices, extracts a subsequence where certain edge relationships behave uniformly.  This lemma is not difficult to prove directly, and is also a special case of a standard result in model theory (see e.g. Theorem 2.4(2) in \cite{Shelah.1990o5n}).  For these reasons we omit the proof.

\begin{lemma}\label{lem:prime}
For all integers $k\geq 1$ there exists an integer $N\geq 1$ such that the following holds.  Suppose $G=(V,E)$ is a graph and $(x_1,y_1),\ldots, (x_N,y_N)$ are pairwise distinct elements of $V^2$ such that for each $1\leq i\leq N$, $|\{x_i,y_i\}|=2$, and such that for all $1\leq i\neq j\leq N$, $x_i\neq x_j$ and $y_i\neq y_j$.   Then there exists $\alpha\in \{0,1\}$ and a subsequence $1\leq \beta_1<\ldots<\beta_k\leq N$ such that $|\{x_{\beta_1},y_{\beta_1},\ldots,x_{\beta_k},y_{\beta_k}\}|=2k$, and such that for all $1\leq v<u\leq k$, $x_{\beta_u}y_{\beta_v}\in E^{\alpha}$.
\end{lemma}

We now prove Theorem \ref{thm:prime}.  In it, we will use the notation $x\gg y$ to mean $x$ is sufficiently large compared to $y$.

\vspace{2mm}
  
\noindent{\bf Proof of Theorem \ref{thm:prime}.}
Fix $k\geq 1$ and let $n_1\gg n_2\gg n_3\gg k$.   Assume $G=(V,E)$ is an almost prime graph with $|V|\geq n_1$.  We will use throughout the proof that the following holds because $G$ is almost prime.
\begin{align}\label{al:irredg}
\text{for any $V'\subseteq V$ satisfying $|V'|>2$, there exist $v,v'\in V'$ such that $v\nsim_Gv'$.}
\end{align}

\underline{Step $1$:} Since $|V|\geq n_1$, (\ref{al:irredg}) implies there exist $v_0,v_1\in V$ with $v_0\nsim_Gv_1$.  By definition of $\sim_G$, this means there exists $x_1\in V\setminus\{v_0,v_1\}$  such that $x_1v_1\in E$ and $x_1v_0\notin E$, or vice versa.  After relabeling, we may assume $x_1v_1\in E$ and $x_1v_0\notin E$.  Let $\alpha(1)\in \{0,1\}$ be such that 
$$
|N_{G^{\alpha(1)}}(x_1 )\cap (V\setminus \{v_0,v_1,x_1 \})|\geq \frac{|V\setminus \{v_0,v_1,x_1 \}|}{2},
$$
 and let $\beta(1)$ be such that $\{0,1\}=\{\alpha(1),\beta(1)\}$.  Set $z_1=v_{\beta(1)}$, and  set
$$
Z_1= N_{G^{\alpha(1)}}(x_1 )\cap (V\setminus \{x_1,z_1 \}).
$$
Note $|Z_1|\geq \frac{|V\setminus\{v_0,v_1,x_1\}|}{2} \geq \frac{|V|}{4}$, where the last inequality is because $|V|$ is sufficiently large.  Note that by construction $x_1\neq z_1$ and $Z_1\cap \{x_1, z_1\}=\emptyset$. 

\underline{Step $\ell+1$:} Suppose $1\leq \ell<n_2$ and assume by induction we have constructed the following:
\begin{itemize}
\item tuples $(x_1,\ldots,x_\ell), (z_1,\ldots, z_\ell) \in V^{\ell}$, 
\item pairs $(\alpha(1),\beta(1))\ldots, (\alpha(\ell),\beta(\ell))\in \{0,1\}^{2}$, and
\item a set  $Z_\ell\subseteq V$, 
 \end{itemize} 
 such that the following hold.
\begin{enumerate}[(a)]
\item For all $1\leq i\leq \ell$, $\{\alpha(i),\beta(i)\}=\{0,1\}$,
\item  $|\{x_1, \ldots, x_{\ell}\}|=|\{z_1,\ldots, z_{\ell}\}|=\ell$, 
\item For all $1\leq i\leq \ell$, $x_i z_i\in E^{\beta(i)}$,
\item For all $1\leq i<j\leq \ell$,  $x_iz_j\in E^{\alpha(i)}$,
\item $Z_\ell\subseteq \Big(\bigcap_{i=1}^\ell N_{G^{\alpha(i)}}(x_i ) \Big)\setminus \{x_1,\ldots, x_\ell, z_1,\ldots, z_\ell\}$ and $|Z_\ell|\geq |V|/4^\ell$.
\end{enumerate}
Since $|V|\geq n_1\gg n_2>\ell$, $|Z_\ell|\geq |V|/4^\ell$ implies $|Z_\ell|>2$. Consequently, by (\ref{al:irredg}), there exist $w_0,w_1\in Z_\ell$ such that $w_0\nsim_G w_1$. This means there exists $x_{\ell+1} \in V\setminus \{w_1,w_0\}$ so that $x_{\ell+1} w_1\in E$ and $x_{\ell+1} w_0\notin E$, or vice versa. After relabeling, we may assume $x_{\ell+1} w_1\in E$ and $x_{\ell+1} w_0\notin E$.  Using induction hypothesis (e), and the fact $w_0,w_1\in Z_{\ell}$, we see that  it must be the case that $x_{\ell+1}\notin \{x_1,\ldots, x_{\ell}\}$.  Combining with induction hypothesis (b), we have $|\{x_1,\ldots, x_{\ell+1}\}|=\ell+1$.  Let $\alpha(\ell+1)\in \{0,1\}$ be such that 
\begin{align}\label{al:zl}
|N_{G^{\alpha(\ell+1)}}(x_{\ell+1} )\cap (Z_\ell\setminus \{x_{\ell+1}, w_{0},w_1\})|\geq \frac{|Z_\ell\setminus \{x_{\ell+1} ,w_{0},w_1\}|}{2}.
\end{align}
Let $\beta(\ell+1)$ be such that $\{\alpha(\ell+1),\beta(\ell+1)\}=\{0,1\}$. Set $z_{\ell+1}=w_{\beta(\ell+1)}$, and let 
$$
Z_{\ell+1}= N_{G^{\alpha(\ell+1)}}(x_{\ell+1} )\cap (Z_{\ell}\setminus \{x_{\ell+1}, z_{\ell+1}\}).
$$
Since $z_{\ell+1}\in Z_{\ell}$, induction hypotheses (b) and (e) imply $|\{ z_1,\ldots,  z_{\ell+1}\}|= \ell+1 $. By  construction, $x_{\ell+1}z_{\ell+1} \in E^{\beta(\ell+1)}$, so combining with induction hypothesis (c), we have that for all $1\leq i\leq \ell+1$, $x_iz_i\in E^{\beta(i)}$. Since $z_{\ell+1}\in Z_{\ell}$, induction hypotheses (d) and (e) imply that for all $1\leq i<j\leq \ell+1$, $x_iz_j\in E^{\alpha(i)}$.  By definition, and induction hypothesis (e), $Z_{\ell+1}\subseteq \Big(\bigcap_{i=1}^{\ell+1}N_{G^{\alpha(i)}}(x_i ) \Big)\setminus \{x_1,\ldots, x_{\ell+1}, z_1,\ldots, z_{\ell+1}\}$.  Further, 
$$
|Z_{\ell+1}| \geq \frac{|Z_\ell\setminus\{x_{\ell+1}, w_0,w_1\}|}{2}\geq \frac{1}{2}\left(\frac{|V|}{4^{\ell}}-3\right) \geq \frac{|V|}{4^{\ell+1}},
$$
where the first inequality is by (\ref{al:zl}), the second uses that $|Z_\ell|\geq |V|/4^\ell$, and the last uses that $|V|\geq n_1\gg n_2>\ell$.  This completes step $\ell+1$ of the construction.
  
Since $n_1\gg n_2$, we can perform $n_2$-many steps of this construction. After $n_2$ steps, we will have constructed
\begin{itemize}
\item sequences $(x_1,\ldots, x_{n_2})$ and $(z_1,\ldots, z_{n_2})\in V^{n_2}$, and 
 \item pairs  $(\alpha(1),\beta(1)),\ldots, (\alpha(n_2),\beta(n_2))\in \{0,1\}^2$
\end{itemize}
such that
\begin{itemize}
\item For all $1\leq i\leq n_2$, $\{\alpha(i),\beta(i)\}=\{0,1\}$,
\item  $|\{x_1, \ldots, x_{n_2}\}|=|\{z_1,\ldots, z_{n_2}\}|=n_2$ and for each $1\leq i\leq n_2$, $x_i\neq z_i$, and
\item For all $1\leq i<j\leq n_2$,  $x_i z_i\in E^{\beta(i)}$, $x_j z_j\in E^{\beta(j)}$, and $x_i z_j\in E^{\alpha(i)}$.
\end{itemize}
By the Pigeonhole Principle, there is a choice of $\alpha\neq \beta\in \{0,1\}$ such that 
$$
|\{i\in [n_2]: (\alpha(i),\beta(i))=(\alpha,\beta)\}|\geq \frac{n_2}{4}.
$$
Since $n_2\gg n_3$, we can choose an increasing sequence $1\leq i_1<\cdots<i_{n_3}\leq n_2$  such that for each $u\in [n_3]$, we have $(\alpha(i_u),\beta(i_u))=(\alpha,\beta)$.  This tells us that for all $1\leq u\leq n_3$, $x_{i_u}z_{i_u}\in E^{\beta}$ and for all $1\leq u<v\leq n_3$, $x_{i_u}z_{i_v}\in E^{\alpha}$.

Now applying Lemma \ref{lem:prime} to the pairs $(x_{i_1}, z_{i_1}),\ldots, (x_{i_{n_3}},z_{i_{n_3}})$ yields some $\tau\in \{0,1\}$ and a sequence $ j_1<\ldots<j_{k+1}$ with $\{j_1,\ldots, j_{k+1}\}\subseteq \{i_1,\ldots, i_{n_3}\}$ such that 
$$
|\{x_{j_1},\ldots, x_{j_{k+1}},z_{j_1},\ldots, z_{j_{k+1}}\}|=2(k+1),
$$
 and such that for all $1\leq v<u\leq k+1$, $x_{j_u} z_{j_v}\in E^{\tau}$. 

Suppose first $\alpha=1$ and $\beta=0$. If $\tau=1$, then for each $1\leq i\leq k$, let $a_i=x_{j_i}$ and $b_i=z_{j_i}$.  We then have $G[\{a_1,\ldots, a_k,b_1,\ldots, b_k\}]\in \Irr(k)$, since $a_ub_v\in E(G)$ if and only if $u\neq v$.  On the other hand, if $\tau=0$, then for each $1\leq i\leq k$, let $a_i=x_{j_i}$, and let $b_i=z_{j_{i+1}}$.  We then have $G[\{a_1,\ldots, a_k,b_1,\ldots, b_k\}]\in \Irr(k)$, since $a_ub_v\in E(G)$ if and only if $u\leq v$.

Suppose now we are in the case where $\alpha=0$ and $\beta=1$. If we also have $\tau=0$, then for each $1\leq i\leq k$, let $a_i=x_{j_i}$ and $b_i=z_{j_i}$.  In this case, we have $G[\{a_1,\ldots, a_k,b_1,\ldots, b_k\}]\in \Irr(k)$, since $a_ub_v\in E(G)$ if and only if $u=v$.  On the other hand, if $\tau=1$, then for each $1\leq i\leq k$, let $a_i=z_{j_i}$, and let $b_i=x_{j_i}$.  We then have $G[\{a_1,\ldots, a_k,b_1,\ldots, b_k\}]\in \Irr(k)$, since $a_ub_v\in E(G)$ if and only if $u\leq v$. This completes the proof.
\qed

\subsection{Proof of Theorem \ref{thm:prime3}}\label{ss:ramsey2}
We begin by stating an analogue of Lemma \ref{lem:prime}.  We will not prove this lemma, as it is again a special case of a standard model theoretic tool  (see  Theorem 2.4(2) in \cite{Shelah.1990o5n}).

\begin{lemma}\label{lem:prime3}
For all integers $k\geq 1$ there exists an integer $N=N(k)\geq 1$ such that the following holds.  Suppose $H=(V,E)$ is a $3$-graph and $(x_1,y_1,z_1),\ldots, (x_N,y_N,z_N)$ are pairwise distinct elements of $V^3$ such that $1\leq i\neq j\leq N$ implies $z_i\neq z_j$ and $(x_i,y_i)\neq (x_j,y_j)$, and such that for all $1\leq i\leq N$, $|\{x_i,y_i,z_i\}|=3$.  Then there exist $\alpha\in \{0,1\}$ and $1\leq i_1<\ldots<i_k\leq N$ such that for all $1\leq v<u\leq k$, $x_{i_u}y_{i_u}z_{i_v}\in E^{\alpha}$ and such that $\{x_{i_1},\ldots, x_{i_k},y_{i_1},\ldots, y_{i_k}\}\cap \{z_{i_1},\ldots, z_{i_k}\}=\emptyset$. 
\end{lemma}

We now prove Theorem \ref{thm:prime3}.
\vspace{2mm}
  
\noindent{\bf Proof of Theorem \ref{thm:prime3}.}
Fix $k\geq 1$ and let $n_1\gg n_2\gg n_3\gg k$.   Assume $H=(V,E)$ is an almost prime $3$-graph with $|V|\geq n_1$.  We will use throughout the proof that the following holds because $H$ is almost prime.
\begin{align}\label{al:irred}
\text{for any $V'\subseteq V$ satisfying $|V'|>3$, there exist $v,v'\in V'$ such that $v\nsim_Hv'$.}
\end{align}

\underline{Step $1$:} Since $|V|\geq n_1$, (\ref{al:irred}) implies there exist $v_0,v_1\in V$ with $v_0\nsim_Hv_1$.  By definition of $\sim_H$, this means there exist $x_1\neq y_1\in V\setminus\{v_0,v_1\}$  such that $x_1y_1v_1\in E$ and $x_1y_1v_0\notin E$, or vice versa.  After relabeling, we may assume $x_1y_1v_1\in E$ and $x_1y_1v_0\notin E$.  Let $\alpha(1)\in \{0,1\}$ be such that 
$$
|N_{H^{\alpha(1)}}(x_1y_1)\cap (V\setminus \{v_0,v_1,x_1,y_1\})|\geq \frac{|V\setminus \{v_0,v_1,x_1,y_1\}|}{2},
$$
 and let $\beta(1)$ be such that $\{0,1\}=\{\alpha(1),\beta(1)\}$.  Set $z_1=v_{\beta(1)}$ and set 
$$
Z_1=N_{H^{\alpha(1)}}(x_1y_1)\cap (V\setminus \{x_1,y_1,z_1\}).
$$
Note $|Z_1|>\frac{|V\setminus\{v_0,v_1,x_1,y_1\}|}{2}\geq\frac{|V|}{2}-2\geq |V|/4$, where the last inequality is because $|V|$ is sufficiently large.  Note that by construction $Z_1\cap \{x_1,y_1,z_1\}=\emptyset$. 

\underline{Step $\ell+1$:} Suppose $1\leq \ell<n_2$ and assume by induction we have constructed the following:
\begin{itemize}
\item tuples $(x_1,\ldots,x_\ell), (y_1,\ldots, y_\ell),(z_1,\ldots, z_\ell)\in V^{\ell}$, 
\item pairs $(\alpha(1),\beta(1))\ldots, (\alpha(\ell),\beta(\ell))\in \{0,1\}^{2}$, 
\item a set  $Z_\ell\subseteq V$, 
 \end{itemize} 
 such that the following hold.
\begin{enumerate}[(a)]
\item For all $1\leq i\leq \ell$, $\{\alpha(i),\beta(i)\}=\{0,1\}$,
\item For all $1\leq i\leq \ell$, $|\{x_i,y_i,z_i\}|=3$,
\item $|\{z_1,\ldots, z_{\ell}\}|=|\{(x_1,y_1),\ldots, (x_{\ell},y_{\ell})\}|=\ell$,
\item For all $1\leq i\leq \ell$,  $x_iy_iz_i\in E^{\beta(i)}$,
\item $Z_\ell\subseteq \Big(\bigcap_{i=1}^\ell N_{H^{\alpha(i)}}(x_iy_i)\Big)\setminus \{x_1,\ldots, x_\ell,y_1,\ldots, y_\ell,z_1,\ldots, z_\ell\}$ and $|Z_\ell|\geq |V|/4^\ell$.
\end{enumerate}
Since $|V|\geq n_1\gg \ell$, $|Z_\ell|\geq |V|/4^\ell$ implies $|Z_\ell|>3$. Consequently, by (\ref{al:irred}), there exist $w_0,w_1\in Z_\ell$ such that $w_0\nsim_H w_1$. This means there exist $x_{\ell+1}\neq y_{\ell+1}\in V\setminus \{w_1,w_0\}$ so that $x_{\ell+1}y_{\ell+1}w_1\in E$ and $x_{\ell+1}y_{\ell+1}w_0\notin E$, or vice versa. After relabeling, we may assume $x_{\ell+1}y_{\ell+1}w_1\in E$ and $x_{\ell+1}y_{\ell+1}w_0\notin E$.  This along with the inductive hypothesis (e) imply $(x_{\ell+1},y_{\ell+1})\notin \{(x_1,y_1),\ldots, (x_{\ell},y_{\ell})\}$. Combining with induction hypothesis (c), we have $|\{(x_1,y_1),\ldots, (x_{\ell+1},y_{\ell+1})\}|=\ell+1$. Let $\alpha(\ell+1)\in \{0,1\}$ be such that 
$$
|N_{H^{\alpha(\ell+1)}}(x_{\ell+1}y_{\ell+1})\cap (Z_\ell\setminus \{x_{\ell+1},y_{\ell+1},w_{0},w_1\})|\geq \frac{|Z_\ell\setminus \{x_{\ell+1},y_{\ell+1},w_{0},w_1\}|}{2}.
$$
Let $\beta(\ell+1)$ be such that $\{\alpha(\ell+1),\beta(\ell+1)\}=\{0,1\}$, set $z_{\ell+1}=w_{\beta(\ell+1)}$, and set 
$$
Z_{\ell+1}=N_{H^{\alpha(\ell+1)}}(x_{\ell+1}y_{\ell+1})\cap (Z_{\ell}\setminus \{x_{\ell+1},y_{\ell+1},z_{\ell+1}\}).
$$
Since $z_{\ell+1}\in Z_{\ell}$, induction hypotheses (c) and (e) imply $|\{z_1,\ldots, z_{\ell+1}\}|=\ell+1$. By construction, we  have $\{\alpha(\ell+1),\beta(\ell+1)\}=\{0,1\}$, $|\{x_{\ell+1},y_{\ell+1},z_{\ell+1}\}|=3$, $x_{\ell+1}y_{\ell+1}z_{\ell+1}\in E^{\beta(\ell+1)}$, and $Z_{\ell+1}\subseteq \Big(\bigcap_{i=1}^{\ell+1}N_{E^{\alpha(i)}}(x_iy_i)\Big)\setminus \{x_1,\ldots, x_{\ell+1},y_1,\ldots, y_{\ell+1},z_1,\ldots, z_{\ell+1}\}$.  Further, 
$$
|Z_{\ell+1}|\geq \frac{|Z_\ell\setminus\{x_{\ell+1},y_{\ell+1},w_0,w_1\}|}{2}\geq \frac{|Z_\ell|}{2}-2\geq \frac{|V|}{2\cdot 4^\ell} -2\geq \frac{|V|}{4^{\ell+1}},
$$
where the second inequality uses that $|Z_\ell|\geq |V|/4^\ell$ and the last inequality uses that $|V|\geq n_1\gg n_2>\ell$.  This completes step $\ell+1$ of the construction.
  
Since $n_1\gg n_2$, we can perform at least $n_2$ many such steps.  After $n_2$ steps we will have constructed
\begin{itemize}
\item sequences $(x_1,\ldots, x_{n_2})$, $(y_1,\ldots, y_{n_2})$, and $(z_1,\ldots, z_{n_2})\in V^{n_2}$, and 
 \item pairs  $(\alpha(1),\beta(1)),\ldots, (\alpha(n_2),\beta(n_2))\in \{0,1\}^2$
\end{itemize}
such that
\begin{itemize}
\item For all $1\leq i\leq n_2$, $\{\alpha(i),\beta(i)\}=\{0,1\}$,
\item For all $1\leq i<j\leq n_2$, $|\{x_i,y_i,z_i\}|=|\{x_j,y_j,z_j\}|=3$,
\item $|\{z_1,\ldots, z_{n_2}\}|=|\{(x_1,y_1),\ldots, (x_{n_2},y_{n_2})\}|=n_2$,  
\item For all $1\leq i<j\leq n_2$,  $x_iy_iz_i\in E^{\beta(i)}$, $x_jy_jz_j\in E^{\beta(j)}$, and $x_iy_iz_j\in E^{\alpha(i)}$.
\end{itemize}
By the Pigeonhole Principle, there is a choice of $\alpha\neq \beta\in \{0,1\}$ such that 
$$
|\{i\in [n_2]: (\alpha(i),\beta(i))=(\alpha,\beta)\}|\geq \frac{n_2}{4}.
$$
Since $\frac{n_2}{4}\geq n_3$, we can choose an increasing sequence $1\leq i_1<\cdots<i_{n_3}\leq n_2$  such that for each $u\in [n_3]$, we have $(\alpha(i_u),\beta(i_u))=(\alpha,\beta)$.

Now applying Lemma \ref{lem:prime3} to the tuples $(x_{i_1},y_{i_1},z_{i_1}),\ldots, (x_{i_{n_3}},y_{i_{n_3}},z_{i_{n_3}})$ yields the existence of some $\tau\in \{0,1\}$ and a sequence $ j_1<\ldots<j_{k+1}$ with $\{j_1,\ldots, j_{k+1}\}\subseteq \{i_1,\ldots, i_{n_3}\}$ such that for all $1\leq v<u\leq k+1$, $x_{j_u}y_{j_u}z_{j_v}\in E^{\tau}$ and such that the sets $\{z_{j_1},\ldots, z_{j_{k+1}}\}$ and $\{x_{j_1},\ldots, x_{j_{k+1}},y_{j_1},\ldots, y_{j_{k+1}}\}$ are disjoint.  

Suppose first $\alpha=1$ and $\beta=0$. If we also have $\tau=1$, then for each $1\leq i\leq k$, let $a_i=z_{j_i}$ and $b_i=x_{j_i}$, and $c_i=y_{j_i}$.  We then have $H[\{a_1,\ldots, a_k,b_1,\ldots, b_k,c_1,\ldots, c_k\}]$ is in $\widehat{\Irr(k)}$, since $a_ub_vc_v\in E(H)$ if and only if $u\neq v$.  On the other hand, if $\tau=0$, then for each $1\leq i\leq k$, let $a_i=z_{j_{k-i+2}}$, and let $b_i=x_{j_{k-i+1}}$ and $c_i=y_{j_{k-i+1}}$.  We then have $H[\{a_1,\ldots, a_k,b_1,\ldots, b_k,c_1,\ldots, c_k\}]\in \widehat{\Irr(k)}$, since $a_ub_vc_v\in E(H)$ if and only if $u\leq v$.

Suppose now we are in the case where $\alpha=0$ and $\beta=1$. If we also have $\tau=0$, then for each $1\leq i\leq k$, let $a_i=x_{j_i}$ and $b_i=y_{j_i}$, and $c_i=z_{j_i}$.  We then have that $H[\{a_1,\ldots, a_k,b_1,\ldots, b_k,c_1,\ldots, c_k\}]\in \widehat{\Irr(k)}$, since  $a_ub_vc_v\in E(H)$ if and only if $u=v$.  On the other hand, if $\tau=1$, then for each $1\leq i\leq k$, let $a_i=z_{j_{i}}$, $b_i=x_{j_{i}}$, and $c_i=y_{j_{i}}$. We then have $H[\{a_1,\ldots, a_k,b_1,\ldots, b_k,c_1,\ldots, c_k\}]\in \widehat{\Irr(k)}$, since $a_ub_vc_v\in E(H)$ if and only if $u\leq v$. This completes the proof.
 
\qed

\section{Lemmas related to VC-dimension and SVC-dimension}\label{ss:appclose}

In this section we prove Proposition  \ref{prop:equiv22}, Fact \ref{fact:svcbip2}, and Theorem \ref{thm:vdischom}.  We begin by covering some basic facts about VC-dimension. We will use the well-known Vapnik-Chervonenkis-Perles-Sauer-Shelah lemma for set systems of bounded VC-dimension (see e.g. Theorem 10.2.5 in \cite{Matousek}).   

\begin{theorem}\label{thm:ss}
Suppose $k\geq 0$ is an integer and $(X,\calF)$ is a set system with $\VC$-dimension at most $k$.  Then for all nonempty $Y\subseteq X$, 
$$
|\{F\cap Y: F\in \calF\}|\leq \sum_{i=0}^k{|Y|\choose i}\leq (k+1) |Y|^k.
$$
\end{theorem}  

We will also need to work with the dual of a set system, defined as follows. Given a set system $\calS=(X,\calF)$, its \emph{dual set system} is defined to be the set system $\calS^{\text{dual}}=(\calF, \calX)$, where $\calX=\{S_x: x\in X\}$ and  for each $x\in X$, $S_x=\{ S\in \calF: x\in S\}$.  It is easy to see $\VC((\calS^{\text{dual}})^{\text{dual}})=\VC(\calS)$.  We will use the well-known fact that if $\VC(\calS)=k$, then $\VC(\calS^{\text{dual}})<2^{k+1}$ (see for example Lemma 10.3.4 in \cite{Matousek}).  Our proof of Proposition \ref{prop:equiv22} will use the following lemma.

\begin{lemma}\label{lem:equivhgvcrestate}
For all integers $k\geq 1$, there exists an integer $N\geq 1$ so that if $H$ is a $3$-graph with $\VC(H)\geq N$, then $H$ contains an element of $\widehat{\PS(k)}$ as an induced sub-$3$-graph. 
\end{lemma}
\begin{proof}
Fix $k\geq 1$ and let $K$ be sufficiently large compared to $k$, and set $N=2^{K+1}$.  Assume $H$ is a $3$-graph with $\VC(H)\geq N=2^{K+1}$. By definition, this means the set system $\calS=(V(H), \{N_H(uv): uv\in {V(H)\choose 2}\})$ has VC-dimension at least $2^{K+1}$.  By the remarks preceding the statement of Lemma \ref{lem:equivhgvcrestate}, $\VC(\calS)\leq 2^{\VC(\calS^{\text{dual}})+1}$, and consequently, we must have $\VC(\calS^{\text{dual}})\geq K$. 

By definition of $\calS$ and the dual operation, this implies there exists a set of vertices $\{z_S: S\subseteq [K]\}\subseteq V(H)$ and a set of pairs $\{x_1y_1,\ldots, x_Ky_K\}\subseteq  {V(H)\choose 2}$ such that 
\begin{align}\label{al:equivhgvc}
\text{ $x_iy_iz_S\in E(H)$ when $i\in S$ and $x_iy_iz_S\notin E(H)$ when $i\notin S$.}
\end{align}
 We now perform a  cleaning procedure  because we do not know the sets $\{x_i,y_i: i\in [K]\}$ and $\{z_S: S\subseteq [K]\}$ are disjoint.  It is easy to see (\ref{al:equivhgvc}) implies that the pairs $x_1y_1,\ldots, x_Ky_K$ are pairwise distinct as elements of ${V(H)\choose 2}$. Indeed, if $1\leq i\neq j\leq K$, then $x_iy_iz_{\{i\}}\in E(H)$ and $x_jy_jz_{\{i\}}\notin E(H)$ imply $x_iy_i\neq x_jy_j$.  It is also easy to see the elements in the set $\{z_S: S\subseteq [K]\}$ must be pairwise distinct via a similar argument. 
 
 We next define a set system $(X,\calF)$. First,  the ground set $X$ is defined to be
$$
X=\{x_{1}y_{1},\ldots, x_{K}y_{K}\}\subseteq {V(H)\choose 2}.
$$
By the remarks above, $|X|=K$. Given a set $S\subseteq \{1,\ldots,K\}$, define 
$$
F_S=\{x_{i}y_{i}\in X: i\in S\}\subseteq X.
$$
Clearly $S\neq S'$ implies $F_S\neq F_{S'}$. Now define
$$
\calF=\{F_S: S\subseteq [K]\text{ and }z_S\notin  \{x_{1},\ldots, x_{K},y_{1},\ldots, y_{K}\}\}.
$$
Using that $|\{x_{1},\ldots, x_{K},y_{1},\ldots, y_{K}\}|\leq 2K$ and the fact the elements in $\{z_S:S\subseteq[K]\}$ are pairwise distinct,  we have 
$$
|\calF|\geq 2^K-| \{x_{1},\ldots, x_{K},y_{1},\ldots, y_{K}\}|\geq 2^K-2K >(k+1)K^k,
$$
 where the last inequality uses that $K$ is sufficiently large compared to $k$.   By  Theorem \ref{thm:ss}, there exists a set $J\subseteq  X$ of size $k$ which is shattered by $\calF$.

Let  $\{j_1,\ldots, j_k\}\subseteq \{1,\ldots, K\}$ be such that $  j_1<\ldots<j_k $ and $J=\{x_{j_1}y_{j_1},\ldots, x_{j_k}y_{j_k}\}$. Since $J$ is shattered by $\calF$, we know that for all $T\subseteq \{j_1,\ldots, j_k\}$, there is $F\in \calF$ such that $F\cap J=\{x_{j_u}y_{j_u}: j_u\in T\}$.  By definition of $\calF$, this means there is a vertex $d_T\in V(H)$ satisfying $d_T\notin \{x_{j_1},\ldots, x_{j_k},y_{j_1},\ldots, y_{j_k}\}$, such that for all $u\in [k]$, $x_{j_u}y_{j_u}d_T\in E(H)$ if and only if $j_u\in T$. Define 
$$
H'=H[\{x_{j_1},\ldots, x_{j_k},y_{j_1},\ldots, y_{j_k}\}\cup \{d_T: T\subseteq \{j_1,\ldots, j_k\}\}].
$$
Then  $H'\in \widehat{\PS(k)}$ by construction, so we are done. \end{proof}

We now prove Proposition \ref{prop:equiv22}, restated below for the convenience of the reader.

\begin{proposition}\label{prop:equiv22repeat}
For any hereditary $3$-graph property $\calH$, the following are equivalent.
\begin{enumerate}
\item $\calH$ has infinite VC-dimension.
\item For all  integers $k\geq 1$, $ \calH  \cap \widehat{\PS(k)}\neq \emptyset$. 
\end{enumerate}
\end{proposition}
\begin{proof}
It follows immediately from Lemma \ref{lem:equivhgvcrestate} that (1) implies (2).  

Assume now (2) holds. Fix $k\geq 1$ and set $K=2^{k+1}$. By assumption, there exists some $H\in \widehat{\PS(K)}\cap \calH$.  Let $\calS$ be the neighborhood set system of $H$, i.e. 
$$
\calS=\left(V(H), \left\{N_H(uv): uv\in {V(H)\choose 2}\right\}\right).
$$
 By definition, the VC-dimension of the $3$-graph $H$ is $\VC(\calS)$.  Since $H\in \widehat{\PS(K)}$, there exists a set of vertices $\{z_S: S\subseteq [K]\}\subseteq V(H)$ and a set of pairs $\{x_1y_1,\ldots, x_Ky_K\}\subseteq  {V(H)\choose 2}$ such that  $x_iy_iz_S\in E(H)$ when $i\in S$ and $x_iy_iz_S\notin E(H)$ when $i\notin S$.  By definition, this tells us $\VC(\calS^{\text{dual}})\geq K$.  By the remarks preceding the statement of Lemma \ref{lem:equivhgvcrestate}, this implies $\VC(\calS)=\VC((\calS^{\text{dual}})^{\text{dual}})\geq k$, and thus $\VC(H)\geq k$. We have just shown that for all $k\geq 1$, there exists $H\in \calH$ satisfying $\VC(H)\geq k$, so (1) holds by definition.   
\end{proof}

\vspace{2mm}

We next turn to proving Fact \ref{fact:svcbip2} and Theorem \ref{thm:vdischom}.  These proofs will use the following definition.

\begin{definition}
Suppose $\calH$ is a hereditary $3$-graph property.  Define the \emph{link property corresponding to $\calH$} as follows (recalling $\calG^{(2)}$ denotes the class of all finite graphs).
$$
\calL_{\calH}=\{G\in \calG^{(2)}: G=H_x\text{ for some $H\in \calH$ and $x\in V(H)$}\}.
$$
\end{definition}

The following is an exercise we leave to the reader.

\begin{fact}\label{fact:linkprop}
Suppose $\calH$ is a hereditary $3$-graph property. Then the following hold.
\begin{enumerate}
\item   $\calL_{\calH}$ is a hereditary graph property.
\item $\SVC(\calH)=\VC(\calL_{\calH})$.
\item $\calL_{\calH}=\{G\in \calG^{(2)}: (1\otimes G)\cap \calH\neq \emptyset\}$.
\end{enumerate}
\end{fact}

We now prove Fact \ref{fact:svcbip2}, repeated below for the convenience of the reader.

\begin{fact}\label{fact:svcbip2repeat}
Suppose $\calH$ is a hereditary $3$-graph property. Then the following are equivalent.
\begin{enumerate}
\item $\SVC(\calH)=\infty$.
\item For all $k\geq 1$, $\calH\cap (1\otimes \PS(k))\neq \emptyset$.
\end{enumerate}
\end{fact}
\begin{proof}
By Fact \ref{fact:linkprop}, $\SVC(\calH)=\infty$ if and only if $\VC(\calL_{\calH})=\infty$. By Fact \ref{fact:bipvc}, $\VC(\calL_{\calH})=\infty$ if and only if for all $k\geq 1$, $\calL_{\calH}\cap \PS(k)\neq \emptyset$. By Fact \ref{fact:linkprop}(3), $\calL_{\calH}\cap \PS(k)\neq \emptyset$ for all $k\geq 1$ if and only if $\calH\cap (1\otimes \PS(k))\neq \emptyset$ for all $k\geq 1$.  Thus we have shown (1) if and only if (2).
\end{proof}

We next prove Theorem \ref{thm:vdischomrepeat} below, which provides equivalent characterizations for when a hereditary $3$-graph property is far from finite SVC-dimension.    Theorem \ref{thm:vdischom} then follows immediately.

\begin{theorem}\label{thm:vdischomrepeat}
Suppose $\calH$ is a hereditary $3$-graph property.  The following are equivalent.
\begin{enumerate}
\item For some $\e\in (0,1)$, $M_{\calH}^{\hom}(\e)=\infty$.\footnote{This property is called $\vdisc_3$-homogeneity in \cite{Terry.2021b}.}
\item $\calH$ is far from finite SVC-dimension. 
\item ${\bf B}_{\calH}$ has infinite SVC-dimension.
\item For all $k\geq 1$, ${\bf B}_{\calH}\cap (1\otimes \PS(k))\neq \emptyset$.
\item One of the following holds.
\begin{enumerate}
\item For every bipartite graph $G$, $(1\otimes G)\cap {\bf B}_{\calH}\neq \emptyset$. 
\item For every co-bipartite graph $G$, $(1\otimes G)\cap {\bf B}_{\calH}\neq \emptyset$.
\item For every split graph $G$, $(1\otimes G)\cap {\bf B}_{\calH}\neq \emptyset$.
\end{enumerate}
\item One of the following holds.
\begin{enumerate}
\item For all $n\geq 1$ and every bipartite graph $G$, $(n\otimes G)\cap {\bf B}_{\calH}\neq \emptyset$. 
\item For all $n\geq 1$ and every co-bipartite graph $G$, $(n\otimes G)\cap {\bf B}_{\calH}\neq \emptyset$.
\item For all $n\geq 1$ and every split graph $G$, $(n\otimes G)\cap {\bf B}_{\calH}\neq \emptyset$.
\end{enumerate}
\end{enumerate}
\end{theorem}
\begin{proof} The equivalence of (1) and (2) follows from Theorem 2.34 in \cite{Terry.2021b}.    The equivalence of (3) and (4) follows from Fact \ref{fact:svcbip2} (recalling that Fact \ref{fact:bhhgraphs} tells us ${\bf B}_{\calH}$ is a hereditary $3$-graph property).  

We now show (4) implies (2) via the contrapositive.  Suppose (2) fails. Then $\calH$ is close to some $\calH'$ satisfying $\SVC(\calH')<\infty$.  By Fact \ref{fact:svcbip2repeat}, there is some $K\geq 1$ such that $\calH'\cap (1\otimes \PS(K))=\emptyset$.  By Theorem \ref{thm:blowupthm}, ${\bf B}_{\calH}\cap (1\otimes \PS(K))=\emptyset$, so (4) fails, as desired. 

We now show (2) implies (4) via the contrapositive. Suppose (4) fails. Then there exists $K\geq 1$ such that ${\bf B}_{\calH}\cap (1\otimes \PS(K))= \emptyset$. By Theorem \ref{thm:blowupthm}, $\calH$ is close to $\Forb(1\otimes \PS(K))$.  By Fact \ref{fact:svcbip2repeat}, $\SVC(\Forb(1\otimes \PS(K)))<\infty$, so $\calH$ is close to finite SVC-dimension by definition, so (2) fails, as desired. 

We have now shown (1)-(4) are all equivalent. We next prove (3) implies (5). Suppose ${\bf B}_{\calH}$ has infinite SVC-dimension. By Fact \ref{fact:linkprop}, $\VC(\calL_{{\bf B}_{\calH}})=\infty$.   By Fact \ref{fact:bipvc}, $\calL_{{\bf B}_{\calH}}$ contains either every bipartite graph, every co-bipartite graph, or every split graph. By Fact \ref{fact:linkprop}(3), this tells us (5) holds. 

We now prove (5) implies (4) via the contrapositive.  Assume there exists $K\geq 1$ such that ${\bf B}_{\calH}\cap (1\otimes \PS(K))=\emptyset$.  Since $\PS(K)$ contains a bipartite graph, a co-bipartite graph, and a split graph, this immediately implies (5) is false.  We now have shown (1)-(5) are equivalent.

Finally, it is immediate from Fact \ref{fact:bhhgraphs} that (6) implies (5). On the other hand, that (5) implies (6) follows from the definition of ${\bf B}_{\calH}$.
 \end{proof}

\section{Characterizations of growth rates}\label{ss:charapp}

We include here the combinatorial characterizations of the properties in each growth class.  We intend this section to serve as a reference for the reader.   Beginning with the graphs case, we state and prove our characterizations beginning with the fastest growth rates, and ending with the slowest.  

\begin{theorem}[Tower]\label{thm:graphtower}
Suppose $\calH$ is a hereditary graph property.  Then the following are equivalent (see Definitions  \ref{def:bhgraphs}, \ref{def:vchp}, and \ref{def:pset}). 
\begin{enumerate}
\item $\Tw(\Omega(\e^{-2}))\leq M_{\calH}(\e)\leq \Tw(O(\e^{-4}))$.
\item $\VC(\calH)=\infty$.
\item For all $m\geq 1$, ${\bf B}_{\calH}\cap \PS(m)\neq \emptyset$.
\end{enumerate}
\end{theorem}
\begin{proof} The equivalence of (2) and (3) follows from Fact \ref{fact:bipvc}.  That (2) implies (1) follows from Corollary \ref{cor:fastgraphs}(1). By Corollary \ref{cor:fastgraphs}(2), if $\VC(\calH)<\infty$, then $ M_{\calH}(\e)$ is bounded above by a polynomial in $\e^{-1}$.  This shows (1) implies (2) via the contrapositive, completing the equivalence.\end{proof}

\begin{theorem}[Polynomial]\label{thm:graphpoly}
Suppose $\calH$ is a hereditary graph property.  Then the following are equivalent (see Definitions \ref{def:bhgraphs}, \ref{def:vchp},  \ref{def:pset}, and \ref{def:prime}). 
\begin{enumerate}
\item $\e^{-1+o(1)}\leq M_{\calH}(\e)\leq \e^{-C}$ for some constant $C>0$.
\item $\VC(\calH)<\infty$ and ${\bf B}_{\calH}$ contains infinitely many non-isomorphic almost prime graphs.
\item There exists $k\geq 1$ such that ${\bf B}_{\calH}\cap \PS(k)=\emptyset$, but ${\bf B}_{\calH}\cap \Irr(m)\neq \emptyset$ for all $m\geq 1$. 
\end{enumerate}
\end{theorem}
\begin{proof}
We first show (1) implies (2).  Suppose (1) holds.  Since $M_{\calH}(\e)$ is not bounded below by a tower function,  Theorem \ref{thm:graphtower} implies $\VC(\calH)<\infty$. Since $M_{\calH}(\e)$ is not a constant function, Lemma \ref{lem:finitegraph} implies ${\bf B}_{\calH}$ contains infinitely many non-isomorphic almost prime graphs. Thus (2) holds.

We next show (2) implies (1). Suppose  (2) holds. Since ${\bf B}_{\calH}$ contains infinitely many non-isomorphic almost prime graphs, Corollary \ref{cor:bhinf} implies $M_{\calH}(\e)\geq \e^{-1+o(1)}$. On the other hand, since $\VC(\calH)<\infty$, Corollary \ref{cor:fastgraphs}(2) implies $M_{\calH}(\e)\leq \e^{-C}$ for some $C>0$. Combining these upper and lower bounds yields (1). 

We now show (2) implies (3).  Suppose (2) holds. Since $\calH$ has finite VC-dimension, Fact \ref{fact:bipvc} implies there exists $k\geq 1$ such that ${\bf B}_{\calH}\cap \PS(k)=\emptyset$.   Since ${\bf B}_{\calH}$ contains arbitrarily large almost prime graphs, Corollary \ref{cor:bhchar} implies ${\bf B}_{\calH}\cap \Irr(m)\neq \emptyset$ for all $m\geq 1$. Thus (3) holds.

Finally, we show (3) implies (2).  Suppose (3) holds.  Since every element in $\Irr(m)$ is almost prime and has $2m$ vertices, ${\bf B}_{\calH}$ contains infinitely many non-isomorphic almost prime graphs.  Since ${\bf B}_{\calH}\cap \PS(k)=\emptyset$ for some $k\geq 1$, $\calH$ has finite VC-dimension by Fact \ref{fact:bipvc}. Thus we have shown (2) holds.\end{proof}

\begin{theorem}[Constant]\label{thm:graphconstant}
Suppose $\calH$ is a hereditary graph property.  Then the following are equivalent (see Definitions \ref{def:bhgraphs}, \ref{def:sim}, \ref{def:irr}, and \ref{def:prime}).
\begin{enumerate}
\item $M_{\calH}(\e)=C$ for some constant $C\in \mathbb{N}^{\geq 1}$.
\item ${\bf B}_{\calH}$ contains finitely many non-isomorphic almost prime graphs.
\item $M_{\calH}(\e)=C$ where $C\in \mathbb{N}^{\geq 1}$ satisfies
$$
C=\max\{\ell \in \mathbb{N}^{\geq 1}: \text{ there is $G\in {\bf B}_{\calH}$ with $\ell$-many $\sim_G$-classes}\}. 
$$
\item There exists $k\geq 1$ such that ${\bf B}_{\calH}\cap \Irr(k)=\emptyset$. 
\end{enumerate}
\end{theorem}
\begin{proof}
 That (2) and (4) are equivalent follows from Corollary \ref{cor:bhchar}, so we can ignore (4) for the rest of the proof.

That (2) implies (3) follows from Lemma \ref{lem:finitegraph}, and it is obvious (3) implies (1).  We have left to show (1) implies (2).  Assume (1).  Since $M_{\calH}$ is not bounded below by a polynomial, Theorem \ref{thm:graphpoly} implies ${\bf B}_{\calH}$ contains finitely many non-isomorphic almost prime graphs, so (2) holds.\end{proof}

We now prove characterizations for $3$-graphs from fastest to slowest.

\begin{theorem}[Tower]\label{thm:introtower}
Suppose $\calH$ is a hereditary $3$-graph property.  Then the following are equivalent (see Definitions  \ref{def:bhgraphs},  \ref{def:clsvc}, and  \ref{def:psotimes}).
\begin{enumerate}
\item $\Tw(\Omega(\e^{-1}))\leq M_{\calH}(\e)\leq \Tw(O(\e^{-4}))$.
\item $\calH$ is far from finite SVC-dimension.
\item ${\bf B}_{\calH}$ has infinite SVC-dimension.
\item  For all $k\geq 1$, ${\bf B}_{\calH}\cap (1\otimes \PS(k))\neq\emptyset$.
\end{enumerate}
\end{theorem}
\begin{proof}
 Theorem \ref{thm:sl1} yields immediately that (1) and (2) are equivalent.  The equivalence of (2)-(4) follows from Theorem \ref{thm:vdischom}.\end{proof}

\begin{theorem}[Almost Exponential]\label{thm:expintro}
Suppose $\calH$ is a hereditary $3$-graph property.  Then the following are equivalent (see Definitions  \ref{def:bhgraphs}, \ref{def:ukhat}, \ref{def:clvc}, \ref{def:clsvc}, and \ref{def:psotimes}).
\begin{enumerate}
\item $2^{\Omega(\e^{-1/8})}\leq M_{\calH}(\e)\leq 2^{2^{\e^{-C}}}$ for some constant $C>0$.
\item $\calH$ is far from finite VC-dimension and close to finite SVC-dimension.
\item ${\bf B}_{\calH}$ has infinite VC-dimension and finite SVC-dimension. 
\item For all $k\geq 1$, ${\bf B}_{\calH}\cap \widehat{\PS(k)}\neq \emptyset$, but for some $m\geq 1$, ${\bf B}_{\calH}\cap (1\otimes \PS(m))=\emptyset$. 
\end{enumerate}
\end{theorem}
\begin{proof}
 The equivalence of (2) and (3) is from Theorem \ref{thm:vdischom} and Proposition \ref{prop:equiv2}.  That (2) implies (1) is immediate from Theorem  \ref{thm:exp}(2). The equivalence of (3) and (4) are from Proposition \ref{prop:equiv22} and Fact \ref{fact:svcbip2}.
 
 Assume now (1) holds. Since $M_{\calH}$ is not bounded below by a tower function,  Theorem \ref{thm:introtower} implies $\calH$ is close to finite SVC-dimension.   Since $M_{\calH}$ is not bounded above by a polynomial, Theorem \ref{thm:exp}(1) implies $\calH$ cannot be close to finite VC-dimension, i.e. it must be far from finite VC-dimension. Thus (2) holds.\end{proof}

\begin{theorem}[Polynomial]\label{thm:3graphpoly}
Suppose $\calH$ is a hereditary $3$-graph property.  Then the following are equivalent (see Definitions \ref{def:bhgraphs}, \ref{def:ukhat}, \ref{def:clvc}, \ref{def:3graphirred}, and \ref{def:prime3}).
\begin{enumerate}
\item $\Omega(\e^{-1/8})\leq M_{\calH}(\e)\leq \e^{-C}$ for some constant $C>0$.
\item $\calH$ is close to finite VC-dimension and ${\bf B}_{\calH}$ contains infinitely many non-isomorphic almost prime $3$-uniform hypergraphs.
\item ${\bf B}_{\calH}$ has finite VC-dimension and contains infinitely many non-isomorphic almost prime $3$-uniform hypergraphs.
\item There exists $k\geq 1$ such that ${\bf B}_{\calH}\cap \widehat{\PS(k)}=\emptyset$, but for all $m\geq 1$,  ${\bf B}_{\calH}\cap \widehat{\Irr(m)}\neq \emptyset$.
\end{enumerate}
\end{theorem}
\begin{proof}
The equivalence of (2) and (3) follows from Proposition  \ref{prop:equiv2}.  The equivalence of (3) and (4) follows from Propositions \ref{prop:equiv2} and \ref{prop:equiv4}.  That (3) implies (1) follows immediately from  Proposition  \ref{prop:equiv2}  and Theorem \ref{thm:exp}(1) (for the upper bound), and Corollary \ref{cor:vcfar3} (for the lower bound). 

We now show (1) implies (2).  Assume (1) holds. Since $M_{\calH}$ is not bounded below by a tower function, Theorem \ref{thm:sl1} implies $\calH$ is close to finite SVC-dimension. Since $M_{\calH}$ is not bounded below by an exponential, Theorem \ref{thm:exp}(2) then implies $\calH$ must be close to finite VC-dimension.  Since $M_{\calH}$ is not a constant function, Theorem \ref{thm:bhfinite3} implies ${\bf B}_{\calH}$ contains infinitely many non-isomorphic almost prime $3$-graphs. Thus (2) holds.  This completes the verification of the equivalences (1)-(4).\end{proof}

\begin{theorem}[Constant]\label{thm:3graphconstant}
Suppose $\calH$ is a hereditary $3$-graph property.  Then the following are equivalent (see Definitions \ref{def:bhgraphs}, \ref{def:3graphirred}, and \ref{def:prime3}). 
\begin{enumerate}
\item $M_{\calH}(\e)=C$ for some constant $C\in \mathbb{N}^{\geq 1}$.
\item ${\bf B}_{\calH}$ contains finitely many non-isomorphic almost prime $3$-graphs. 
\item $M_{\calH}(\e)=C$ where $C\in \mathbb{N}^{\geq 1}$ satisfies 
$$
C=\max\{\ell \in \mathbb{N}^{\geq 1}: \text{ there is $H\in {\bf B}_{\calH}$ with $\ell$-many $\sim_H$-classes}\}. 
$$
 \item There exists $k\geq 1$ such that ${\bf B}_{\calH}\cap \widehat{\Irr(k)}=\emptyset$.
\end{enumerate}
\end{theorem}
\begin{proof}
That (2) and (4) are equivalent follows from Proposition \ref{prop:equiv4}, so we will ignore (4) for the rest of this proof.  That (2) implies (3) follows from Theorem \ref{thm:bhfinite3}.  That (3) implies (1) is obvious.  We have left to show (1) implies (2).  

Assume (1) holds. Then $M_{\calH}(\e)$ is not bounded below by a polynomial, so by Theorem \ref{thm:3graphpoly}, ${\bf B}_{\calH}$ contains only finitely many non-isomorphic almost prime $3$-graphs, i.e. (2) holds. This completes the proof. \end{proof}

\noindent{\bf Declaration of generative AI and AI-assisted technologies in the manuscript preparation process.}
During the preparation of this work, the author used ChatGPT Pro for proofreading assistance, including identifying typographical and spelling errors, grammatical issues, and mathematical inconsistencies. The author reviewed and edited the output as needed and takes full responsibility for the content of the published article.

\bibliographystyle{amsplain}

\end{document}